\documentclass[a4paper,10pt]{article}

\usepackage[utf8]{inputenc}
\usepackage[english]{babel}
\usepackage{amsthm,amssymb,amsfonts,amsmath}
\usepackage[authoryear]{natbib}
\usepackage{dsfont}
\usepackage{color,fancybox,graphicx}
\usepackage{url}
\usepackage{psfrag}
\usepackage{color}
\usepackage{pifont}
\usepackage{latexsym,pstricks} 

\usepackage{hyperref}
\usepackage{mathrsfs }
\usepackage{pst-tree}
\usepackage{fancyhdr}
\usepackage{etoolbox}
\usepackage{setspace}
\usepackage{enumitem}
\usepackage[ ruled, linesnumbered]{algorithm2e}

\newcommand{\T}{\Theta}

\newcommand{\E}{\mathds{E}}

\renewcommand{\P}{\mathds{P}}

\newcommand{\N}{\mathbb{N}}

\newcommand{\bX}{{\bf X}}
\newcommand{\bx}{{\bf x}}
\newcommand{\bz}{{\bf z}}

\newcommand{\mtry}{m_{\textrm{try}}}

\newtheorem{theorem}{Theorem}[section]

\newtheorem{lemme}{Lemma}

\newtheorem{corollary}{Corollary}

\newtheoremstyle{break}  
  {\topsep}   
  {\topsep}   
  {\itshape}  
  {0pt}       
  {\bfseries} 
  {}         
  {5pt plus 1pt minus 1pt}  
  {}          
\theoremstyle{break}
\newtheorem*{assumption}{}

\usepackage{natbib}
\bibpunct{(}{)}{;}{a}{,}{,}

\usepackage{hyperref}
\hypersetup{
  colorlinks = true,
  urlcolor = blue, 
  linkcolor = blue,
  citecolor = blue,
}

\begin{document}

\begin{center}
{\Large 
\textbf{\textsf{Impact of subsampling and pruning on random forests.}}}
\medskip
\medskip
\end{center}

\noindent{\bf Roxane Duroux}\\
{\it Sorbonne Universit\'es, UPMC Univ Paris 06, F-75005, Paris, France}\\
\href{mailto:roxane.duroux@upmc.fr}{roxane.duroux@upmc.fr}\

\noindent{\bf Erwan Scornet }\\
{\it Sorbonne Universit\'es, UPMC Univ Paris 06, F-75005, Paris, France}\\
\href{mailto:erwan.scornet@upmc.fr}{erwan.scornet@upmc.fr}\

\medskip
\begin{abstract}
\noindent {\rm Random forests are ensemble learning methods introduced by \citet{Br01} that operate by averaging several decision trees built on a randomly selected subspace of the data set. Despite their widespread use in practice, the respective roles of the different mechanisms at work in Breiman's forests are not yet fully understood, neither is the tuning of the corresponding parameters. In this paper, we study the influence of two parameters, namely the subsampling rate and the tree depth, on Breiman's forests performance. More precisely, we show that fully developed subsampled forests and pruned (without subsampling) forests  have similar performances, as long as respective parameters are well chosen. Moreover, experiments show that a proper tuning of subsampling or pruning lead in most cases to an improvement of Breiman's original forests errors.

\medskip


%
}
\end{abstract}

\section{Introduction}

Random forests are a class of learning algorithms used to solve pattern recognition problems. As ensemble methods, they grow many base learners (\textit{i.e.}, decision trees) and aggregate them to predict. Building several different trees from a single data set requires to randomize the tree building process by, for example, sampling the data set. Thus, there exists a large variety of random forests, depending on how trees are designed and how the randomization is introduced in the whole procedure. 

One of the most popular random forests is that of \citet{Br01} which grows trees based on CART procedure \citep[Classification And Regression Trees,][]{BrFrOlSt84} and randomizes both the training set and the splitting directions. Breiman's \citeyearpar{Br01} random forests have been under active investigation during the last decade mainly because of their good practical performance and their ability to handle high-dimensional data sets. They are acknowledged to be state-of-the-art methods in fields such as genomics \citep[][]{Qi12} and pattern recognition \citep[][]{RoRiRaOrTo08}, just to name a few.

The ease of the implementation of random forests algorithms is one of their key strengths and has greatly contributed to their widespread use.
A proper tuning of the different parameters of the algorithm is not mandatory to obtain a plausible prediction, making random forests a turn-key solution to deal with large, heterogeneous data sets.

Several authors studied the influence of the parameters on random forests accuracy. For example, the number $M$ of trees in the forests has been  thoroughly investigated by \citet{DiAl06} and \citet{genuer}. It is easy to see that the computational cost for inducing a forest increases linearly with $M$, so  a good choice for $M$ results from a trade-off between computational complexity and  accuracy ($M$ must be large enough for  predictions to be stable). \citet{DiAl06} argued that in micro-array classification problems, the particular value of $M$ is irrelevant, assuming that $M$ is large enough (typically over $500$). Several recent studies provided theoretical guarantees for choosing $M$. \citet{Sc15a} proposed a way to tune $M$ so that the error of the forest is minimal. \citet{MeHo14a} and \citet{Wa14} gave a more in-depth analysis by establishing a central limit theorem for random forests prediction and providing a method to estimate their variance. All in all, the role of $M$ on the forest prediction is broadly understood.

Besides the number of trees, forests depend on three parameters: the number $a_n$ of data points selected to build each tree, the number $\mtry$ of preselected variables along which the best split is chosen, and the minimum number \texttt{nodesize} of data points in each cell of each tree. 
The effect of $\mtry$ was thoroughly investigated in \citet{DiAl06} and  \citet{genuer} who claimed that the default value is either optimal or too small, therefore leading to no global understanding of this parameter. The story is roughly the same regarding the parameter \texttt{nodesize} for which the default value has been reported as a good choice by \citet{DiAl06}.
Furthermore, there is no theoretical guarantee to support the default values of parameters or any of the data-driven tuning proposed in the literature. 

Our objective in this paper is two-folds: $(i)$ to provide a theoretical  framework to analyse the influence of the number of points $a_n$ used to build each tree, and the tree depth (corresponding to parameters \texttt{nodesize} or \texttt{maxnode}) on the random forest performances; $(ii)$ to implement several experiments to test our theoretical findings. The paper is organized as follows. Section \ref{S-notations} is devoted to notations and presents Breiman's random forests algorithm. To carry out a theoretical analysis of the subsample size and the tree depth, we study in Section \ref{S-theorie} a particular random forest called median forest. We establish an upper bound for the risk of median forests and by doing so, we highlight the fact that subsampling and pruning have similar influence on median forest predictions. The numerous experiments on Breiman forests are presented in Section \ref{S-simu}. Proofs are postponed to Section \ref{S-preuves}.

\section{First definitions}\label{S-notations}

\subsection{General framework}

In this paper, we consider a training sample $\mathcal D_n=\{(\bX_1,Y_1),$ $ \hdots, (\bX_n,Y_n)\}$ of $[0,1]^d\times$ $  \mathbb R$-valued independent and identically distributed observations of a random pair $(\bX,$ $ Y)$, where $\mathds{E}[Y^2]<\infty$. The variable $\bX$ denotes the predictor variable and $Y$ the response variable. We wish to estimate the  regression function $m(\bx) = \E \left[ Y | \bX = \bx\right]$. In this context, we use random forests to build an estimate $m_n: [0,1]^d \to \mathds{R}$ of $m$, based on the data set $\mathcal{D}_n$.

Random forests are classification and regression methods based on a  collection of $M$ randomized trees. We denote by $m_n(\bx, \Theta_j,\mathcal D_n)$ the predicted value at point $\bx$ given by the $j$-th tree, where $\Theta_1, \hdots,\Theta_M$ are independent random variables, distributed as a generic random variable $\Theta$, independent of the sample $\mathcal D_n$. In practice, the variable $\Theta$ can be used to sample the data set or to select the candidate directions or positions for splitting. The predictions of the $M$ randomized trees are then averaged to obtain the random forest prediction
\begin{align}
m_{M,n}({\bf x}, \Theta_1, \hdots, \Theta_M, \mathcal{D}_n) = \frac{1}{M} \sum_{m=1}^M m_n(\bx, \Theta_m, \mathcal{D}_n). \label{finite_forest}
\end{align}
By the law of large numbers, for any fixed $\bx$, conditionally on $\mathcal{D}_n$, the finite forest estimate tends to the infinite forest estimate 
\begin{align*}
m_{\infty,n}(\bx, \mathcal{D}_n) = \E_{\Theta} \left[m_n(\bx, \Theta)\right].
\end{align*}
For the sake of simplicity, we denote $m_{\infty,n}(\bx, \mathcal{D}_n)$ by $m_{\infty,n}(\bx)$. Since we carry out our analysis within the $\mathbb{L}^2$ regression estimation framework, we say that $m_{\infty,n}$ is $\mathds{L}^2$ consistent if its risk, $ \mathds E [m_{\infty,n}(\bX)-m(\bX)]^2$, tends to zero, as $n$ goes to infinity.

\subsection{Breiman's forests}

Breiman's \citeyearpar{Br01} forest is one of the most used random forest algorithms. 
In Breiman's forests, each node of a single tree is associated with a hyper-rectangular cell included in $[0,1]^d$. The root of the tree is $[0,1]^d$ itself and, at each step of the  tree construction, a node (or equivalently its corresponding cell) is split in two parts. The terminal nodes (or leaves), taken together, form a partition of $[0,1]^d$. In details, the algorithm works as follows: 
\begin{enumerate}
\item Grow $M$ trees as follows:
\begin{enumerate}
\item Prior to the $j$-th tree construction, select uniformly with replacement, $a_n$ data points among $\mathcal{D}_n$. Only these $a_n$ observations are used in the tree construction. 
\item Consider the cell $[0,1]^d$. 
\item Select uniformly without replacement $\mtry$ coordinates among $\{1,$ $\hdots, d\}$. 
\item Select the split minimizing the CART-split criterion \citep[see][for details]{BrFrOlSt84} along the pre-selected $\mtry$ directions. 
\item Cut the cell at the selected split. 
\item Repeat $(c)-(e)$ for the two resulting cells until each cell of the tree contains less than \texttt{nodesize} observations. 
\item For a query point $\bx$, the $j$-th tree outputs the average of the $Y_i$ falling into the same cell as $\bx$.
\end{enumerate}
\item For a query point $\bx$, Breiman's forest outputs the average of the predictions given by the $M$ trees. 
\end{enumerate}

The whole procedure depends on four parameters: the number $M$ of trees,  the number $a_n$ of sampled data points in each tree, the number $\mtry$ of pre-selected directions for splitting, and the maximum number \texttt{nodesize} of observations in each leaf. 
By default in the \texttt{R} package \texttt{randomForest}, $M$ is set to $500$, $a_n=n$ (bootstrap samples are used to build each tree), $\mtry = d/3$ and \texttt{nodesize}$= 5$.

Note that selecting the split that minimizes the CART-split criterion is equivalent to select the split such that the two resulting cells have a minimal (empirical) variance (regarding the $Y_i$ falling into each of the two cells).

%
%

%
%

\section{Theoretical results}\label{S-theorie}

The numerous mechanisms at work in Breiman's forests, such as the subsampling step, the CART-criterion and the trees aggregation, make the whole procedure difficult to theoretically analyse. Most attempts to understand the random forest algorithms \citep[see e.g.,][]{BiDeLu08,IsKo10, DeMaFr13} have focused on simplified procedures, ignoring the subsampling step and/or replacing the CART-split  criterion by a data independent procedure more amenable to analysis. On the other hand, recent studies try to dissect the original Breiman's algorithm in order to prove its asymptotic normality \citep[][]{MeHo14a,Wa14}  or its consistency \citep[][]{ScBiVe15}. When studying the original algorithm, one faces the complexity of the algorithm thus requiring high-level mathematics to prove insightful---but rough--- results.

In order to provide theoretical guarantees on the parameters default values in random forests, we focus in this section on a simplified random forest called median forest \citep[see, for example,][for details on median tree]{BiDe14}.


\subsection{Median Forests}

%

To take one step further into the understanding of Breiman's \citeyearpar{Br01} forest behavior, we study the median random forest, which satisfies the $X$-property \citep{DeGyLu96}. Indeed, its construction depends only on the $X_i$'s which is a good trade off between the complexity of Breiman's \citeyearpar{Br01} forests and the simplicity of totally non adaptive forests, whose construction is independent of the data set. 
Besides, median forests can be tuned such that each leaf of each tree contains exactly one point. In this way, they are closer to Breiman's forests than totally non adaptive forests (whose cells cannot contain a pre-specified number of points) and thus provide a good understanding on  Breiman's forests performance even when there is exactly one data point in each leaf.

We now describe the construction of median forest. In the spirit of Breiman's \citeyearpar{Br01} algorithm, before growing each tree, data are subsampled, that is $a_n$ points ($a_n<n$) are selected, without replacement. Then, each split is performed on an empirical median along a coordinate, chosen uniformly at random among the $d$ coordinates. Recall that the median of $n$ real valued random variables $X_1, \hdots, X_n$ is defined as the only $X_{(\ell)}$ satisfying $F_n(X_{(\ell - 1)}) \leq 1/2 < F_n(X_{(\ell)})$, where the $X_{(i)}$'s are ordered increasingly and $F_n$ is the empirical distribution function of $X$. 
Note that data points on which splits are performed are not sent down to the resulting cells. This is done to ensure that data points are uniformly distributed on the resulting cells (otherwise, there would be at least one data point on the edge of a resulting cell, and thus the data points distribution would not be uniform on this cell). Finally, the algorithm stops when each cell has been cut exactly $k_n$ times, \textit{i.e.}, \texttt{nodesize}$= a_n 2^{-k_n}$. The parameter $k_n$, also known as the level of the tree, is assumed to verify $a_n 2^{-k_n} \geq 4$. The overall construction process is recalled below. 

\begin{enumerate}
\item Grow $M$ trees as follows:
\begin{enumerate}
\item Prior to the $j$-th tree construction, select uniformly without replacement, $a_n$ data points among $\mathcal{D}_n$. Only these $a_n$ observations are used in the tree construction. 
\item Consider the cell $[0,1]^d$. 
\item Select uniformly one coordinate among $\{1,$ $\hdots, d\}$ without replacement.  
\item Cut the cell at the empirical median of the $X_i$ falling into the cell along the preselected direction.  
\item Repeat $(c)-(d)$ for the two resulting cells until each cell has been cut exactly $k_n$ times. 
\item For a query point $\bx$, the $j$-th tree outputs the average of the $Y_i$ falling into the same cell as $\bx$.
\end{enumerate}
\item For a query point $\bx$, median forest outputs the average of the predictions given by the $M$ trees. 
\end{enumerate}

\subsection{Main theorem}

Let us first make some regularity assumptions on the regression model. 

\begin{assumption}\label{assumption_last_th} {\bf (H)}
One has 
\begin{align*}
Y = m(\bX) + \varepsilon,
\end{align*}
where $\varepsilon$ is a centred noise such that $\mathds{V}[\varepsilon|\bX = \bx]\leq \sigma^2$, where $\sigma^2<\infty$ is a constant. Moreover, $\bX$ is uniformly distributed on  $[0,1]^d$ and $m$ is $L$-Lipschitz continuous.
\end{assumption}

Theorem \ref{th_general_bound} presents an upper bound for the $\mathbb{L}^2$ risk of $m_{\infty,n}$.

\begin{theorem}
\label{th_general_bound}
Assume that {\bf (H)} is satisfied. Then, for all $n$, for all $\bx \in [0,1]^d$, 
\begin{align}
\E \big[ m_{\infty,n}(\bx) - m(\bx) \big]^2  \leq  2\sigma^2 \frac{2^k}{n} + dL^2C_1 \bigg(1- \frac{3}{4d} \bigg)^k. \label{general_formula_error}
\end{align}
In addition, let $\beta = 1 - 3/4d$. The right-hand side is minimal for 
\begin{align}
k_n = \frac{1}{\ln 2 - \ln \beta} \bigg[ \ln (n) +  C_3 \bigg], \label{kn_inequality}
\end{align}
under the condition that $a_n \geq C_4 n^{\frac{\ln 2}{\ln 2 - \ln \beta}}$. For these choices of $k_n$ and $a_n$, we have
\begin{align}
\E \big[ m_{\infty,n}(\bx) - m(\bx) \big]^2  \leq C n^{\frac{\ln \beta}{\ln 2 - \ln \beta}}. \label{general_bound_median_forest}
\end{align}
\end{theorem}


Equation (\ref{general_formula_error}) stems from a standard decomposition of the estimation/appro-ximation error of median forests. Indeed, the first term in equation (\ref{general_formula_error}) corresponds to the estimation error of the forest as in \citet{Bi12} or \citet{ArGe14} whereas the second term is the approximation error of the forest, which decreases exponentially in $k$. 
Note that this decomposition is consistent with the existing literature on random forests. Two common assumptions to prove consistency of simplified random forests are $n/2^k \to \infty$ and $k \to \infty$, which respectively controls the estimation and approximation of the forest. According to Theorem \ref{th_general_bound}, making these assumptions for median forests  results in proving their consistency.

Note that the estimation error of a single tree grown with $a_n$ observations is of order $2^{k}/a_n$. Thus, because of the  subsampling step (\textit{i.e.}, since $a_n <n$), the estimation error of median forests $2^{k}/n$ is smaller than that of a single tree. The variance reduction of random forests is a well-known property, already noticed by \citet{Ge12} for a totally non adaptive forest, and by \citet{Sc15a} in the case of median forests. In our case, we exhibit an explicit bound on the forest variance, which allows us to precisely compare it to the individual tree variance therefore highlighting a first benefit of forests over singular trees.

As for the first term in inequality (\ref{general_formula_error}), the second term could be expected. Indeed, in the levels close to the root, a split is very close to the center of a side of a cell (since $\bX$ is uniformly distributed over $[0,1]^d$). Thus, for all $k$ small enough, the approximation error of median forests should be close to that of centred forests studied by \citet{Bi12}. Surprisingly, the rate of consistency of median forests is faster than that of centred forest established  in \citet{Bi12}, which is equal to 
\begin{align}
\E \big[ m_{\infty,n}^{cc}(\bX) - m(\bX) \big]^2 \leq C n^{\frac{ - 3}{4 d \ln 2 +  3}}, \label{bound_centred_forest}
\end{align}
where $m_{\infty,n}^{cc}$ stands for the centred forest estimate.
A close inspection of the proof of Proposition $2.2$ in \citet{Bi12} shows that it can be easily adapted to match the (more optimal) upper bound in Theorem \ref{th_general_bound}.

Noteworthy, the fact that the upper bound (\ref{general_bound_median_forest}) is sharper than (\ref{bound_centred_forest}) appears to be important in the case where $d=1$. In that case, according to Theorem \ref{th_general_bound}, for all $n$, for all $\bx \in [0,1]^d$, 
\begin{align*}
\E \big[ m_{\infty,n}(\bx) - m(\bx) \big]^2  \leq C n^{-2/3}, 
\end{align*}
which is the minimax rate over the class of Lipschitz functions \citep[see, e.g.,][]{St80,St82}. This was to be expected since, in dimension one, median random forests are simply a median tree which is known to reach minimax rate \citep[][]{DeGyLu96}. Unfortunately, for $d=1$, the centred forest bound (\ref{bound_centred_forest}) turns out to be suboptimal since it results in 
\begin{align}
\E \big[ m_{\infty,n}^{cc}(\bX) - m(\bX) \big]^2 \leq C n^{\frac{- 3}{4 \ln 2 +  3}}.
\end{align}

Theorem \ref{th_general_bound} allows us to derive rates of consistency for two particular forests: the pruned median forest, where no subsampling is performed prior to build each tree, and the fully developed median forest, where each leaf contains a small number of points. Corollary \ref{Corollary_1} deals with the pruned forests. 

\begin{corollary}[Pruned median forests]
\label{Corollary_1}
Let  $\beta = 1 - 3/4d$.
Assume that {\bf (H)} is satisfied. Consider a median forest without subsampling (i.e., $a_n=n$) and such that the parameter $k_n$ satisfies (\ref{kn_inequality}). Then, for all $n$, for all $\bx \in [0,1]^d$, 
\begin{align*}
\E \big[ m_{\infty,n}(\bx) - m(\bx) \big]^2  \leq C n^{\frac{\ln \beta}{\ln 2 - \ln \beta}}.
\end{align*}
\end{corollary}

Up to an approximation, Corollary \ref{Corollary_1} is the counterpart of Theorem $2.2$ in \citet{Bi12} but tailored for median forests. Indeed, up to a small modification of the proof of Theorem $2.2$, the rate of consistency provided in Theorem $2.2$ for centred forests and that of Corollary \ref{Corollary_1} for median forests are identical. Note that, for both forests, the optimal depth $k_n$ of each tree is the same.

Corollary \ref{Corollary_2} handles the case of fully grown median forests, that is forests which contain a small number of points in each leaf. Indeed, note that since $k_n = \log_2(a_n) - 2$, the number of observations in each leaf varies between $4$ and $8$. 

\begin{corollary}[Fully grown median forest]
\label{Corollary_2}
Let  $\beta = 1 - 3/4d$.
Assume that {\bf (H)} is satisfied. Consider a fully grown median forest whose parameters $k_n$ and $a_n$ satisfy $k_n = \log_2(a_n) - 2$. The optimal choice for $a_n$ (that minimizes the $\mathbb{L}^2$ error in (\ref{general_formula_error})) is then given by (\ref{kn_inequality}), that is
\begin{align*}
a_n = C_4 n^{\frac{\ln 2}{\ln 2 - \ln \beta}}.
\end{align*}
In that case, for all $n$, for all $\bx \in [0,1]^d$, 
\begin{align*}
\E \big[ m_{\infty,n}(\bx) - m(\bx) \big]^2  \leq C n^{\frac{\ln \beta}{\ln 2 - \ln \beta}}.
\end{align*}
\end{corollary}

Whereas each individual tree in the fully developed median forest is inconsistent (since each leaf contains a small number of points), the whole forest is consistent and its rate of consistency is provided by Corollary \ref{Corollary_2}.
Besides, Corollary \ref{Corollary_2} provides us with the optimal subsampling size for fully developed median forests.

Provided a proper parameter tuning, pruned median forests without subsampling and fully grown median forests (with subsampling) have similar performance. A close look at Theorem \ref{th_general_bound} shows that the subsampling size has no effect on the performance, provided it is large enough. The parameter of real importance is the tree depth $k_n$. Thus, fixing $k_n$ as in equation (\ref{kn_inequality}), and by varying the subsampling rate $a_n/n$ one can obtain random forests that are more-or-less pruned, all satisfying the optimal bound in Theorem \ref{th_general_bound}. In that way, Corollary \ref{Corollary_1} and \ref{Corollary_2} are simply two particular examples of such forests. 

Although our analysis sheds some light on the role of subsampling and tree depth, the statistical performances of median forests does not allow us to choose between pruned and subsampled forests. Interestingly, note that these two types of random forests can be used in two different contexts. If one wants to obtain fast predictions, then subsampled forests, as described in Corollary \ref{Corollary_2}, are to be preferred since their computational time is lower than pruned random forests (described in Corollary \ref{Corollary_1}). However, if one wants to build more accurate predictions, pruned random forests have to be chosen since the recursive random forest procedure allows to build several forests of different tree depths in one run, therefore allowing to select the best model among these forests. 
%
%
%
%
%





\section{Experiments}\label{S-simu}

In the light of Section \ref{S-theorie}, we carry out some simulations to investigate $(i)$ how pruned and subsampled forests compare with Breiman's forests and $(ii)$ the influence of subsampling size and tree depth on Breiman's procedure. 
To do so, we start by defining various  regression models on which the several experiments are based. 
Throughout this section, we assess the forest performances by computing their empirical $\mathbb{L}^2$ error.

\begin{itemize}[label = ,leftmargin=0cm]

\item{\bf Model 1}: $n = 800, d = 50, Y = \tilde{X}_1^2 + \exp(-\tilde{X}_2^2)$

\item {\bf Model 2}: $n = 600, d = 100, Y = \tilde{X}_1 \tilde{X}_2 + \tilde{X}_3^2 - \tilde{X}_4 \tilde{X}_7 + \tilde{X}_8 \tilde{X}_{10} - \tilde{X}_6^2 + \mathcal{N}(0,0.5)$

\item {\bf Model 3}: $n = 600, d = 100, Y = -\sin(2 \tilde{X}_1) + \tilde{X}_2^2 + \tilde{X}_3 - \exp(-\tilde{X}_4) + \mathcal{N}(0,0.5)$

\item {\bf Model 4}: $n = 600, d = 100, Y = \tilde{X}_1 + (2 \tilde{X}_2-1)^2 + \sin(2 \pi \tilde{X}_3) / (2-\sin(2 \pi \tilde{X}_3)) + \sin(2 \pi \tilde{X}_4) + 2 \cos(2\pi \tilde{X}_4) + 3 \sin^2(2\pi \tilde{X}_4 )+ 4 \cos^2(2\pi \tilde{X}_4) + \mathcal{N}(0,0.5)$

\item {\bf Model 5}: $n = 700, d = 20, Y = \mathds{1}_{\tilde{X}_1 > 0 } + \tilde{X}_2^3 +  \mathds{1}_{\tilde{X}_4 + \tilde{X}_6 - \tilde{X}_8  - \tilde{X}_9 > 1 + \tilde{X}_{10} } + \exp(-\tilde{X}_2^2) + \mathcal{N}(0,0.5)$

\item {\bf Model 6}: $n = 500, d = 30, Y = \sum_{k=1}^{10} \mathds{1}_{\tilde{X}_k^3 < 0 } - \mathds{1}_{\mathcal{N}(0,1)> 1.25 }$

\item {\bf Model 7}: $n = 600, d = 300, Y = \tilde{X}_1^2 + \tilde{X}_2^2 \tilde{X}_3  \exp(-|\tilde{X}_4|) + \tilde{X}_6 - \tilde{X}_8 + \mathcal{N}(0,0.5)$


\item {\bf Model 8}: $n= 500,d= 1000,Y=\tilde{X}_1+3\tilde{X}_3^2 - 2 \exp(-\tilde{X}_5) + \tilde{X}_6$


\end{itemize}

For all regression frameworks, we consider covariates $\bX = (X_1, \hdots, X_d)$ that are uniformly distributed over $[0,1]^d$. We also let $\tilde{X}_i = 2(X_i-0.5)$ for $1 \leq i \leq d$. Some of these models are toy models ({\bf Model 1, 5-8}). {\bf Model 2} can be found in \citet{VaPoHu07} and  {\bf Models 3-4} are presented in \citet{MeVaBu09}. All numerical implementations have been performed using the free R software. For each experiment, the data set is divided into a training set ($80\%$ of the data set) and a test set (the remaining $20\%$). Then, the empirical risk ($\mathds{L}^2$ error) is evaluated on the test set.

\subsection{Pruning}

We start by studying Breiman's original forests and pruned Breiman's forests. Breiman's forests are the standard procedure implemented in the \texttt{R} package \texttt{randomForest}, with the parameters default values, as described in Section \ref{S-notations}. Pruned Breiman's forests are similar to Breiman's forests except that the tree depth is controlled via the parameter \texttt{maxnodes} (which corresponds to the number of leaves in each tree) and that the whole sample $\mathcal{D}_n$ is used to build each tree.

In Figure \ref{fig1}, we present, for the {\bf Models 1-8} introduced previously, the evolution of the empirical risk of pruned forests for different numbers of terminal nodes. We add the representation of the empirical risk of Breiman's original forest in order to compare all forests errors at a glance. Every sub-figure of Figure \ref{fig1} presents forests built with $500$ trees. The printed errors are obtained by averaging the risks of $50$ forests. Because of the estimation/approximation compromise, we expect the empirical risk of pruned forests to be decreasing and then increasing, as the number of leaves grows. In most of the models, it seems that the estimation error is too low to be detected, this is why several risks in Figure \ref{fig1} are only decreasing.

\begin{figure}[h!!]
\begin{tabular}{cc}
\includegraphics[scale=0.3]{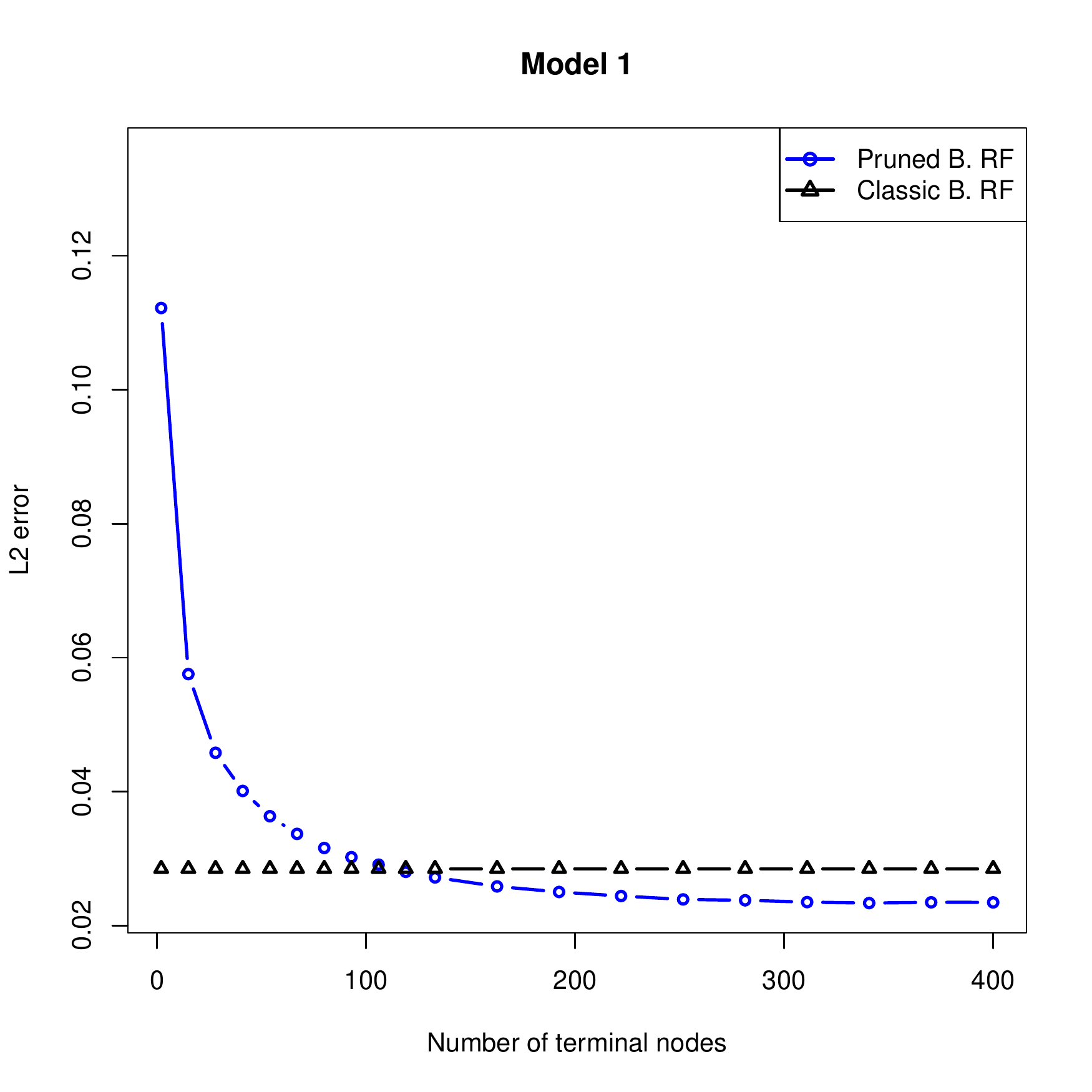}
& \includegraphics[scale=0.3]{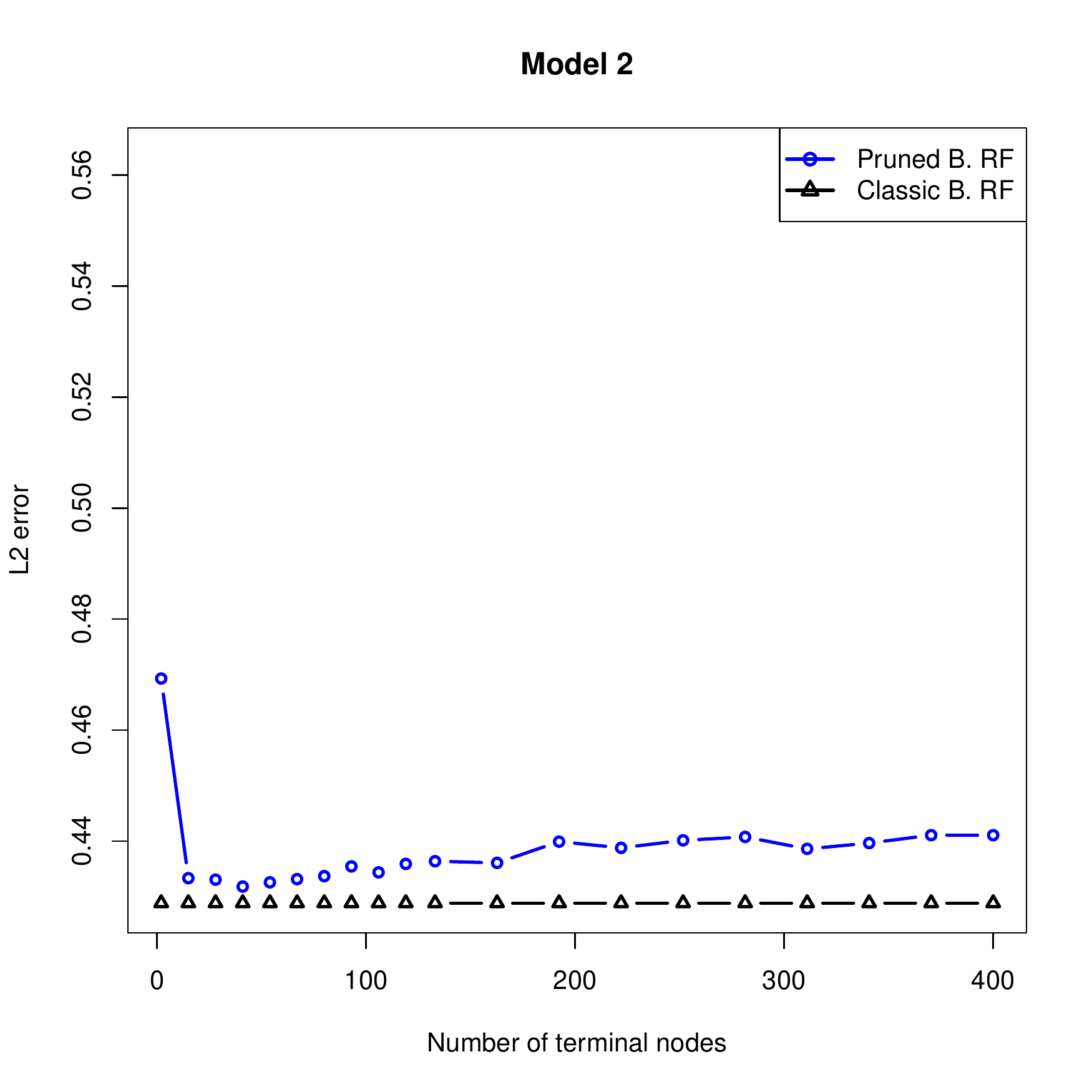}
\\
\includegraphics[scale=0.3]{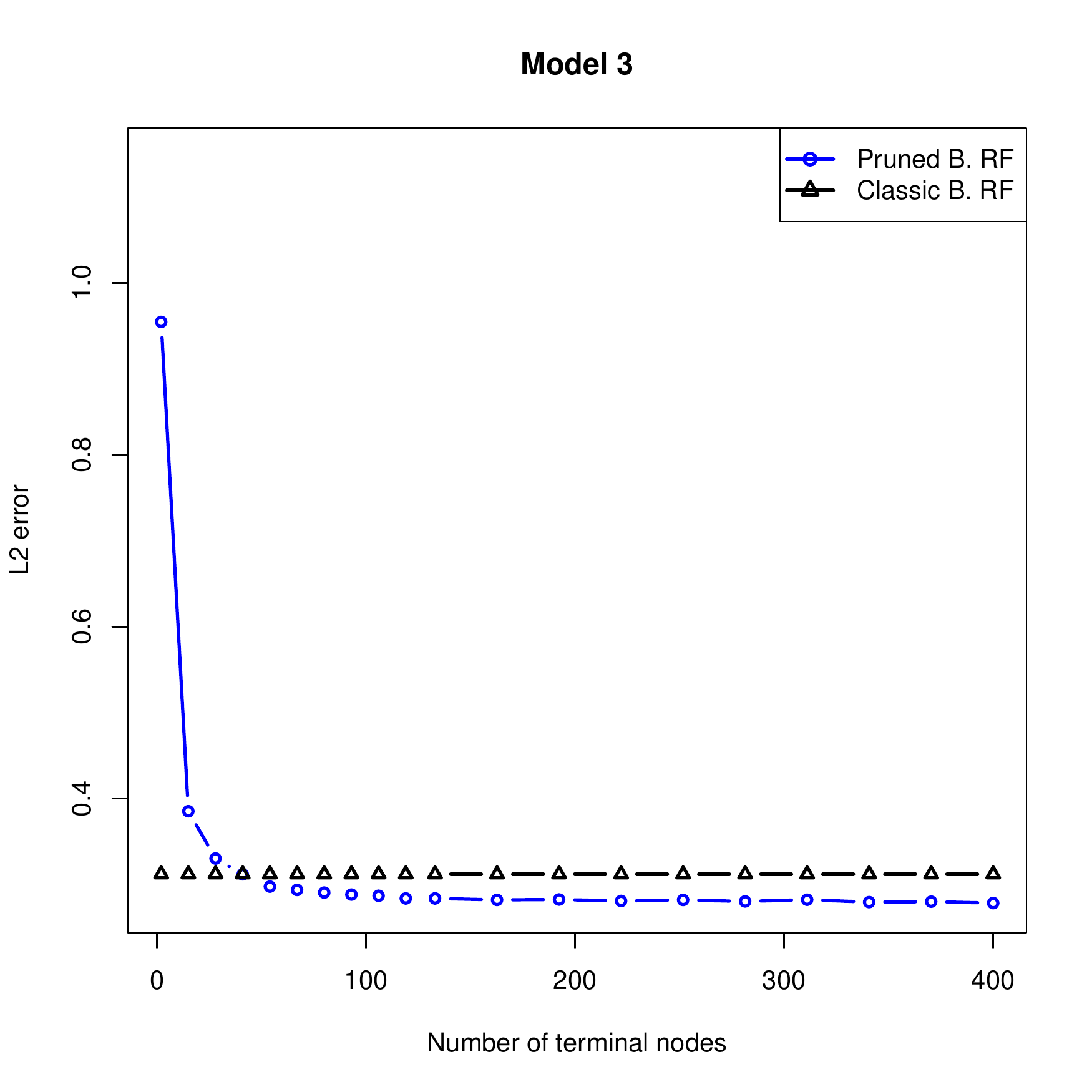}
& \includegraphics[scale=0.3]{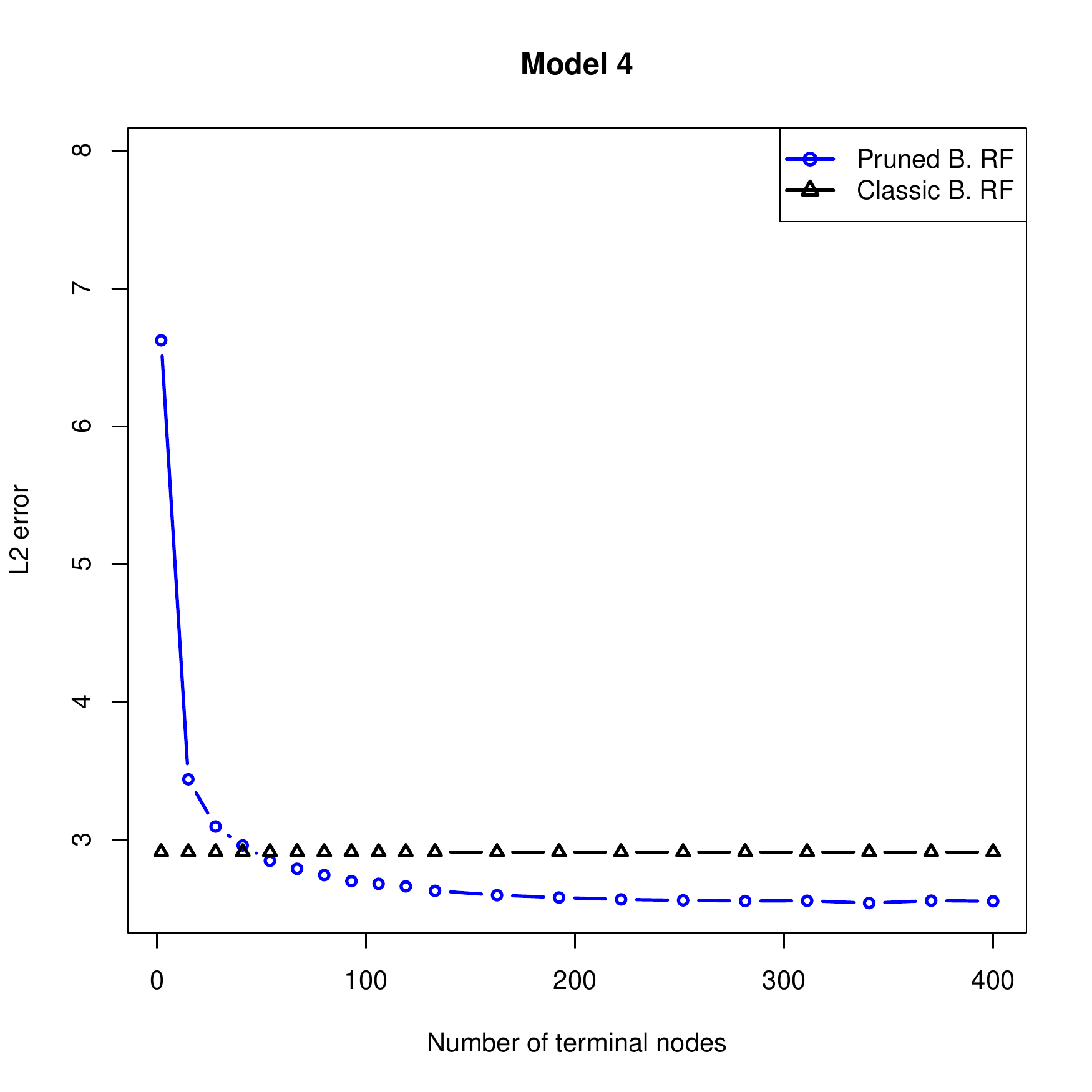}
\\
\includegraphics[scale=0.3]{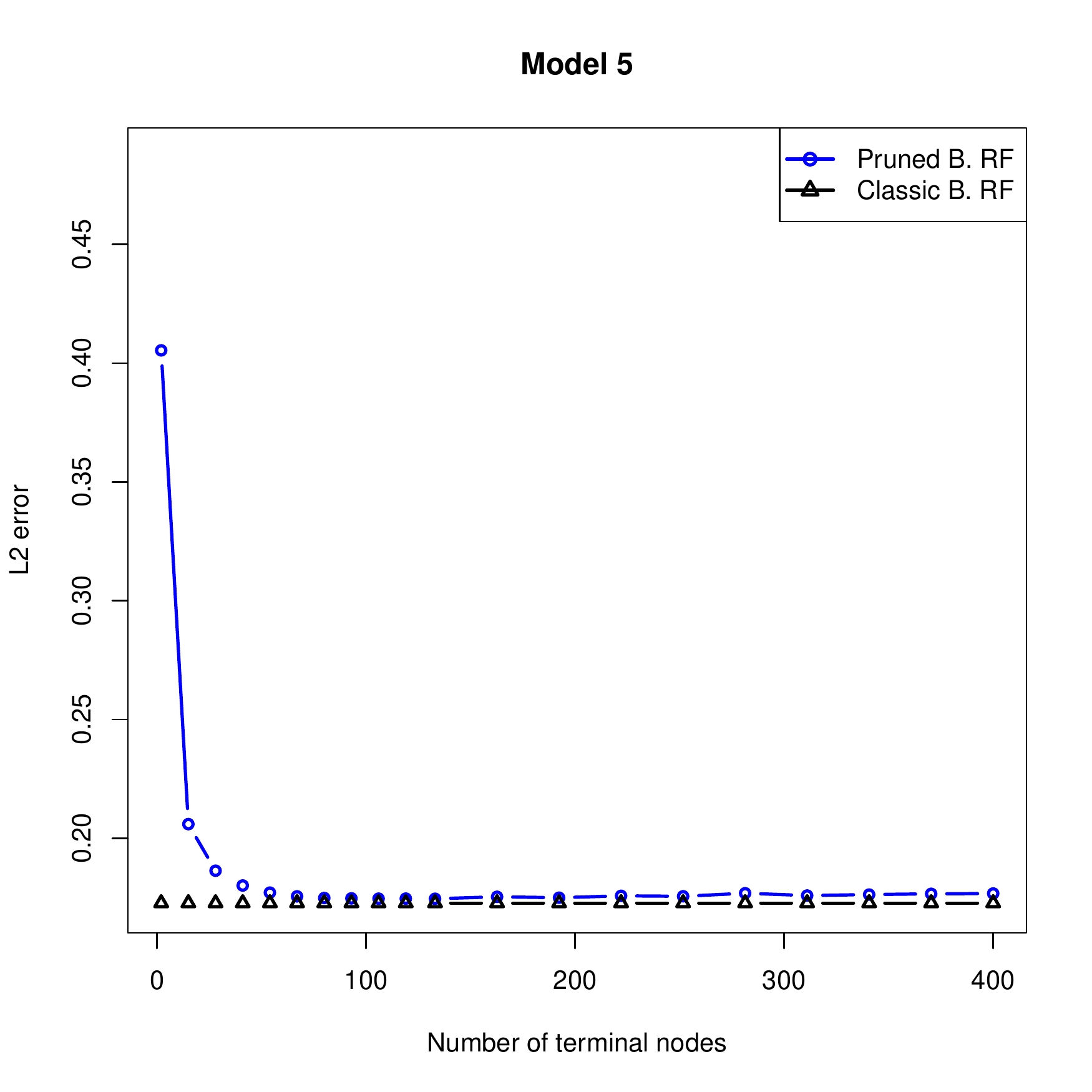}
& \includegraphics[scale=0.3]{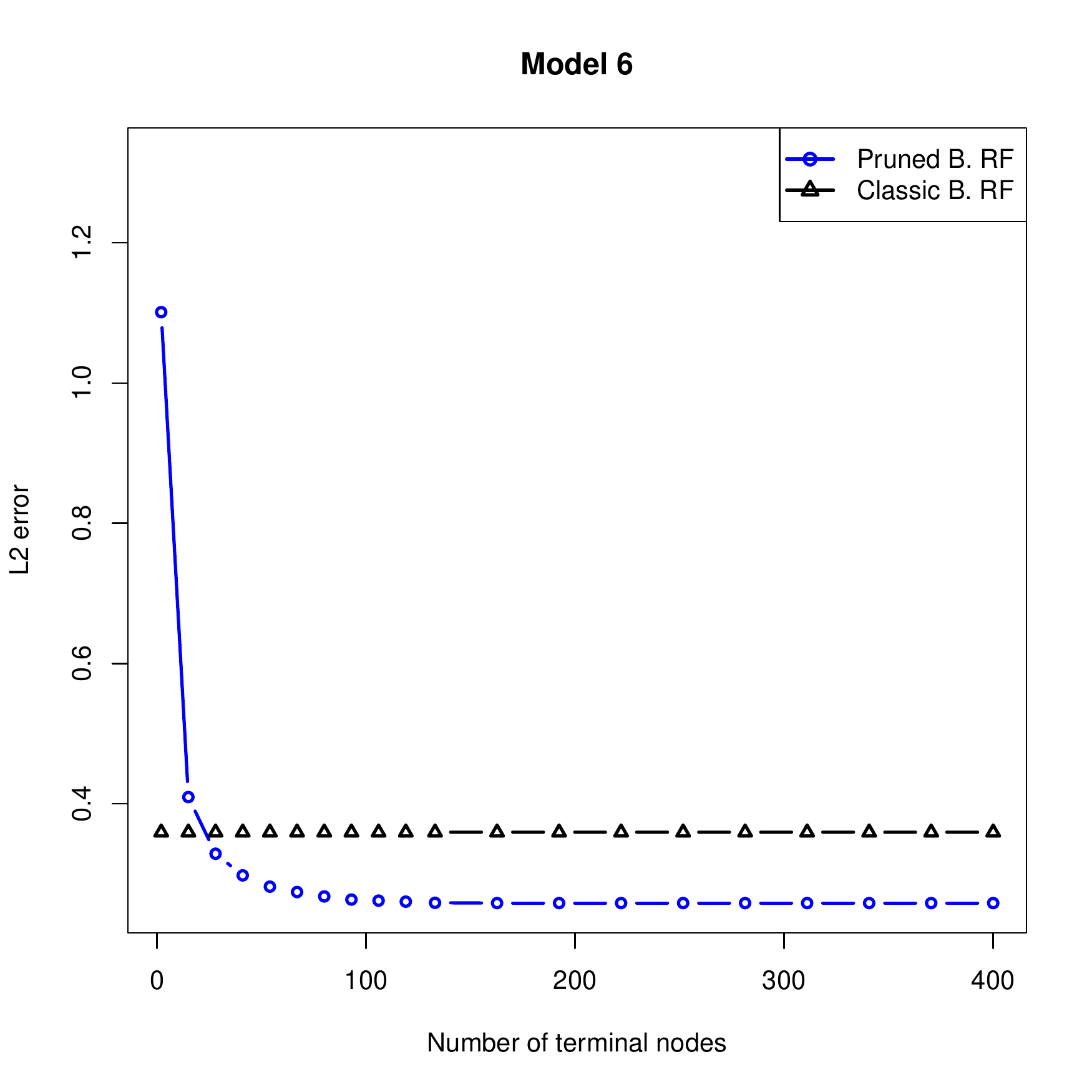}
\\
\includegraphics[scale=0.3]{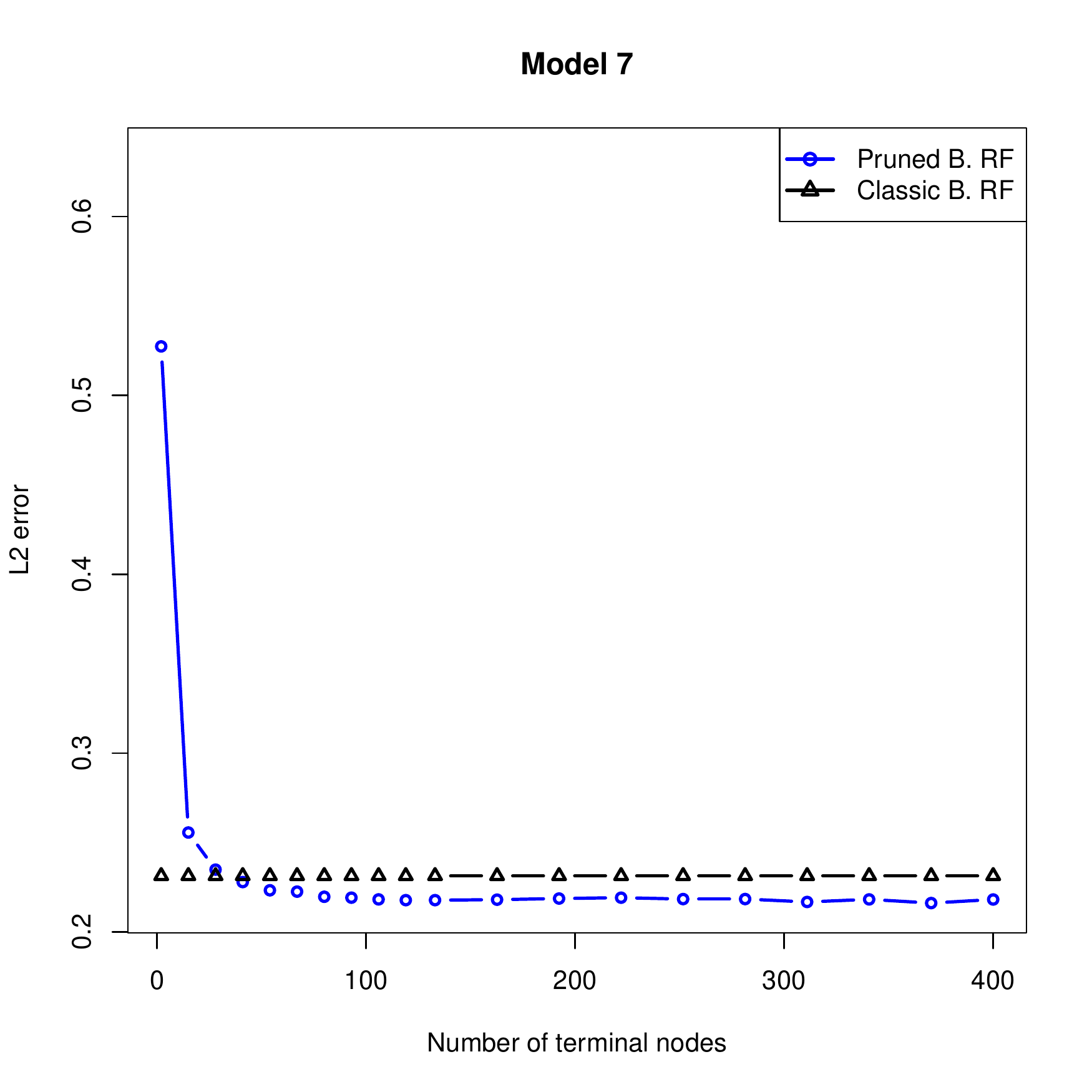}
& \includegraphics[scale=0.3]{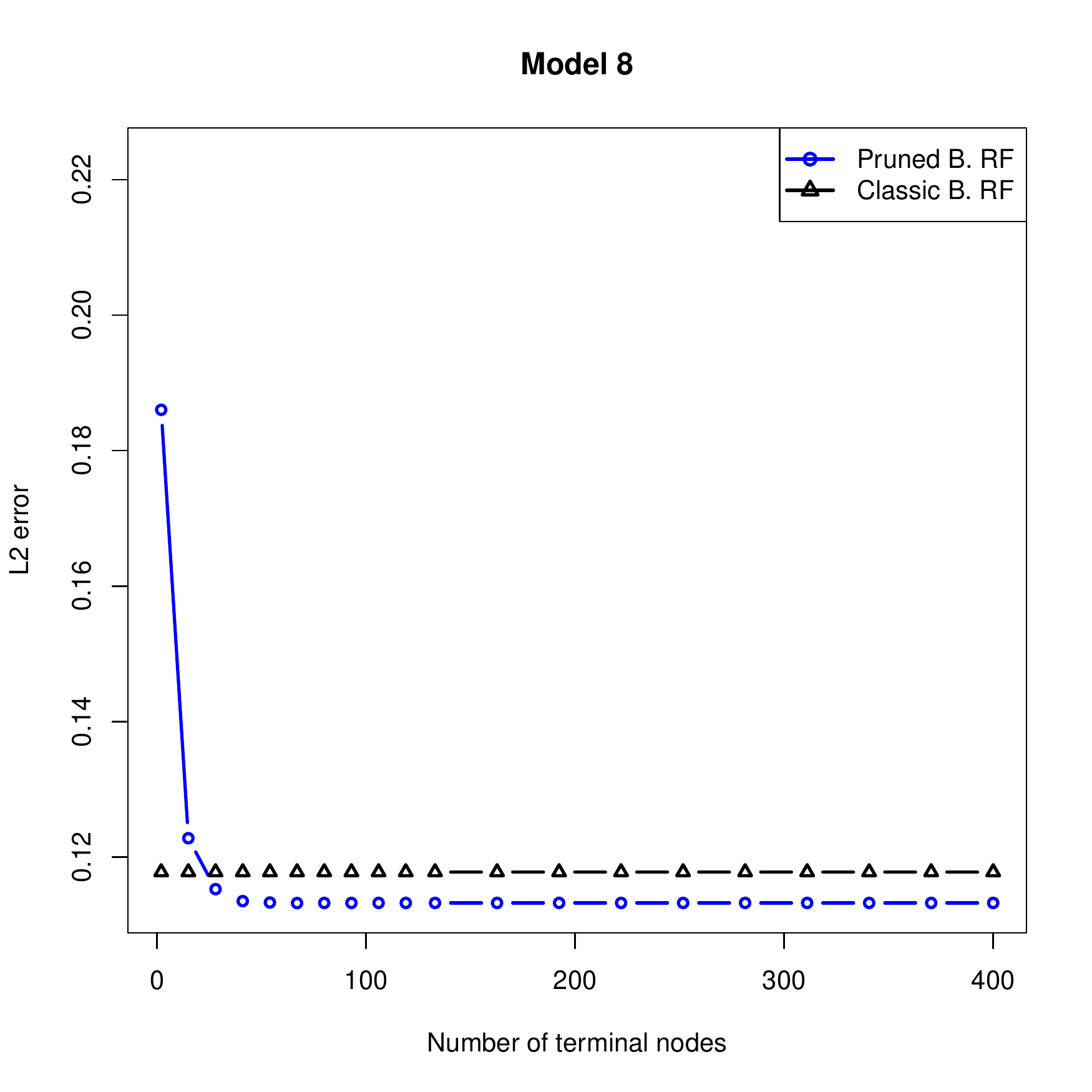}
\\
\end{tabular}
\caption{Comparison of standard Breiman's forests (B. RF) against pruned Breiman's forests in terms of $\mathbb{L}^2$ error.}
\label{fig1}
\end{figure}

For every model, we can notice that pruned forests performance is comparable with the one of standard Breiman's forest, as long as the pruning parameter (the number of leaves) is well chosen. For example, for the {\bf Model 1}, a pruned forest with approximately $110$ leaves for each tree has the same empirical risk as the standard Breiman's forest. In the original algorithm of Breiman's forest, the construction of each tree uses a bootstrap sample of the data. For the pruned forests, the whole data set is used for each tree, and then the randomness comes only from the pre-selected directions for splitting. The performances of bootstrapped and pruned forests are very alike. Thus, bootstrap seems not to be the cornerstone of the Breiman's forest practical superiority to other regression algorithms. As it is shown in Corollary \ref{Corollary_1} and the simulations, pruning and sampling of the data set (here bootstrap) are equivalent.

In order to study the optimal pruning value (\texttt{maxnodes} parameter in the R algorithm), we draw the same curves as in Figure \ref{fig1}, for different learning data set sizes ($n=100$, $200$, $300$ and $400$). We also copy in an other  graph the optimal values that we found for each size of the learning set. The optimal pruning value $m^{\star}$ is defined as 
\begin{align*}
m^{\star} = \min \{m : |\hat{L}_{m} - \min_r \hat{L}_r| < 0.05 \times (\max_r \hat{L}_r - \min_r \hat{L}_r) \}
\end{align*}
%
%
%
where $\hat{L}_r$ is the risk of the forest built with the parameter \texttt{maxnodes}$=r$. The results can be seen in Figure \ref{fig2}. According to the last sub-figure in Figure \ref{fig2}, the optimal pruning value seems to be proportional to the sample size. For {\bf Model 1}, the optimal value $m^{\star}$ seems to verify $0.25n < m^{\star} < 0.3n$. The other models show a similar behaviour, as it can be seen in Figure \ref{fig3}.

\begin{figure}[h!!]
\begin{tabular}{cc}
\includegraphics[scale=0.3]{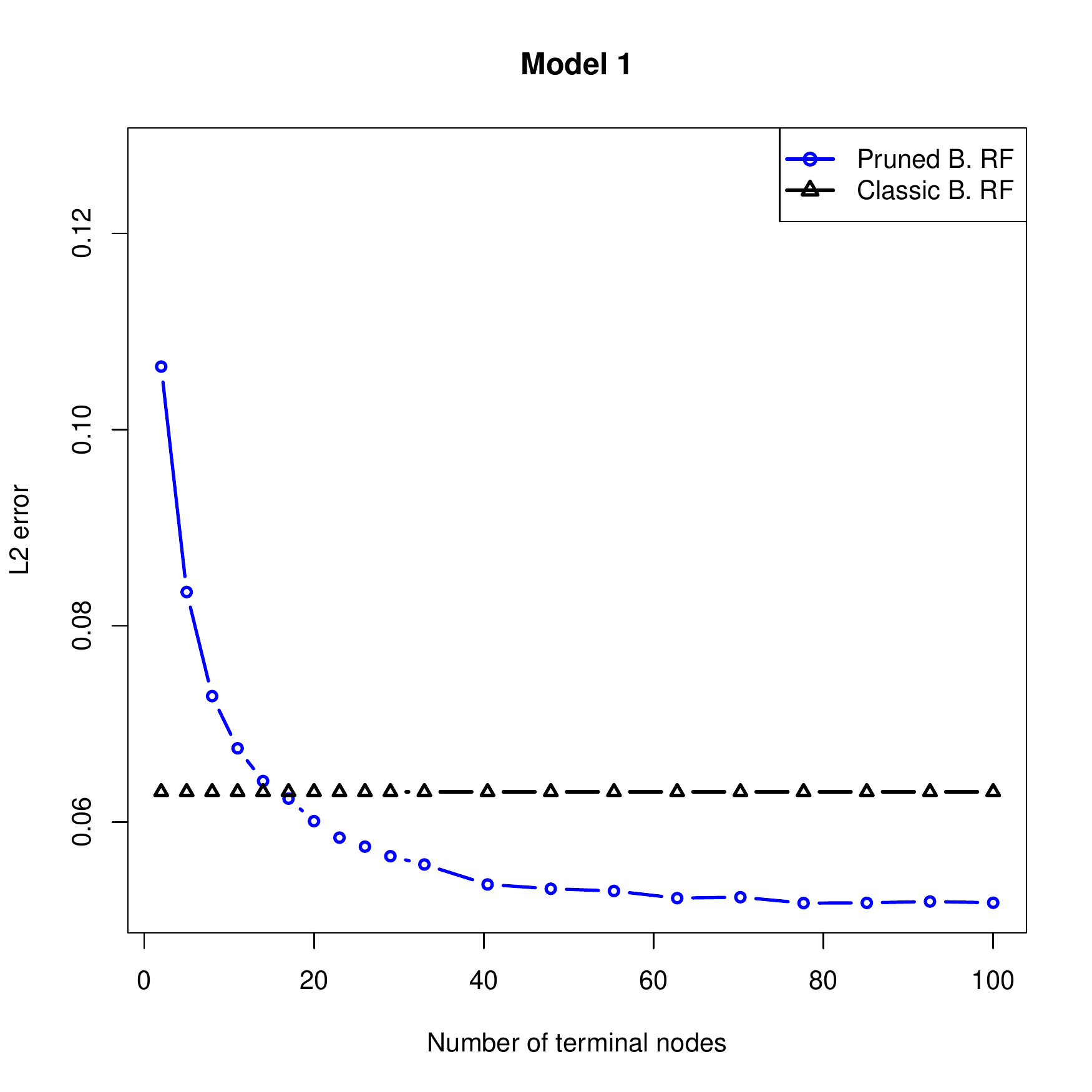}
& \includegraphics[scale=0.3]{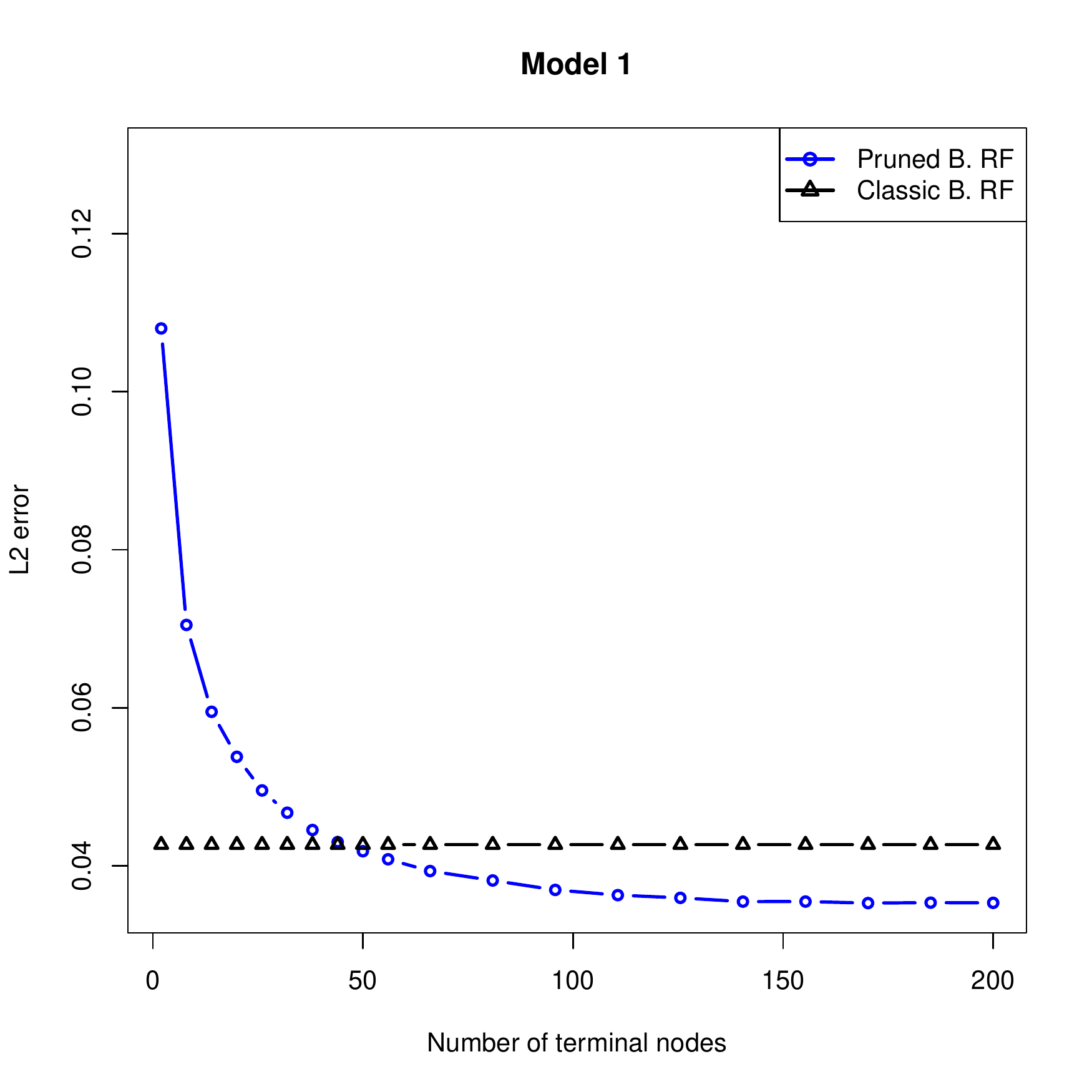}
\\
\includegraphics[scale=0.3]{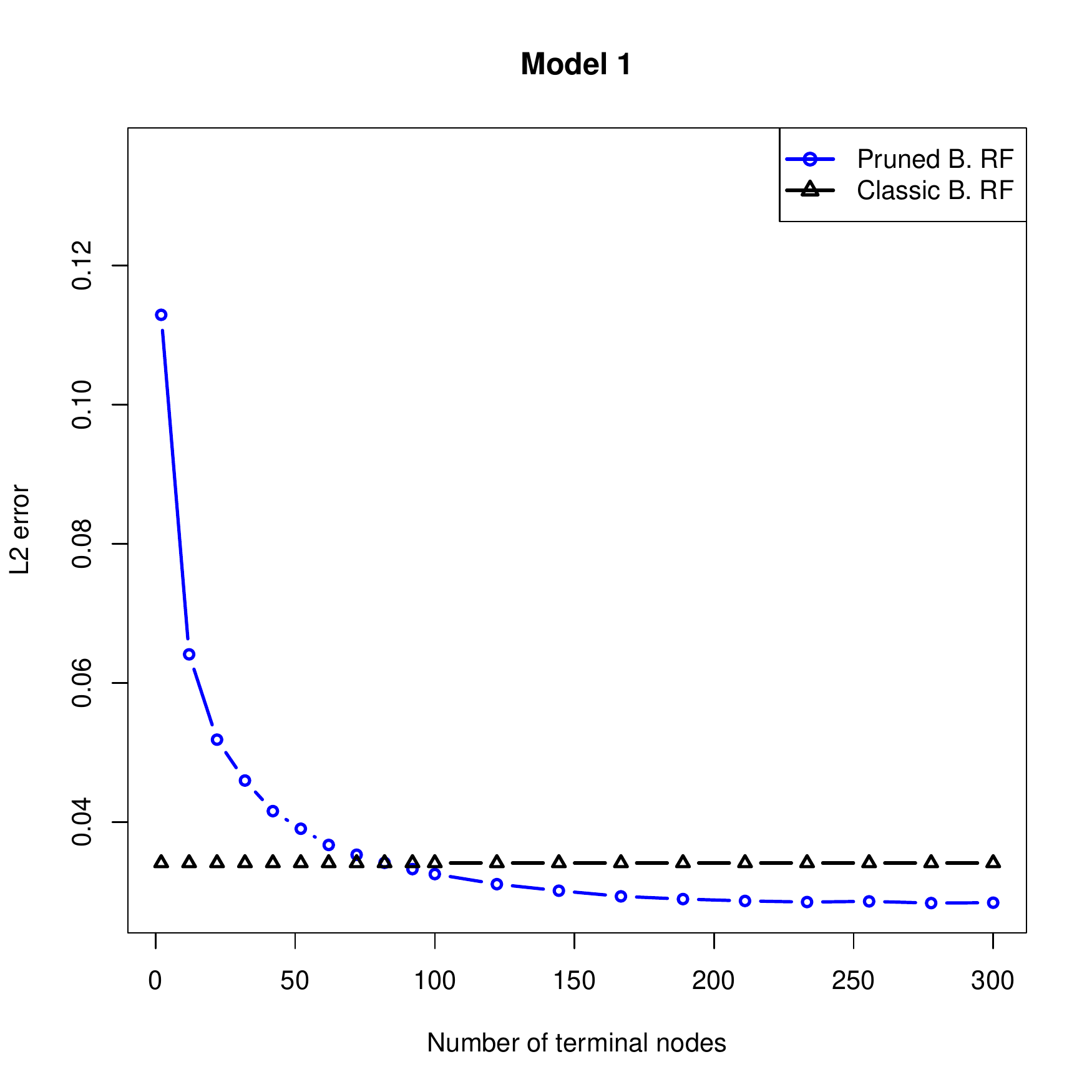}
& \includegraphics[scale=0.3]{choix_profondeur_2_model1_P_vs_B_RF-mod1_ntree500_nb-iter50nombrepoint400_mtrydefault.pdf} \\
\includegraphics[scale=0.3]{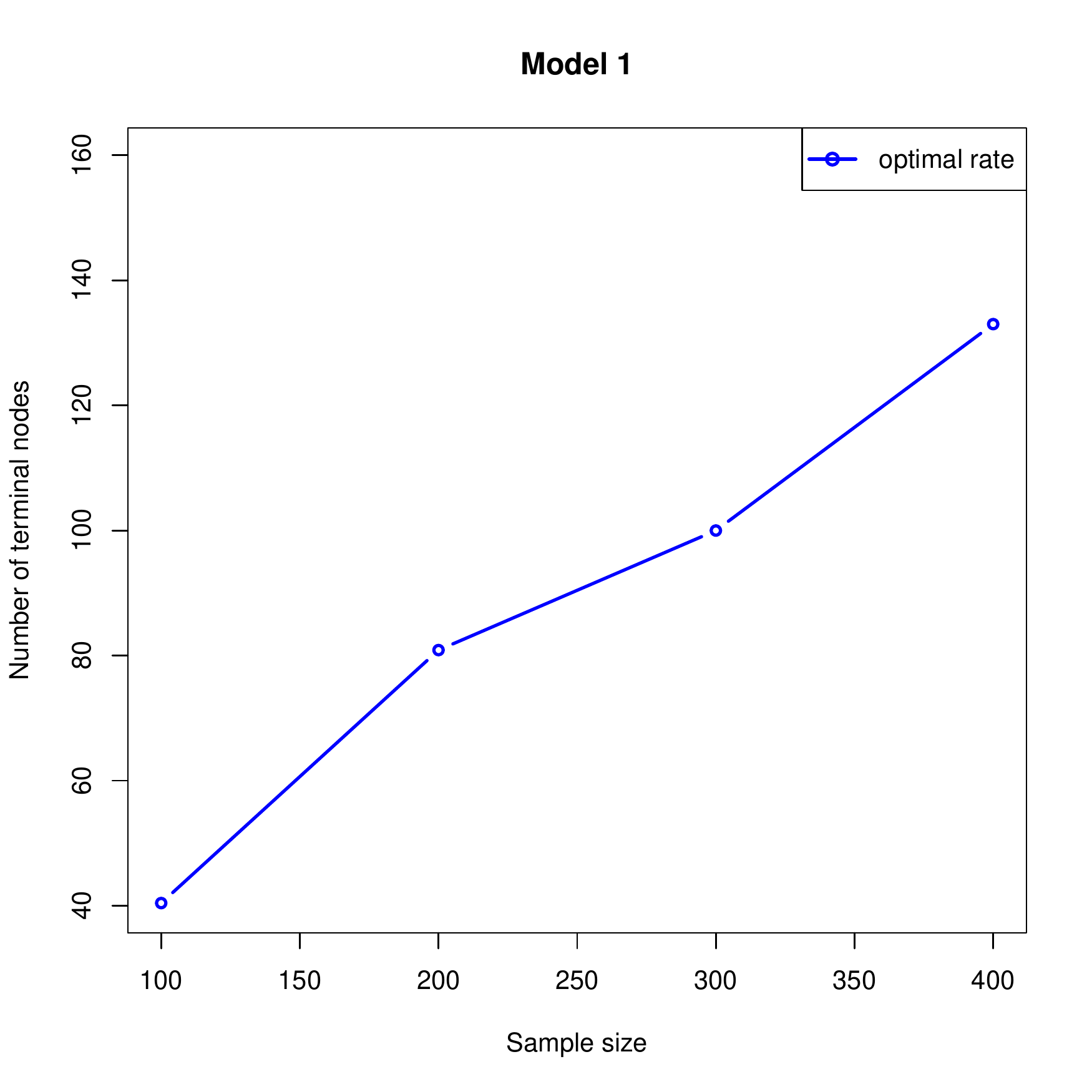}
\end{tabular}
\caption{Tuning of pruning parameter ({\bf Model 1}).}
\label{fig2}
\end{figure}

\begin{figure}[h!!]
\begin{tabular}{cc}
\includegraphics[scale=0.3]{choix_profondeur_2_model1_pruned_optimal_ntree500_nb-iter50_mtrydefault.pdf}
& \includegraphics[scale=0.3]{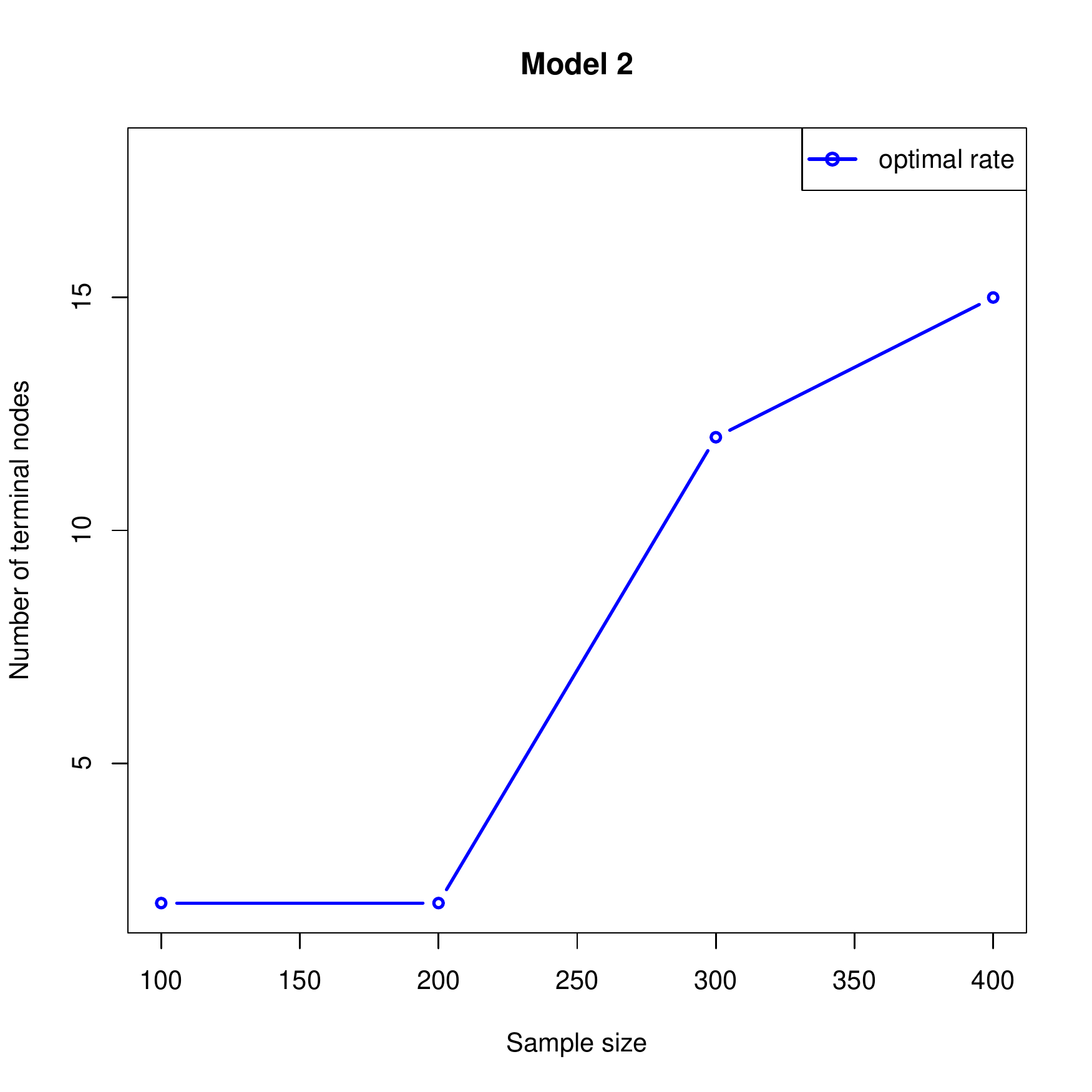}
\\
\includegraphics[scale=0.3]{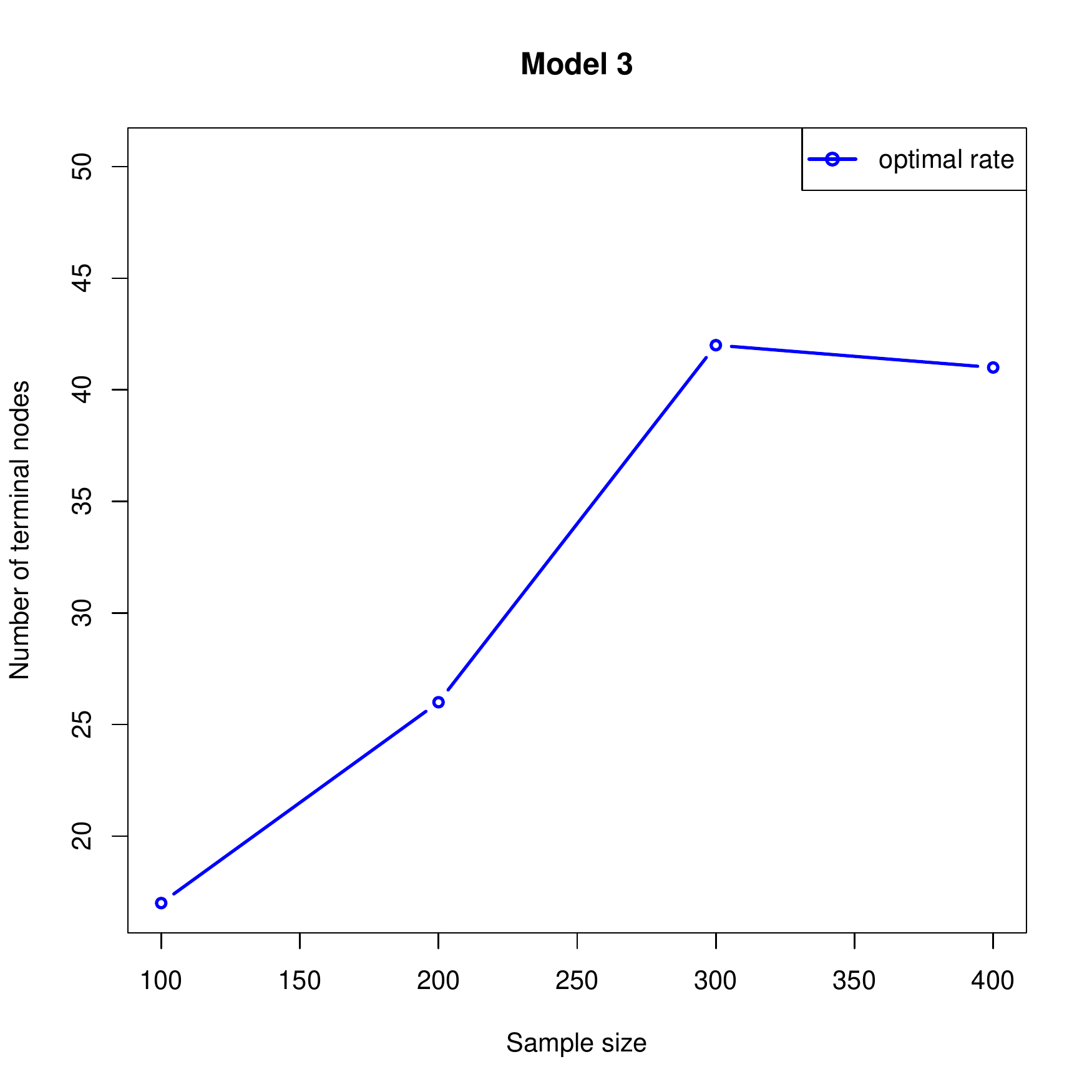}
& \includegraphics[scale=0.3]{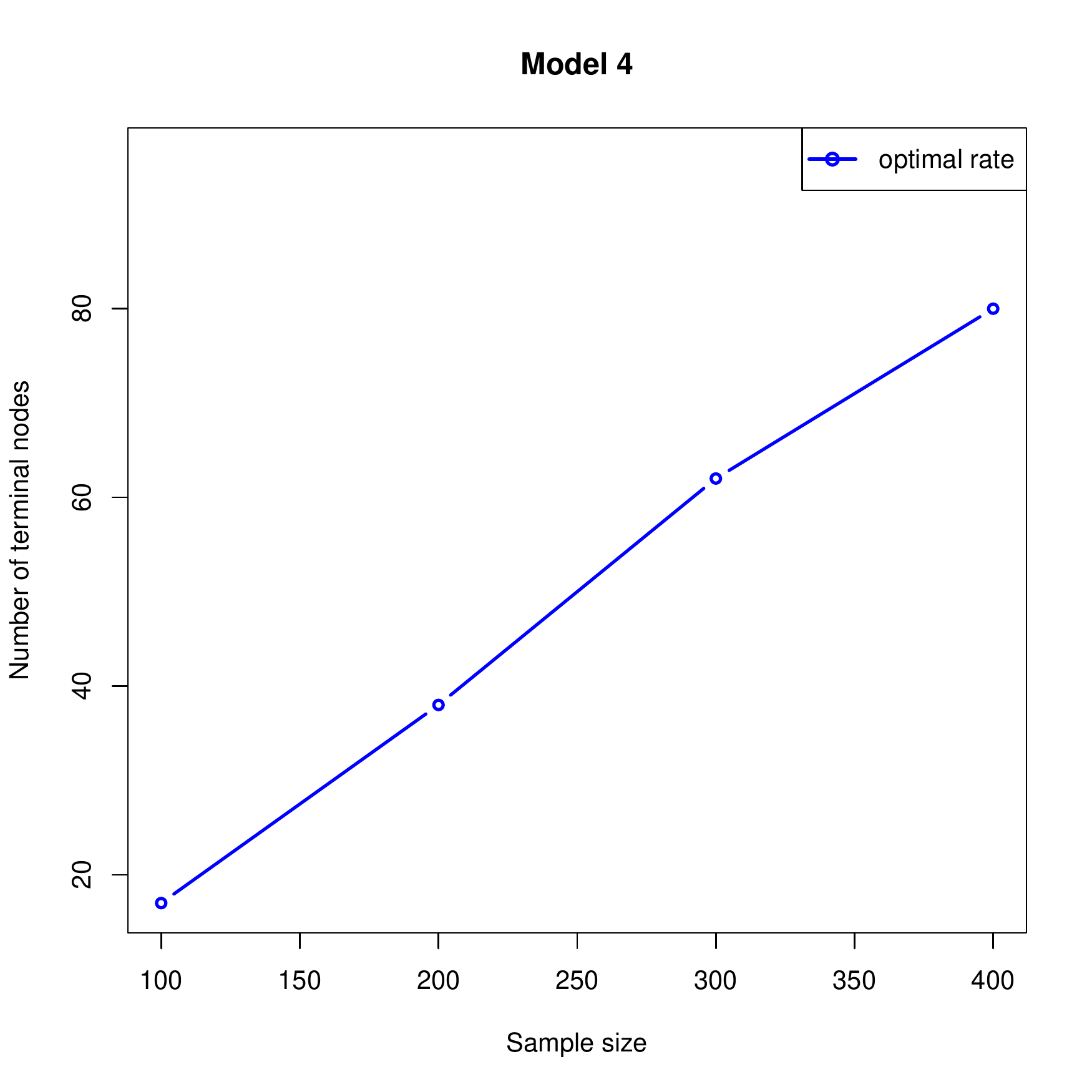}
\\
\includegraphics[scale=0.3]{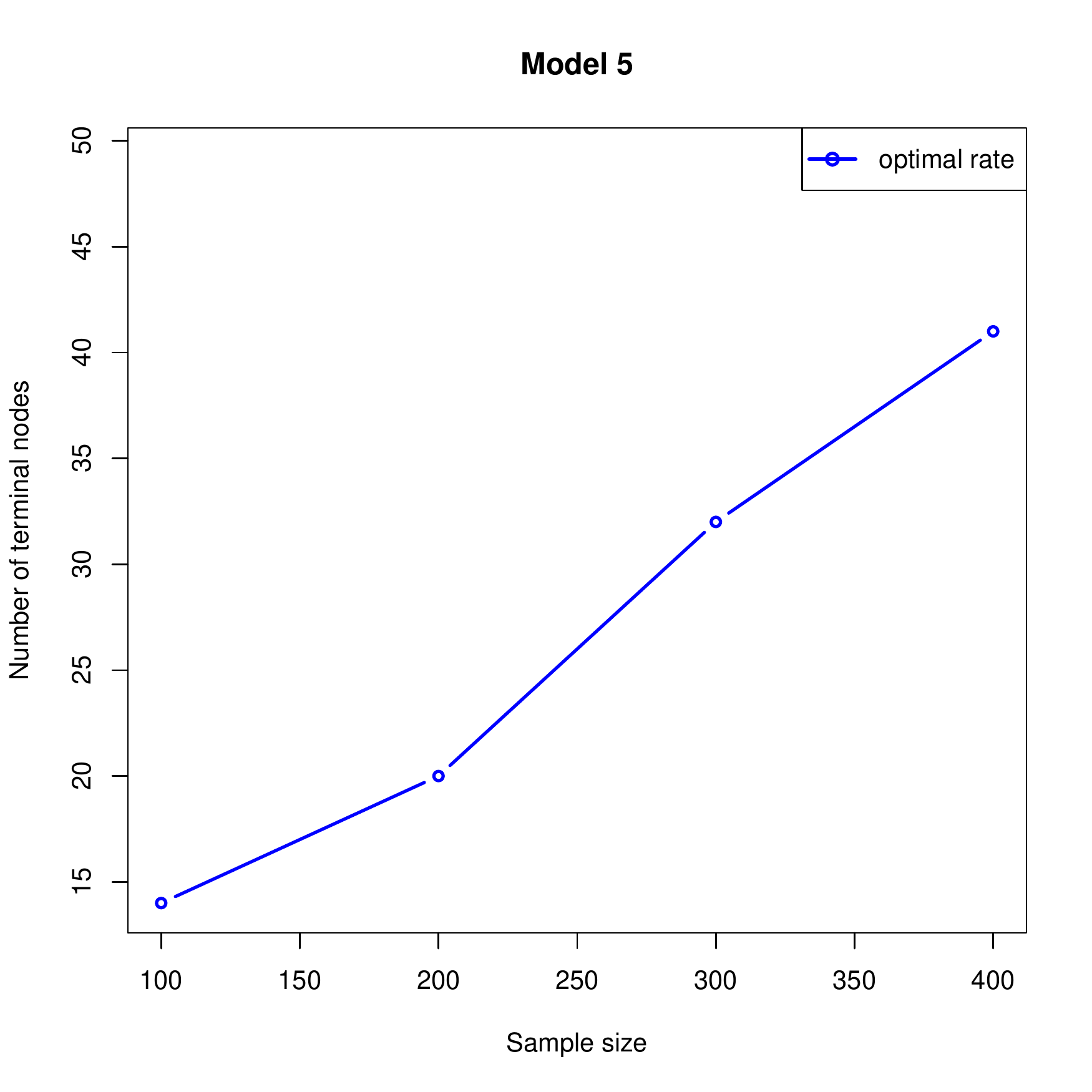}
& \includegraphics[scale=0.3]{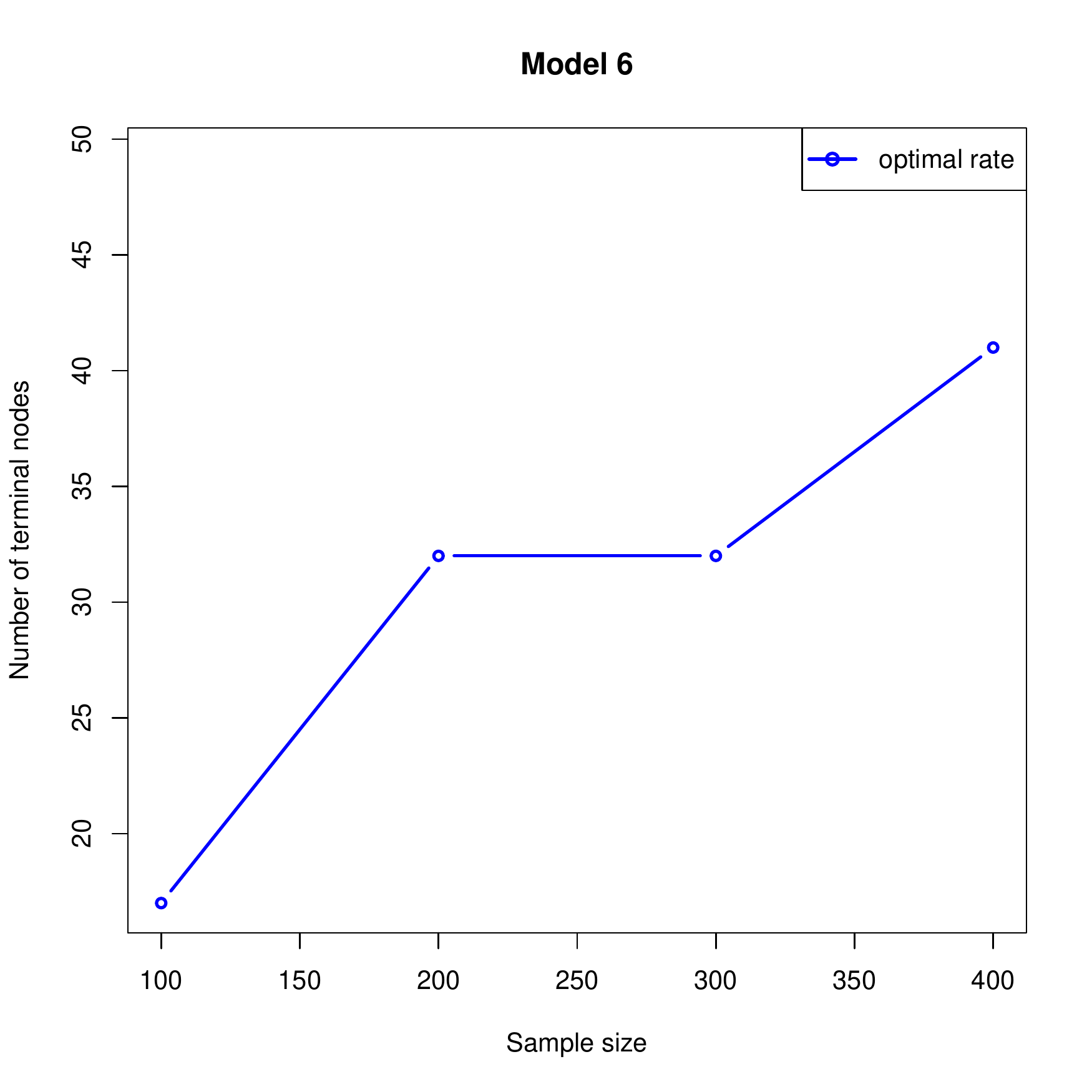}
\\
\includegraphics[scale=0.3]{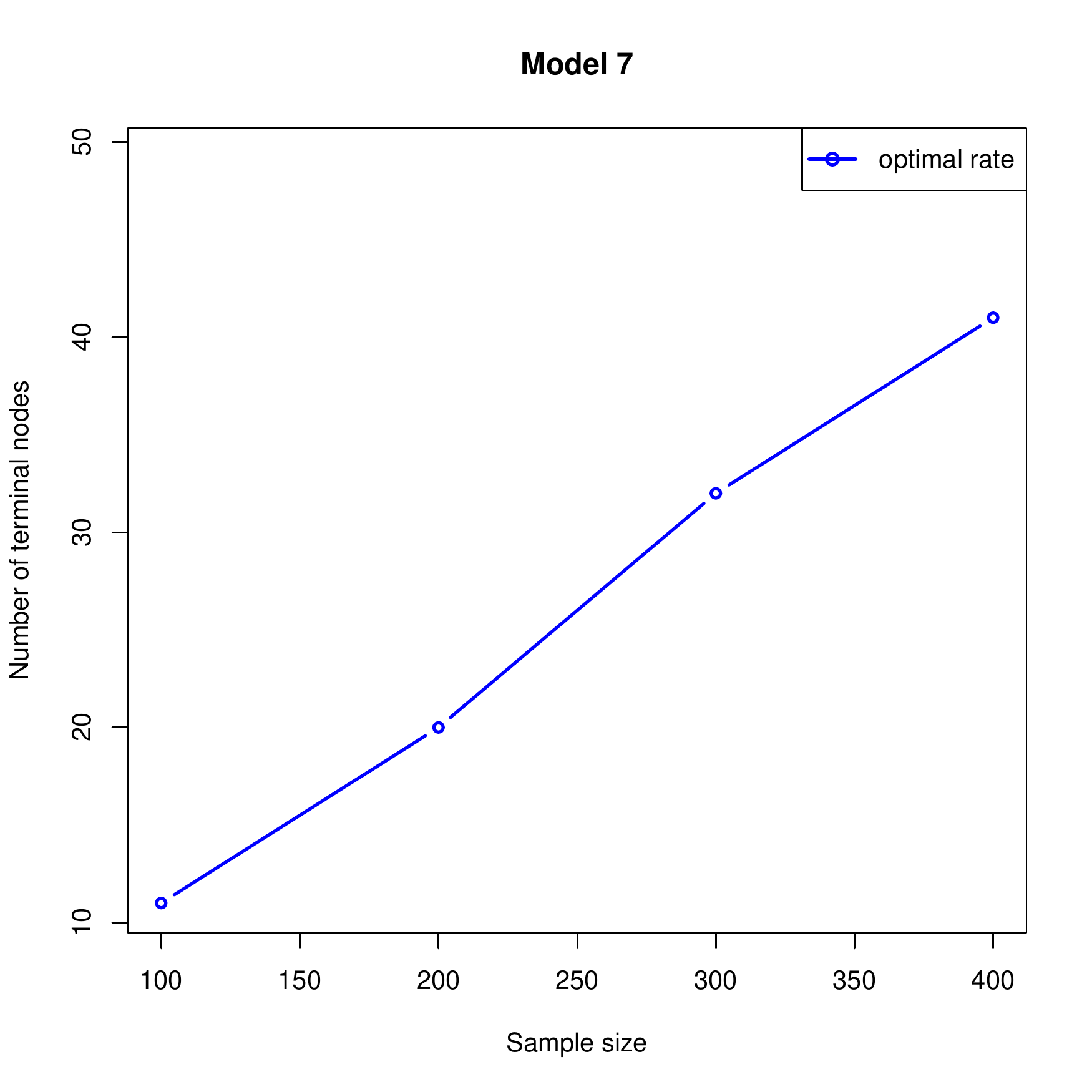}
& \includegraphics[scale=0.3]{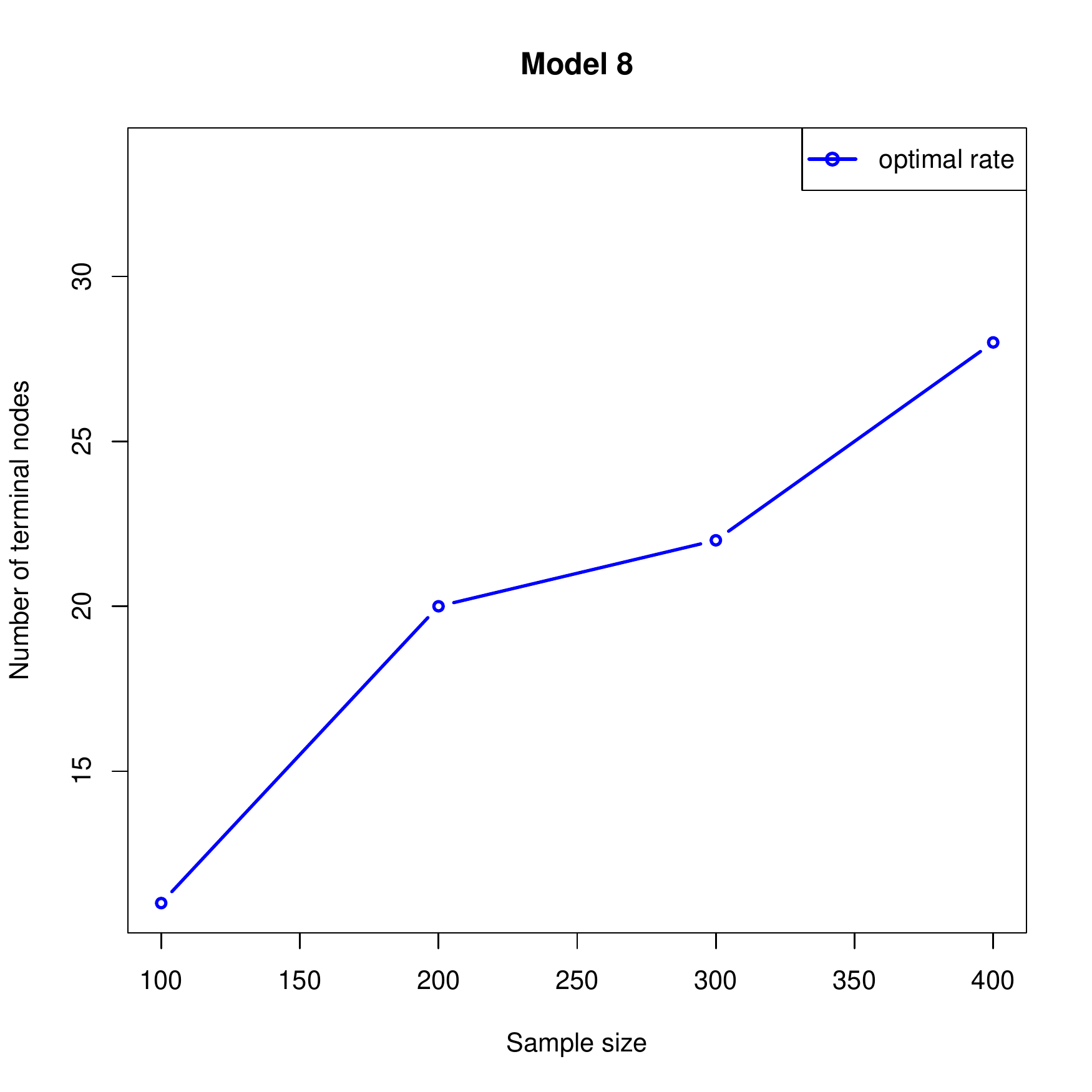}
\end{tabular}
\caption{Optimal values of pruning parameter for {\bf Models 1-8}.}
\label{fig3}
\end{figure}

We also present the $\mathbb{L}^2$ errors of pruned Breiman's forests for different pruning percentages ($10\%$, $30\%$, $63\%$, $80\%$ and $100\%$), when the sample size is fixed, for {\bf Models 1-8}. The results can be found in Figure \ref{fig4} in the form of box-plots. We can notice that the forests with a $30\%$ pruning (\textit{i.e.}, such that \texttt{maxnodes}$=0.3n$) give similar ({\bf Model 5}) or best ({\bf Model 6}) performances than the standard Breiman's forest.

\begin{figure}[h!!]
\begin{tabular}{cc}
\includegraphics[scale=0.3]{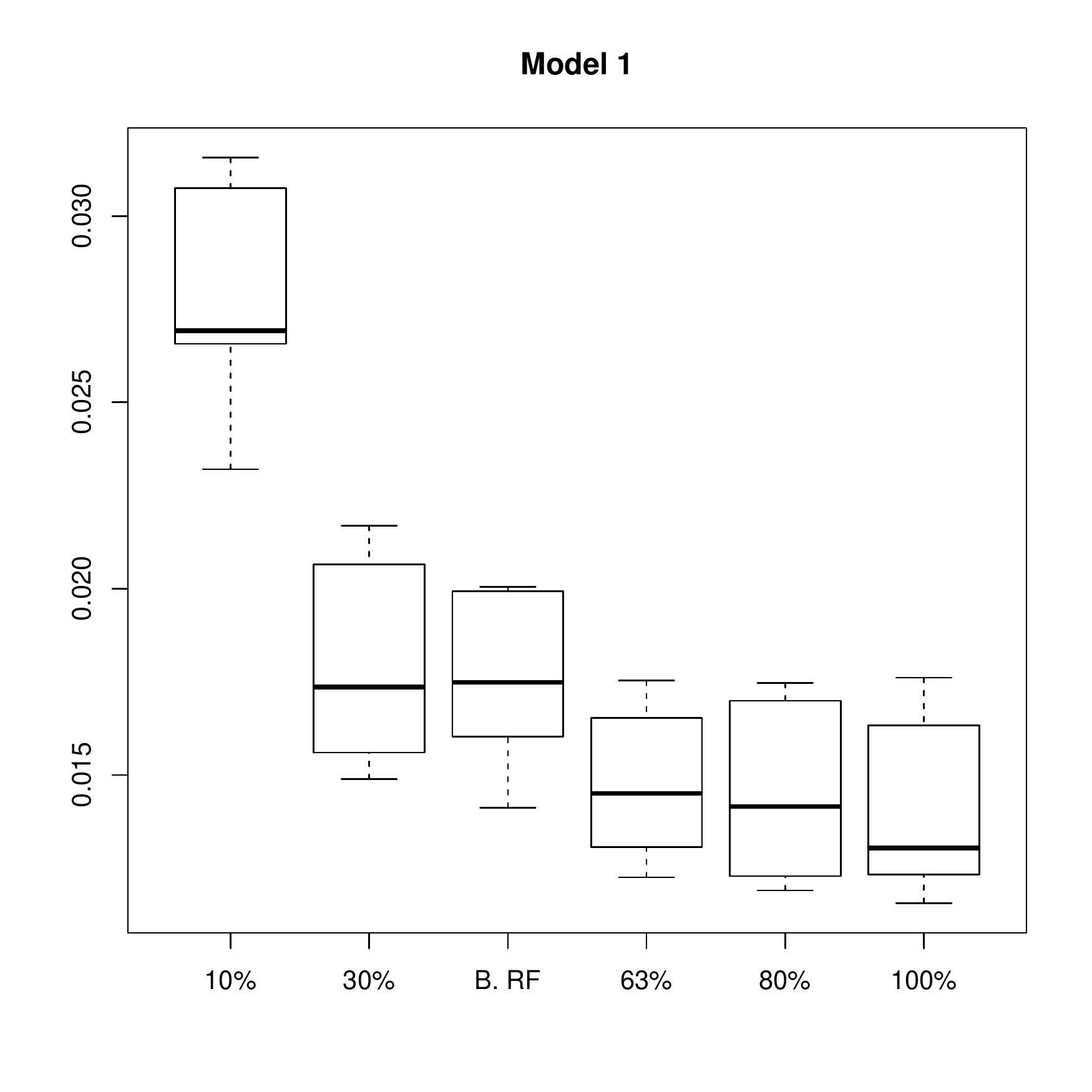}
& \includegraphics[scale=0.3]{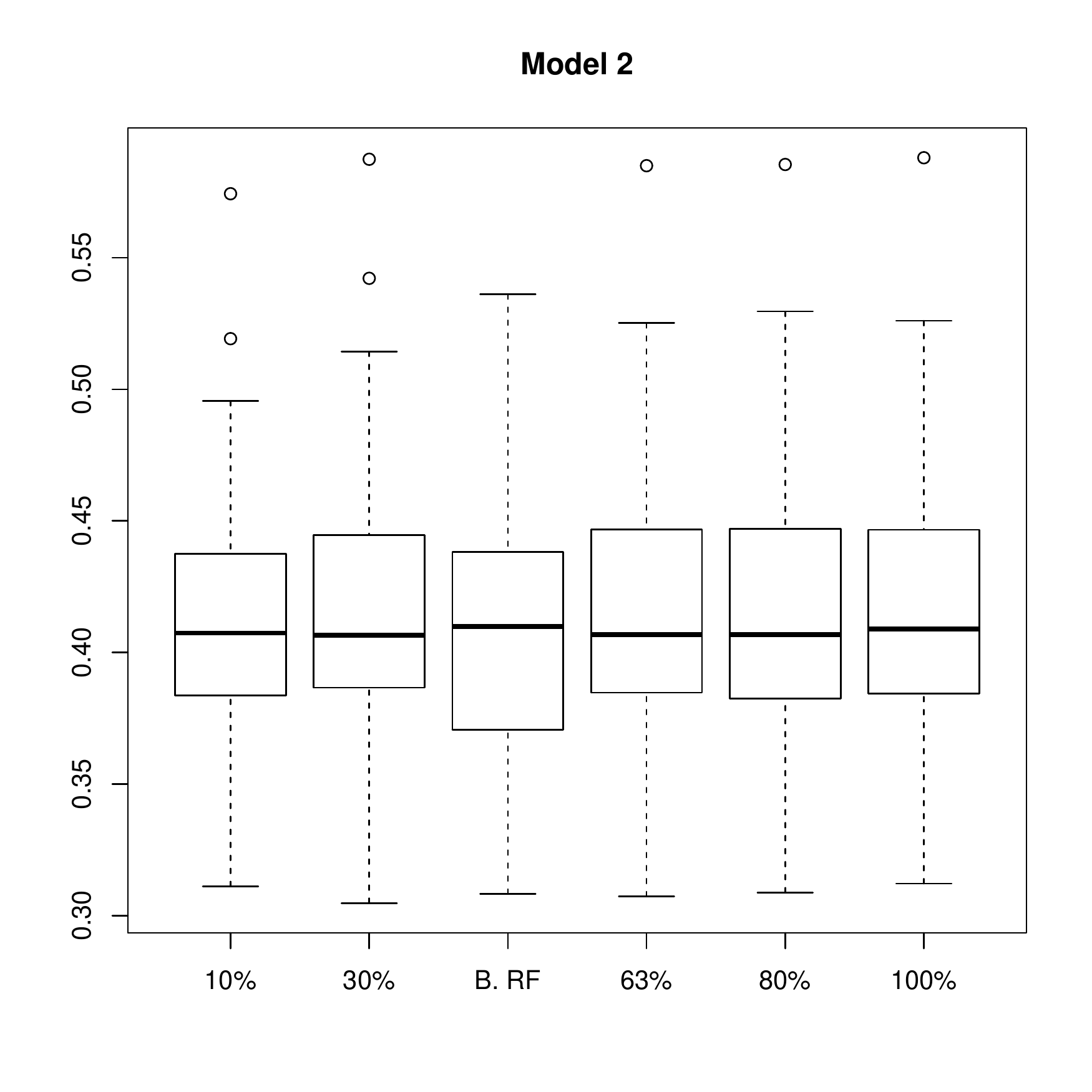}
\\
\includegraphics[scale=0.3]{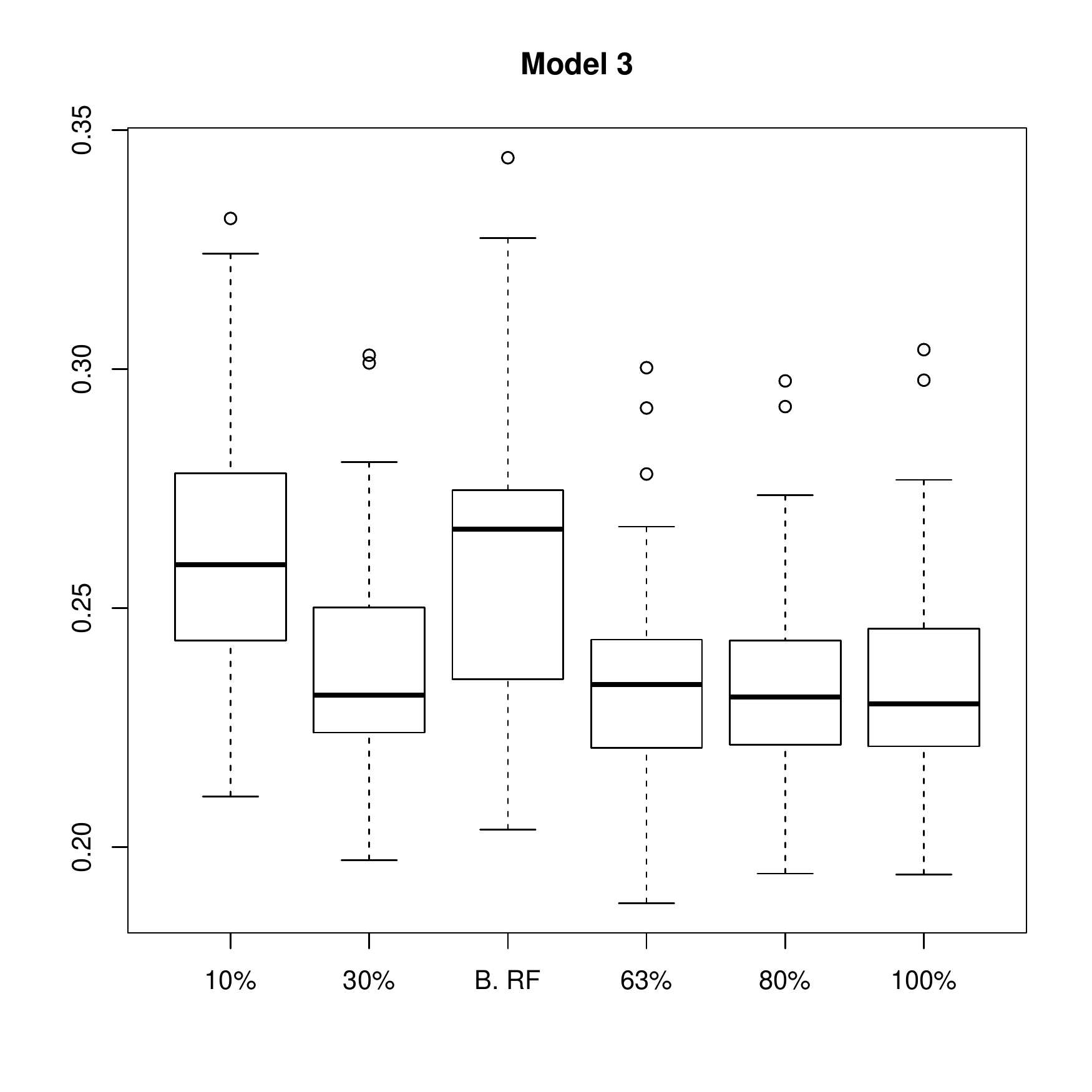}
& \includegraphics[scale=0.3]{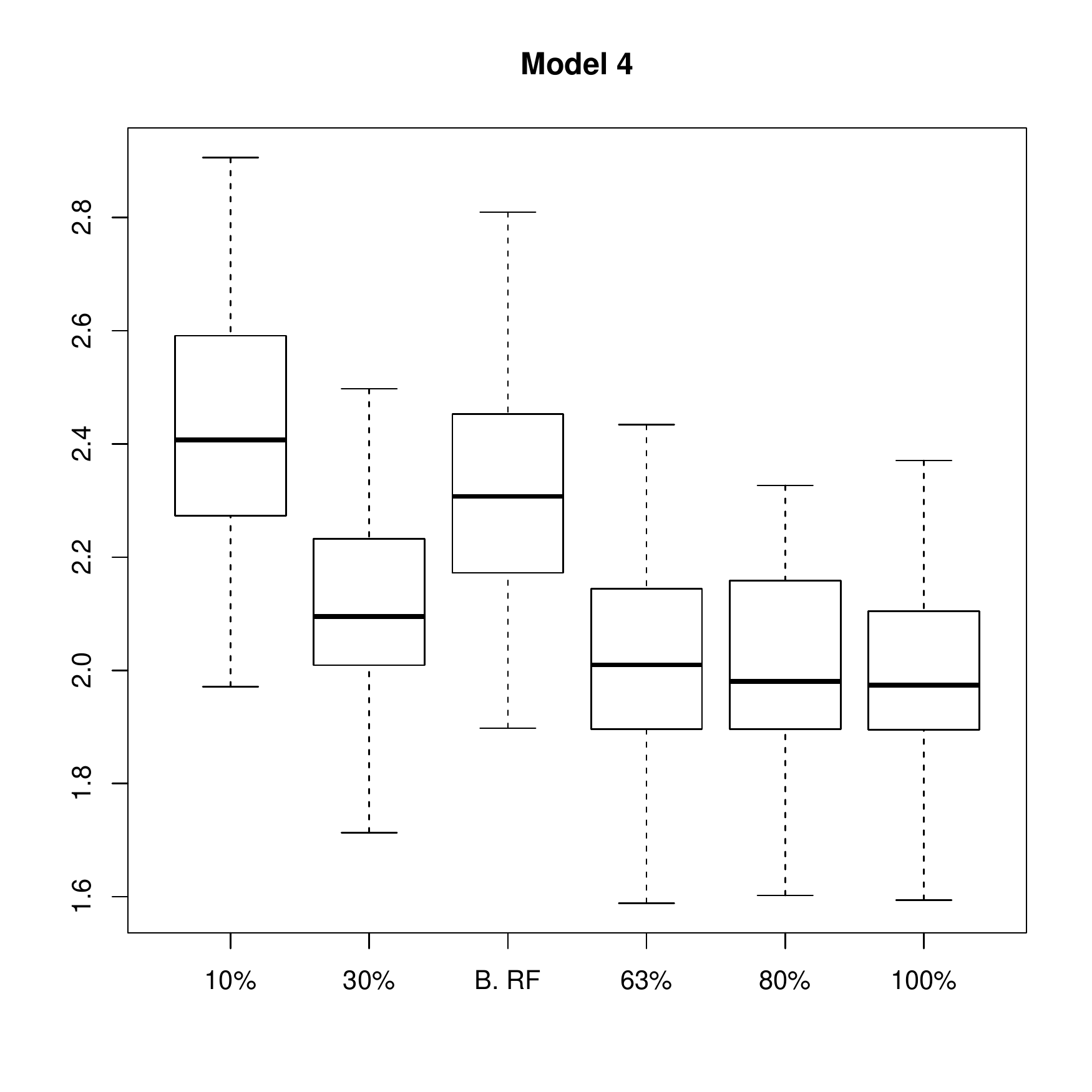}
\\
\includegraphics[scale=0.3]{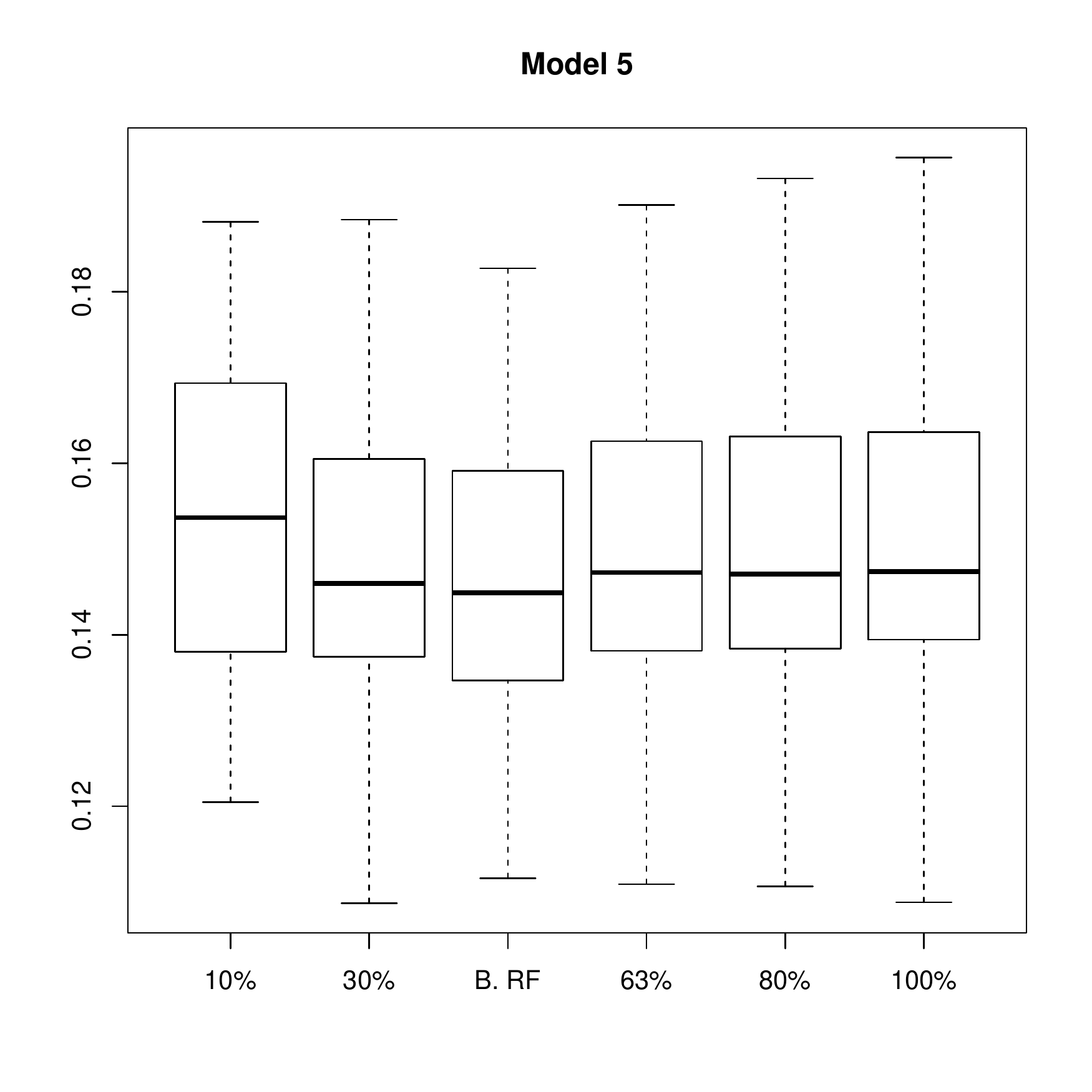}
& \includegraphics[scale=0.3]{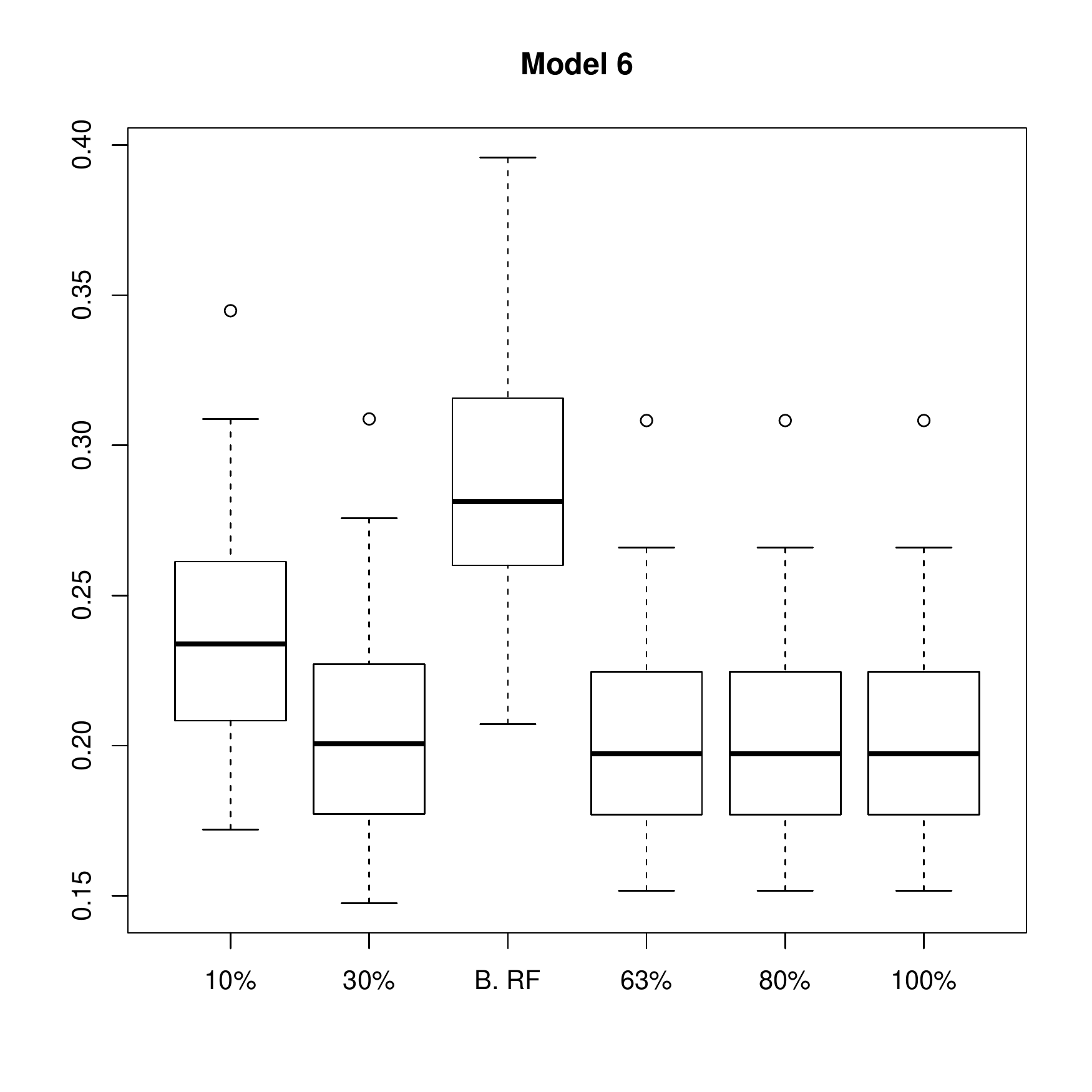}
\\
\includegraphics[scale=0.3]{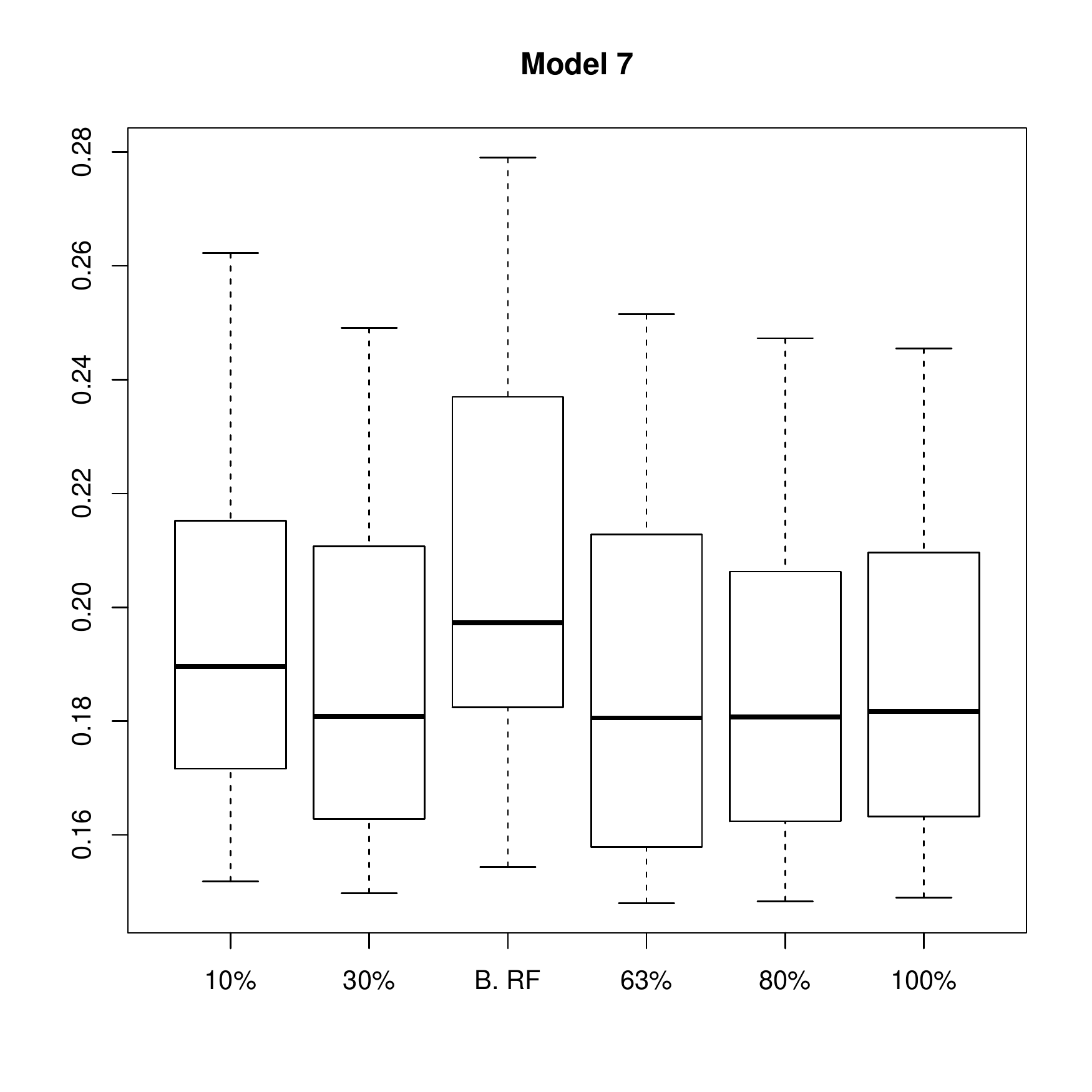}
& \includegraphics[scale=0.3]{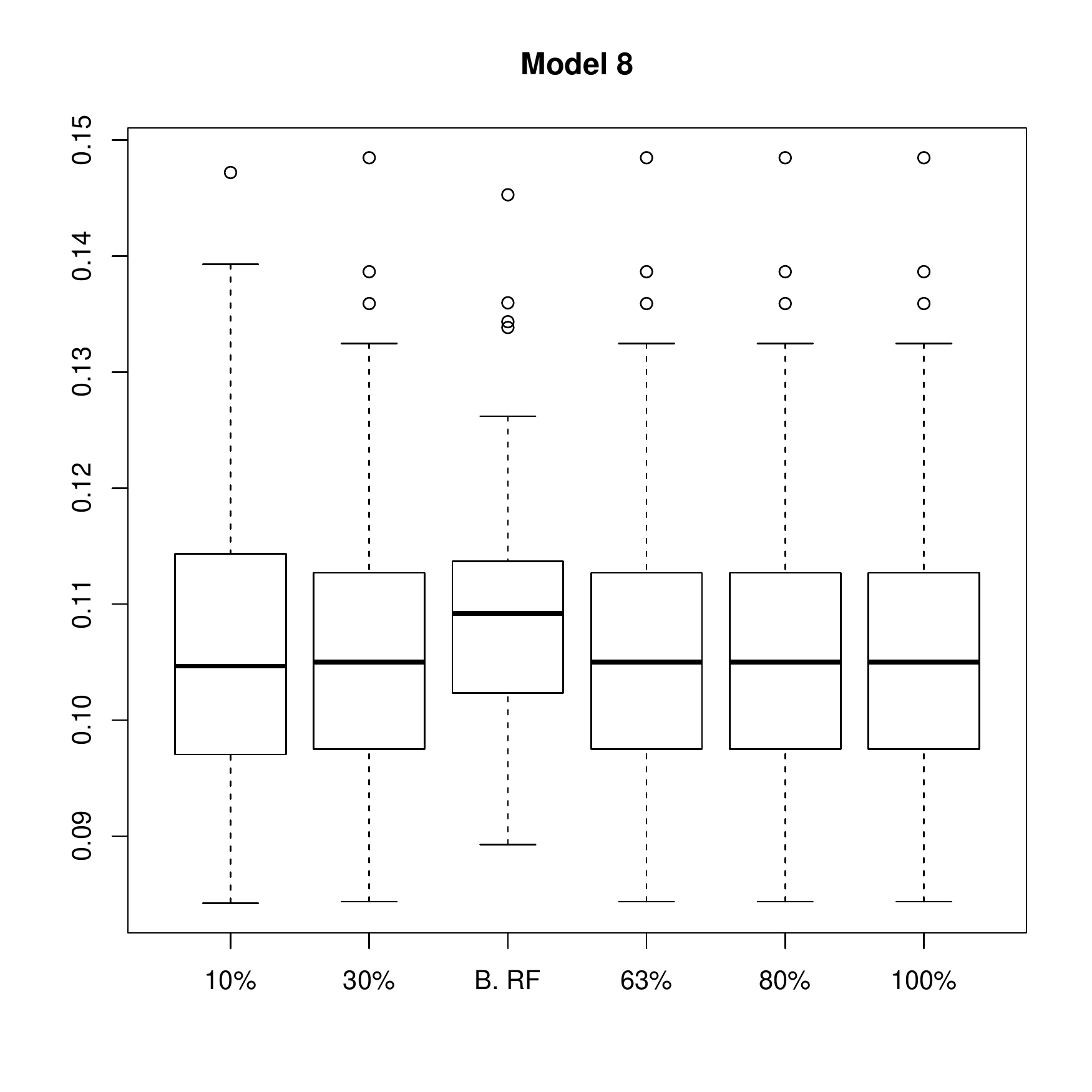}
\end{tabular}
\caption{Comparison of standard Breiman's forests against several pruned Breiman forests in terms of $\mathbb{L}^2$ error.}
\label{fig4}
\end{figure}

\subsection{Subsampling}

In this section, we study the influence of subsampling on Breiman's forests by comparing the original Breiman's procedure with subsampled Breiman's forests. Subsampled Breiman's forests are nothing but Breiman's forests where the subsampling step consists in choosing $a_n$ observations without replacement (instead of choosing $n$ observations among $n$ with replacement), where $a_n$ is the subsample size. Comparison of Breiman's forests and subsampled Breiman's forests is presented in  Figure \ref{fig5} for the {\bf Models 1-8} introduced previously. More precisely, we can see the evolution of the empirical risk of subsampled forests with different subsampling values, and the empirical risk of the Breiman's forest as a reference. Every sub-figure of Figure \ref{fig5} presents forests built with $500$ trees. The printed errors are obtained by averaging the risks of $50$ forests.

\begin{figure}[h!!]
\begin{tabular}{cc}
\includegraphics[scale=0.3]{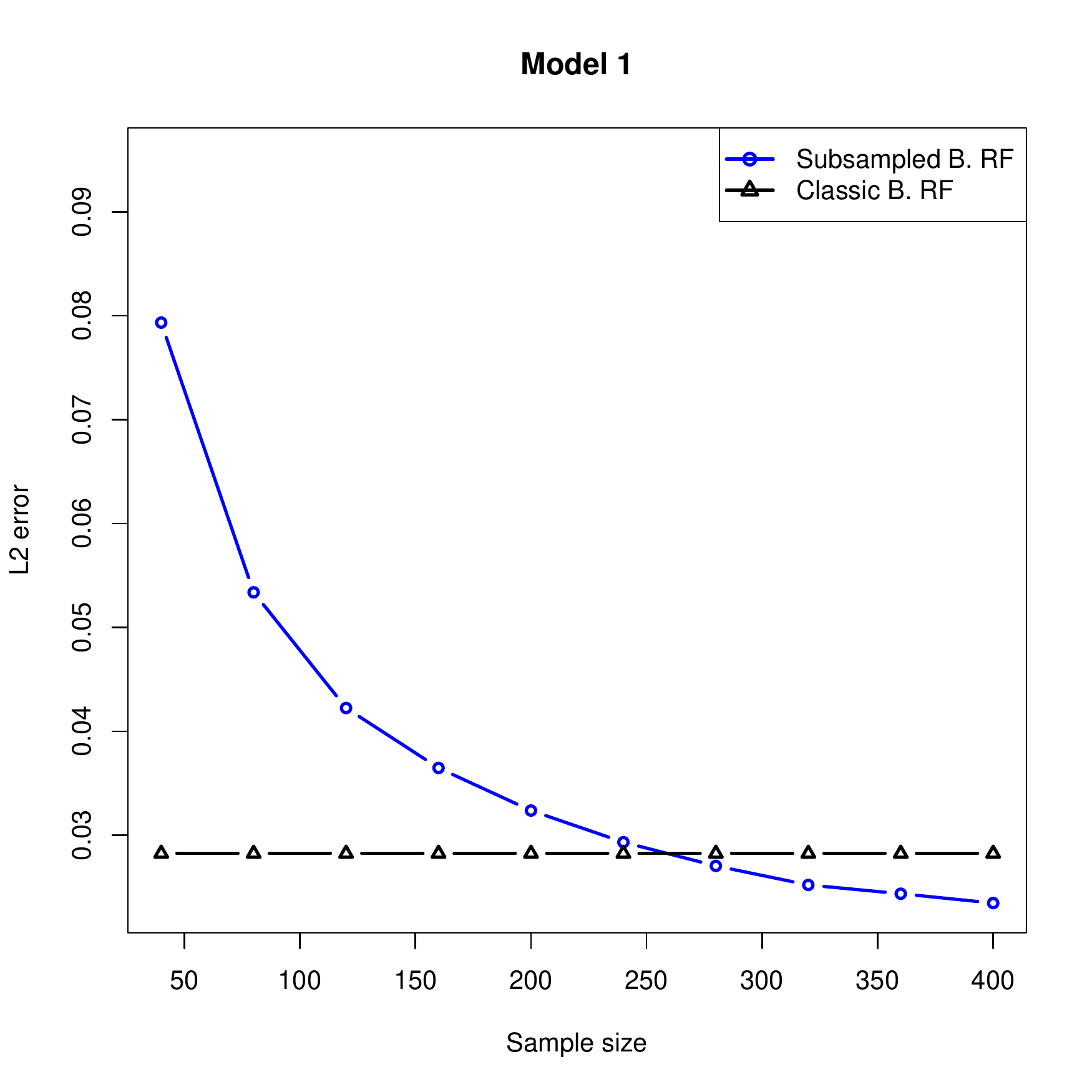}
& \includegraphics[scale=0.3]{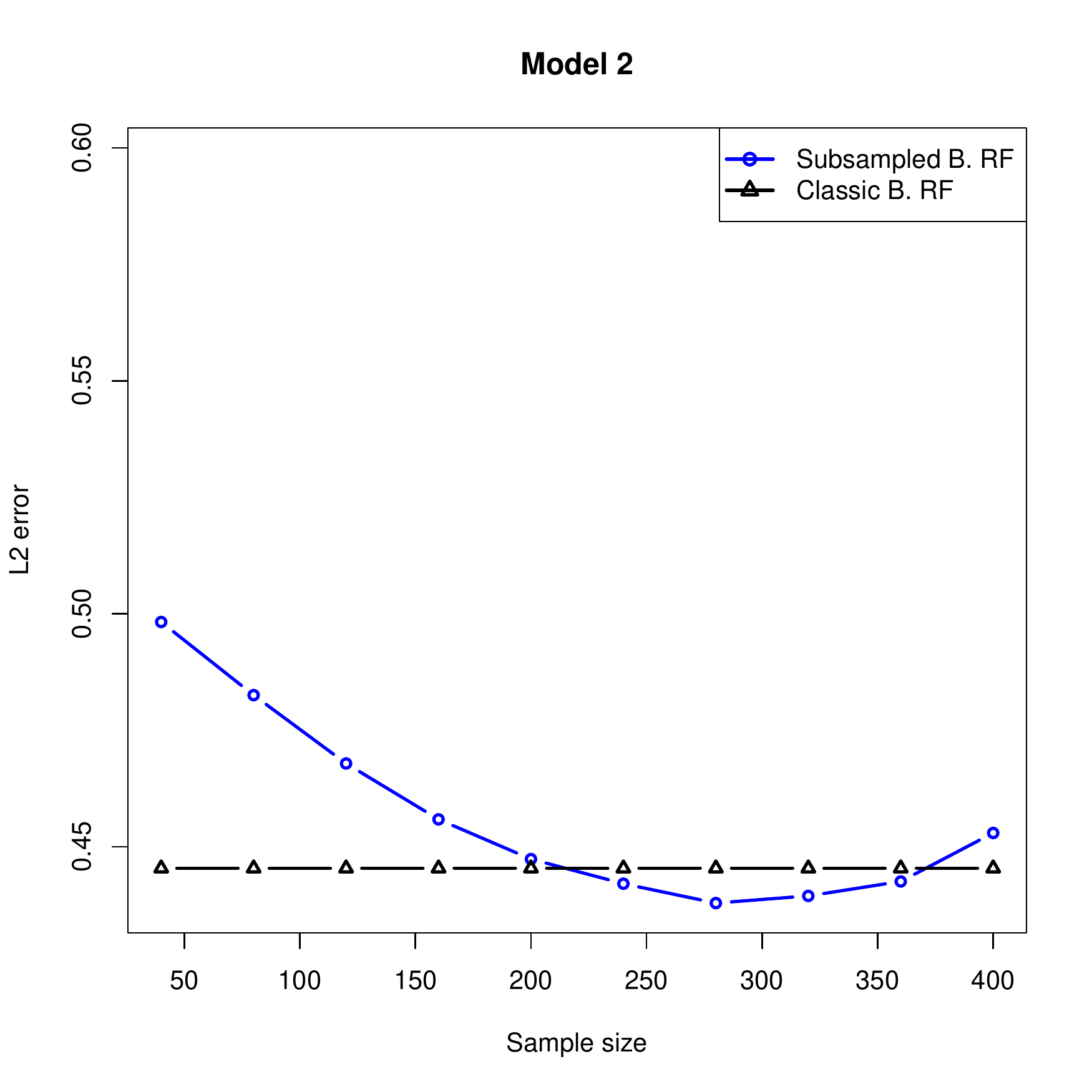}
\\
\includegraphics[scale=0.3]{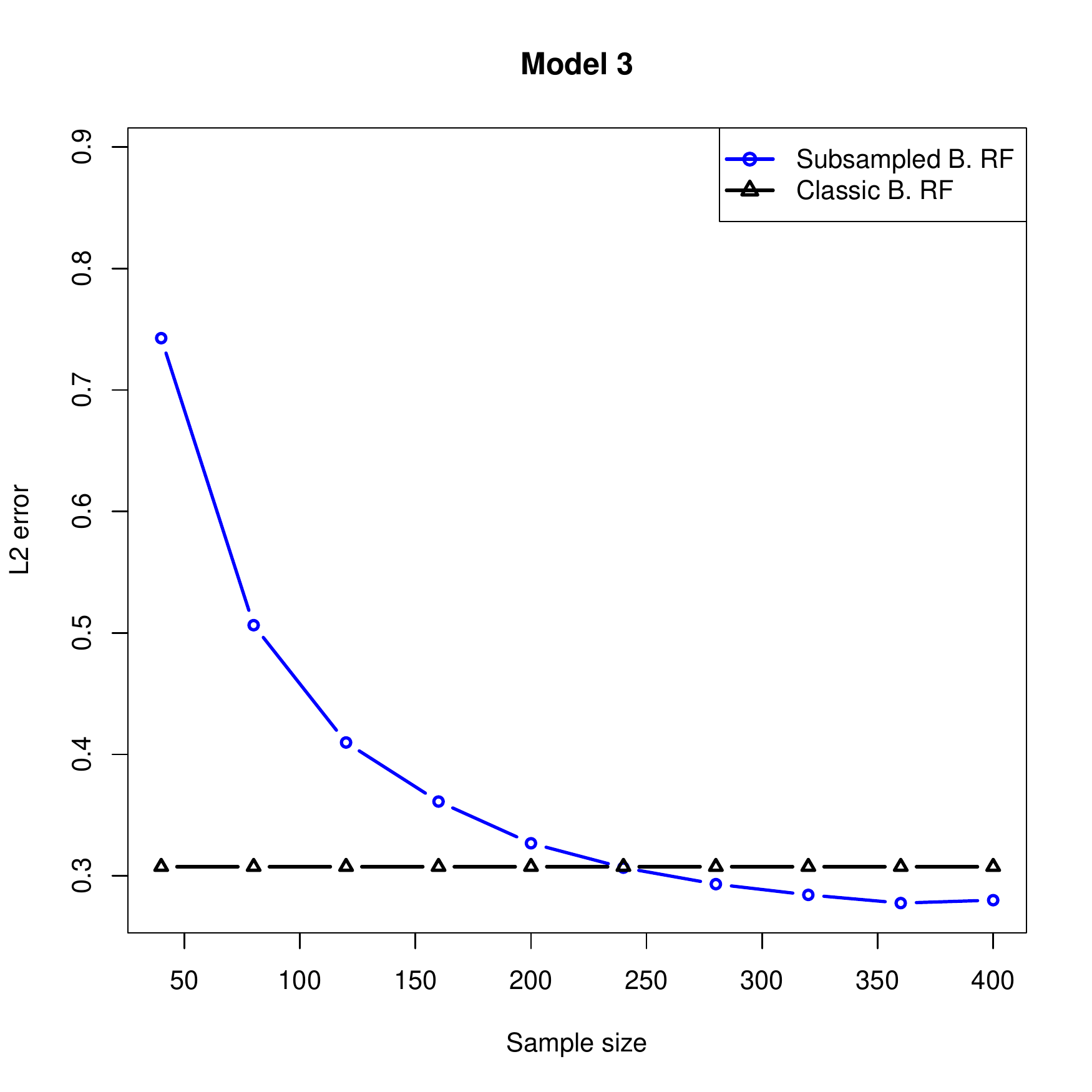}
& \includegraphics[scale=0.3]{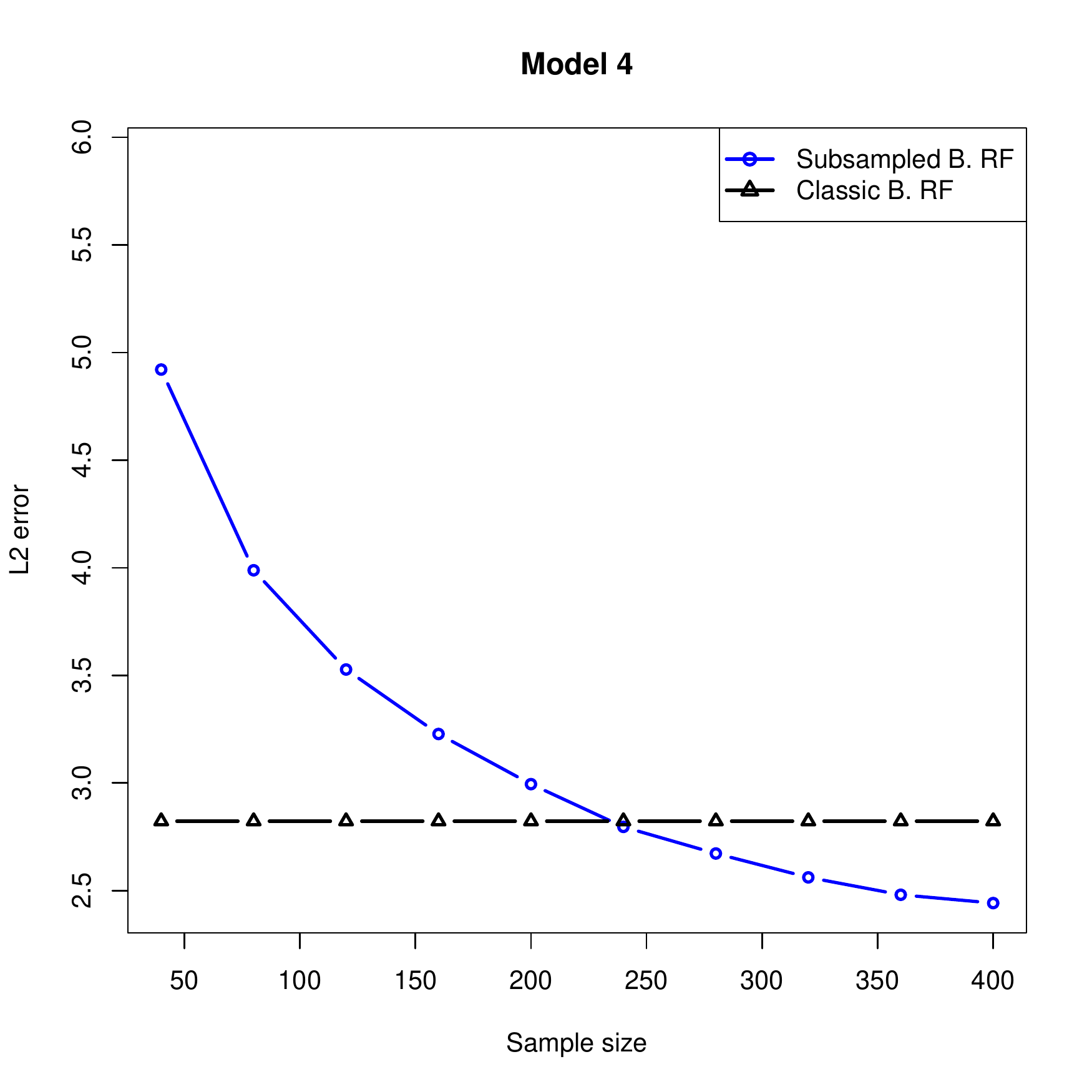}
\\
\includegraphics[scale=0.3]{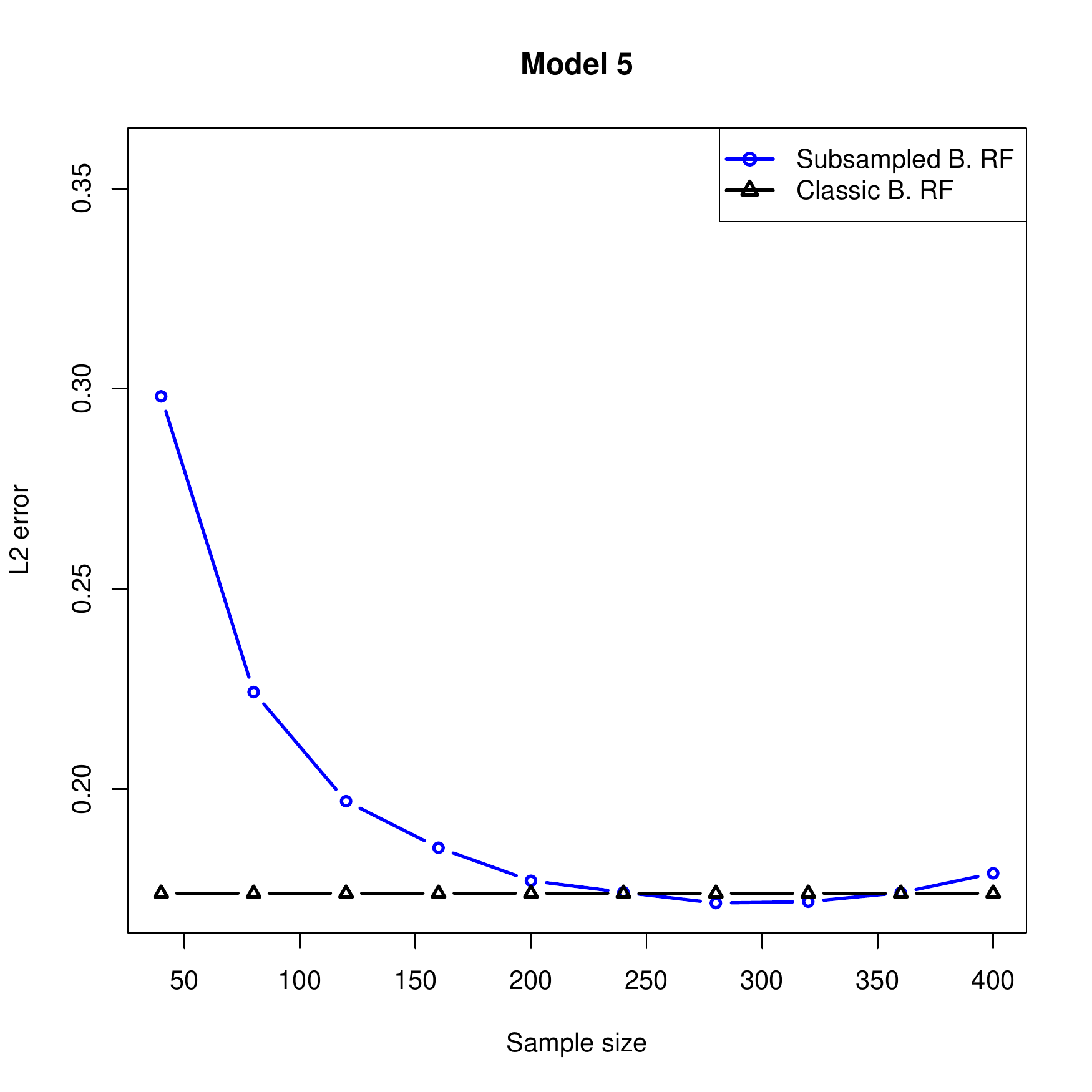}
& \includegraphics[scale=0.3]{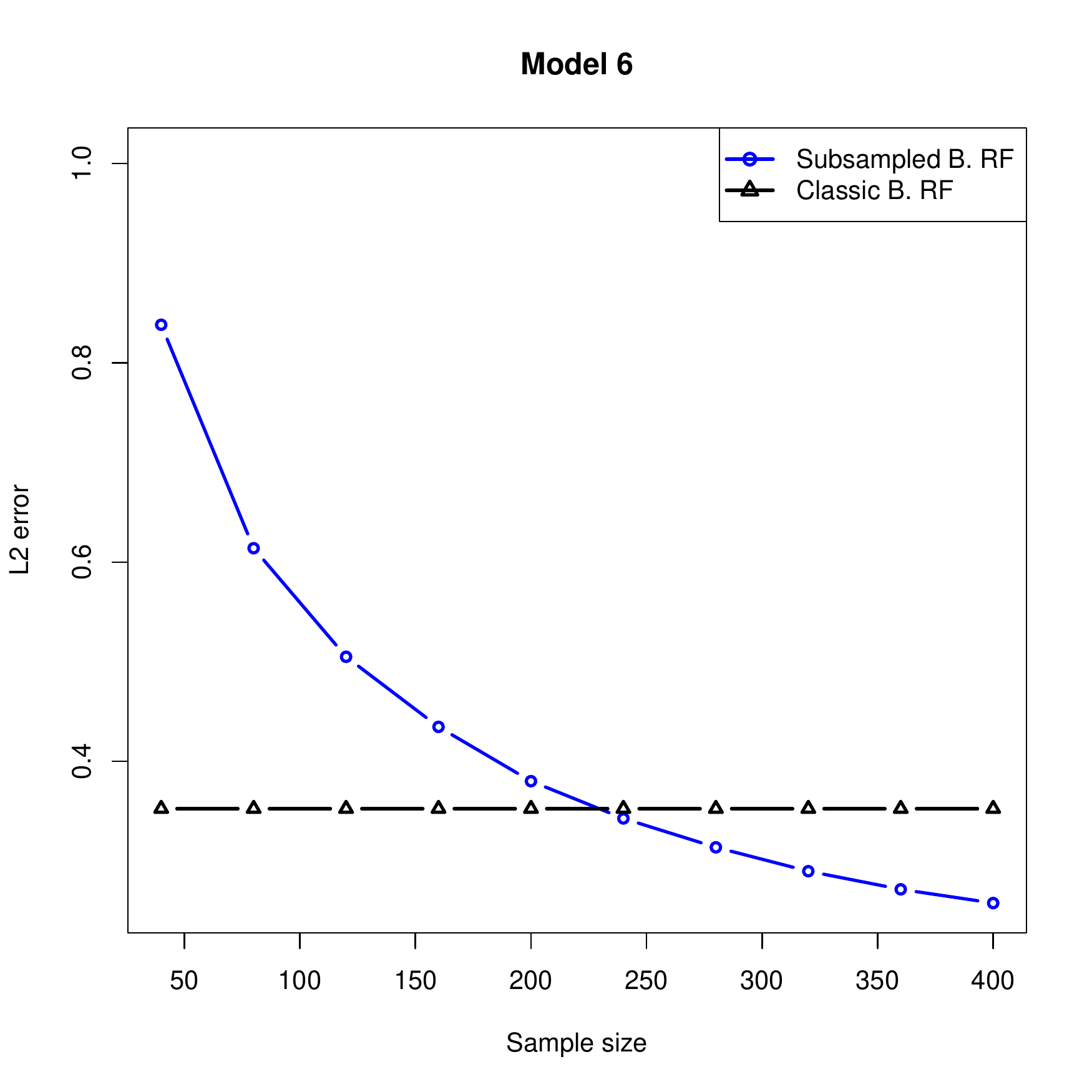}
\\
\includegraphics[scale=0.3]{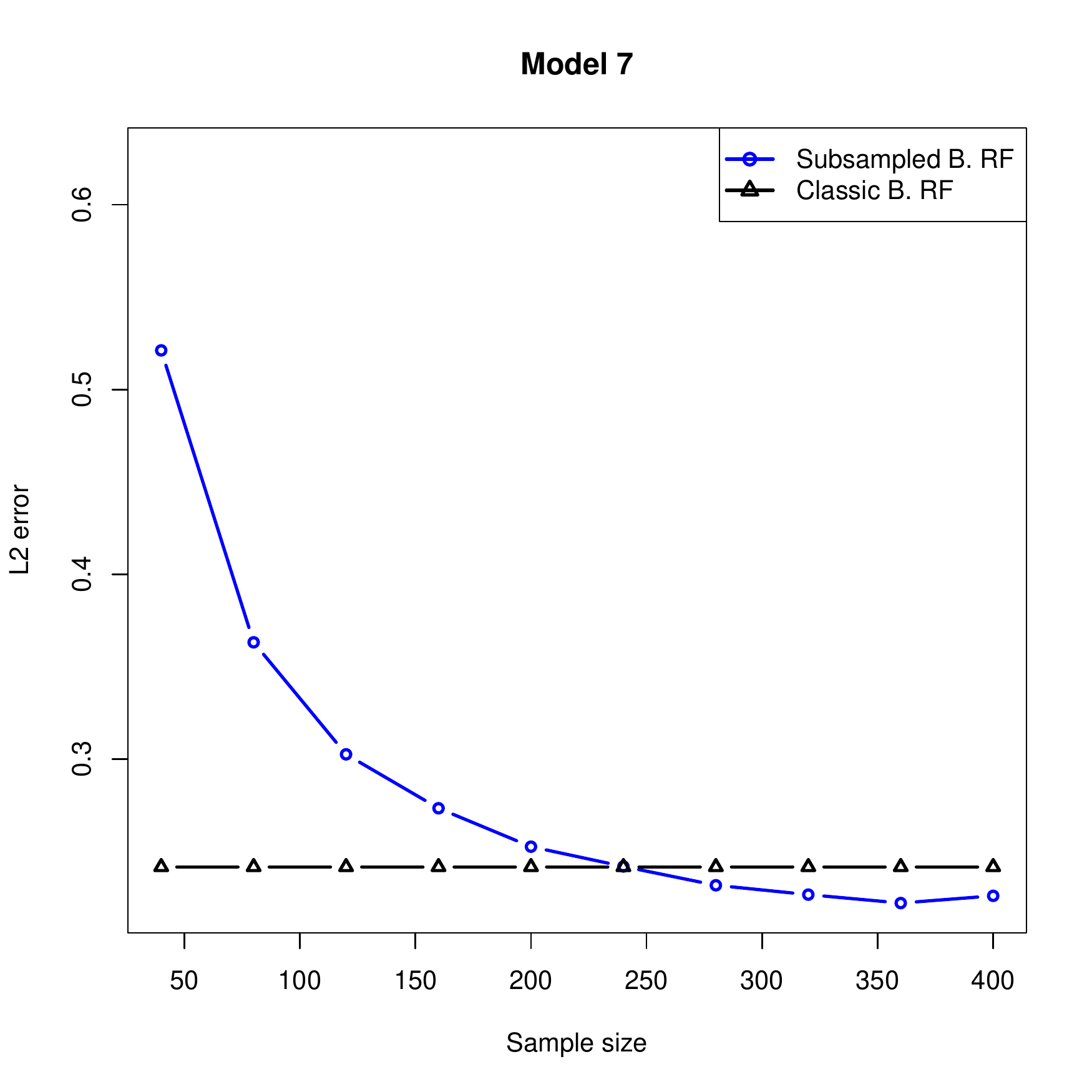}
& \includegraphics[scale=0.3]{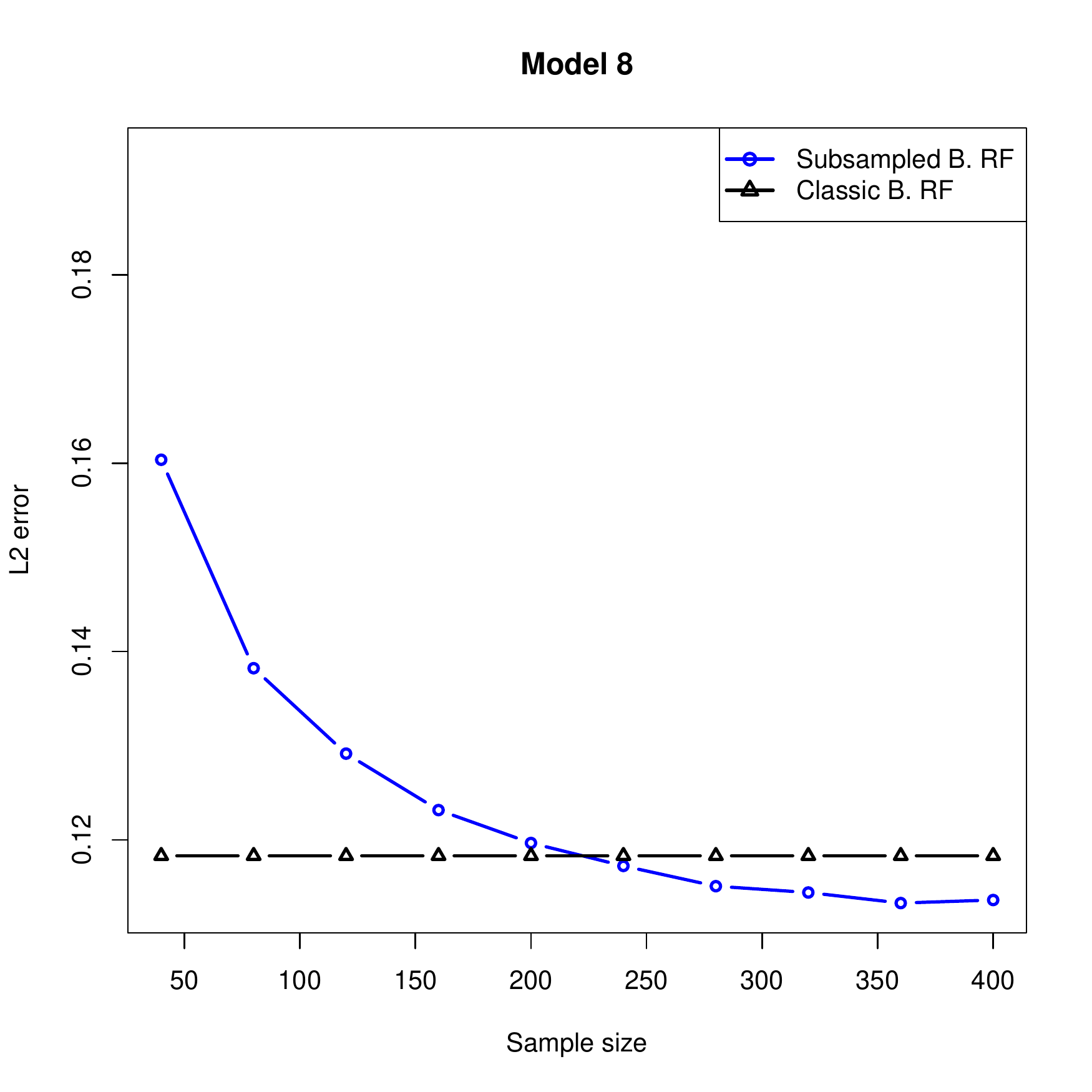}
\\
\end{tabular}
\caption{Standard Breiman Forests versus Subsampled Breiman Forests.}
\label{fig5}
\end{figure}

For every model, we can notice that subsampled forests performance is comparable with the one of standard Breiman's forest, as long as the subsampling parameter is well chosen. For example, a forest with a subsampling rate of $50\%$ has the same empirical risk as the standard Breiman's forest, for {\bf Model 2}. Once again, the similarity between bootstrapped and subsampled Breiman's forests moves aside bootstrap as a performance criteria. As it is shown in Corollary \ref{Corollary_2} and the simulations, subsampling and bootstrap of the data set are equivalent.

We want of course to study the optimal subsampling size (\texttt{samplesize} parameter in the R algorithm). For this, we draw the curves of Figure \ref{fig5} for different learning data set sizes, the same as in Figure \ref{fig2}. We also copy in an other graph the optimal subsample size $a_n^{\star}$ that we found for each size of the learning set.  The optimal subsampling size $a_n^{\star}$ is defined as 
\begin{align*}
a_n^{\star} = \min \{a : |\hat{L}_{a} - \min_s \hat{L}_s| < 0.05 \times (\max_s \hat{L}_s - \min_s \hat{L}_s) \}
\end{align*}
where $\hat{L}_{s}$ is the risk of the forest with parameter $\texttt{sampsize} = s$. The results can be seen in Figure \ref{fig6}. The optimal subsampling size seems, once again, to be proportional to the sample size, as illustrated in the last sub-figure of Figure \ref{fig6}. For {\bf Model 1}, the optimal value $a_n^{\star}$ seems to be close to $0.8n$. The other models show a similar behaviour, as it can be seen in Figure \ref{fig7}.

\begin{figure}[h!!]
\begin{tabular}{cc}
\includegraphics[scale=0.3]{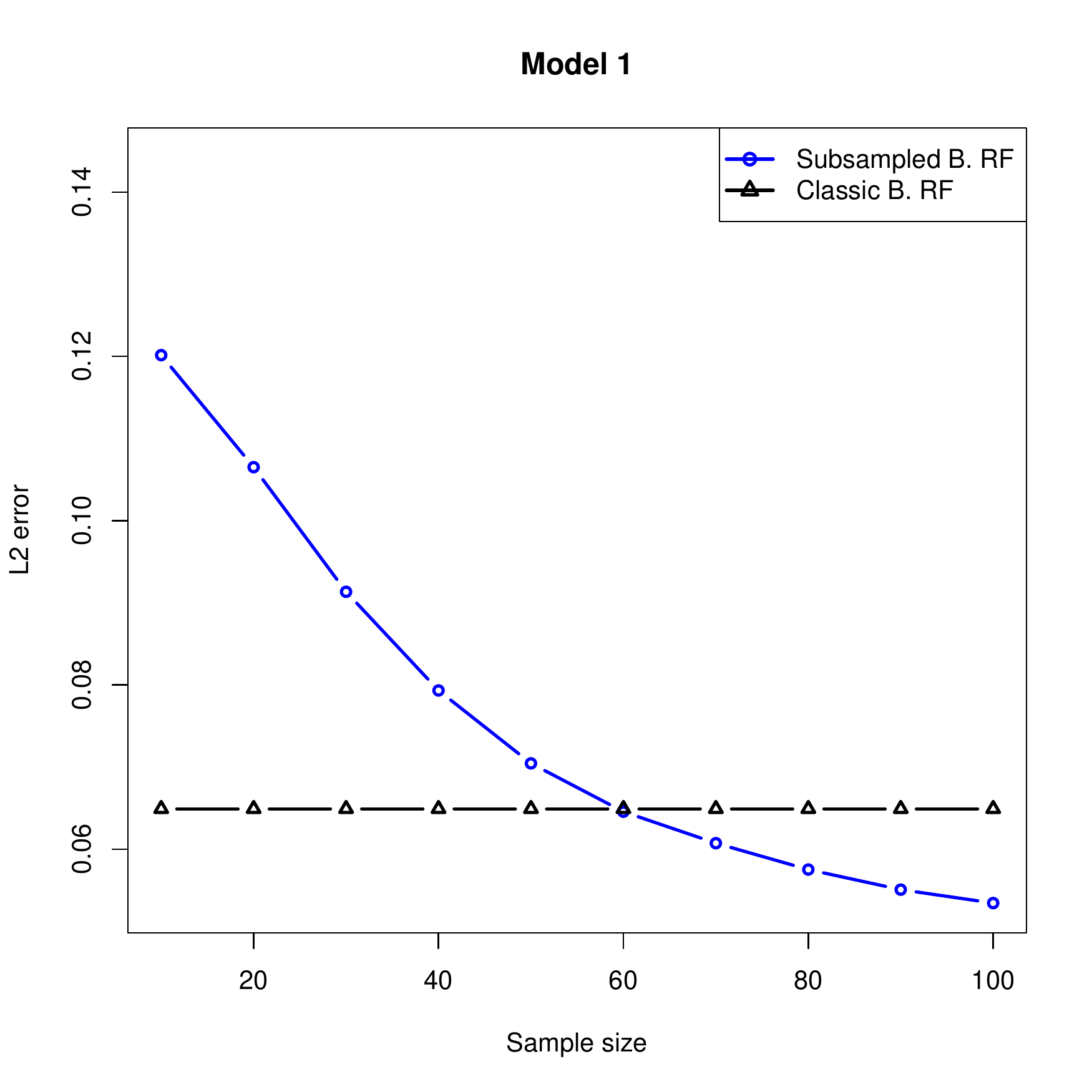}
& \includegraphics[scale=0.3]{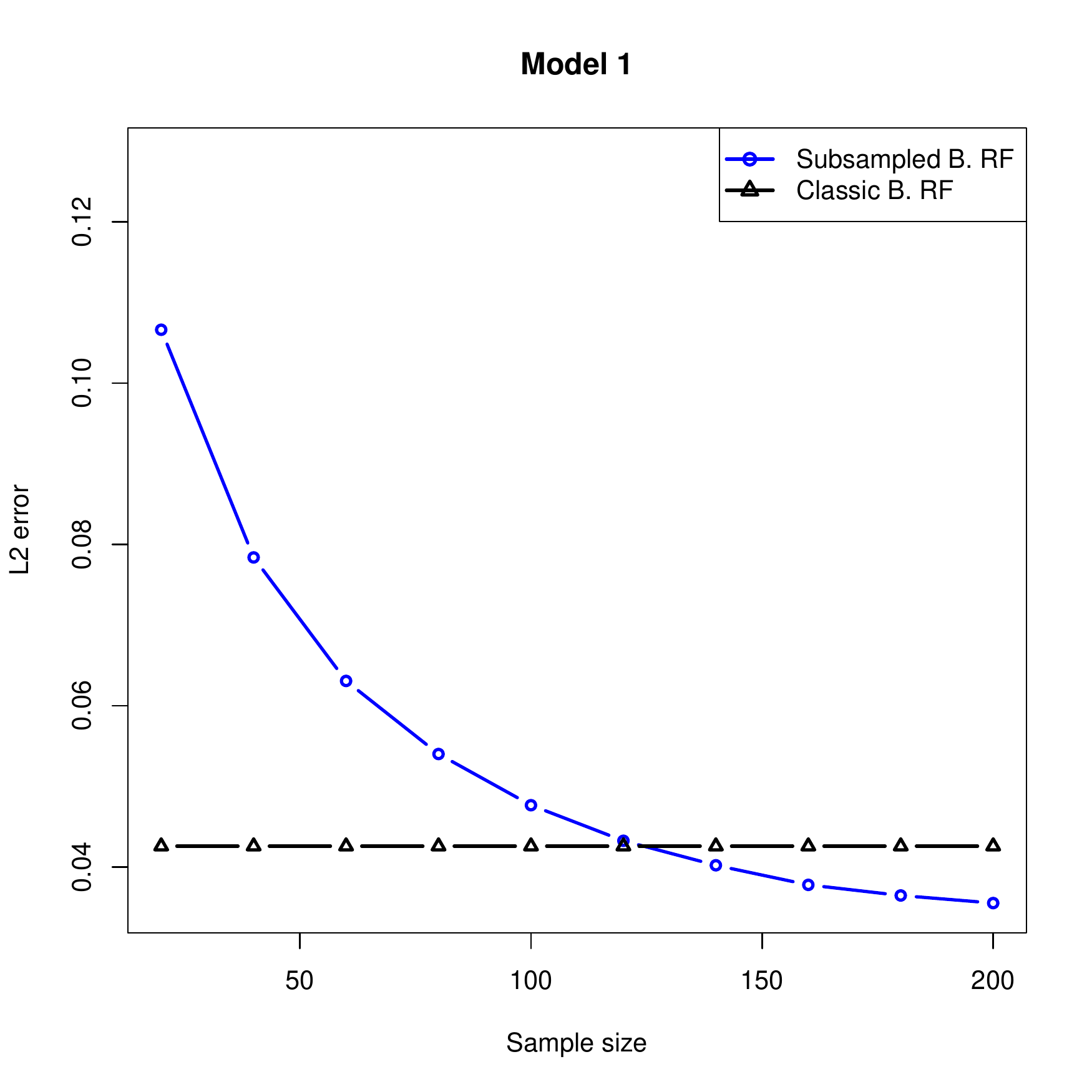}
\\
\includegraphics[scale=0.3]{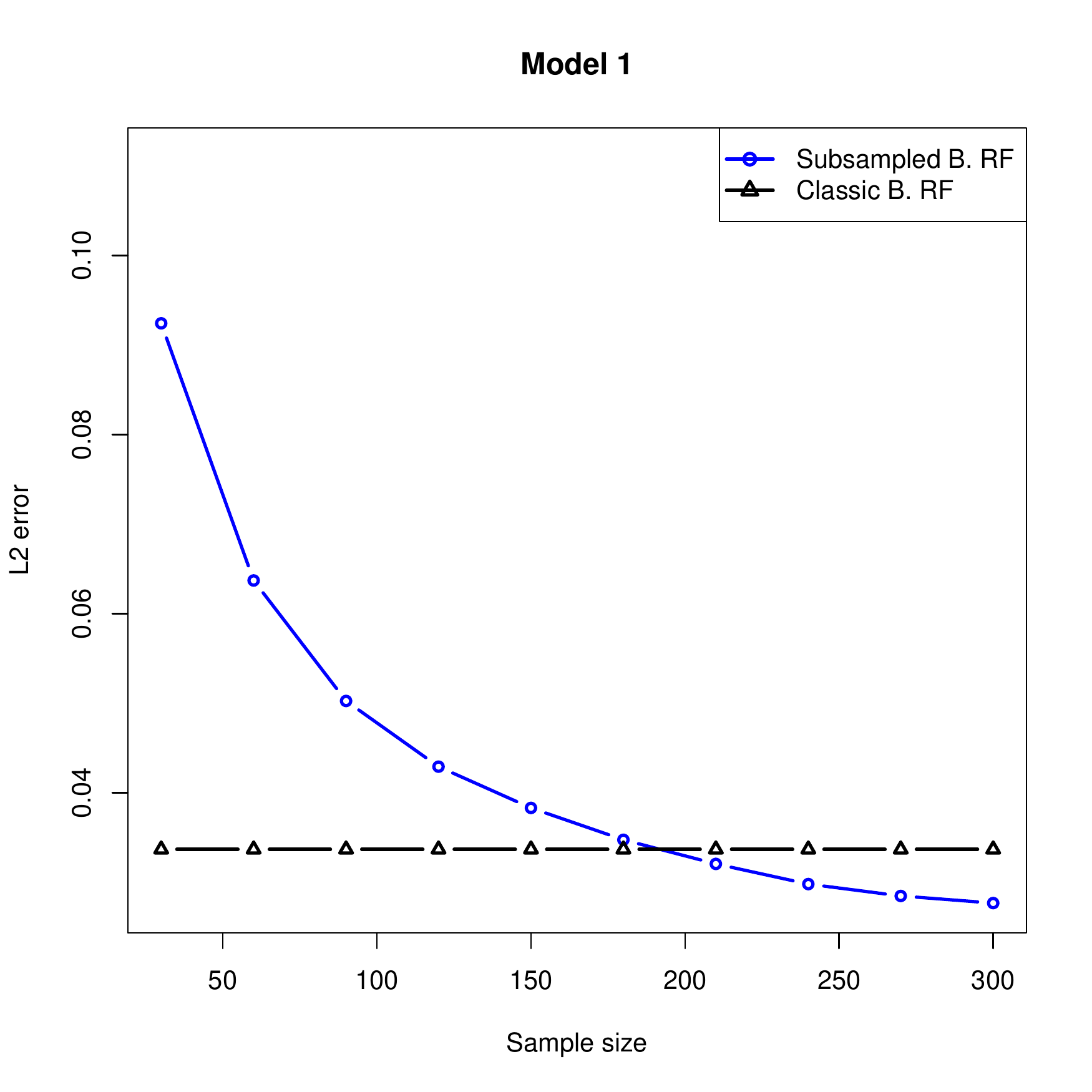}
& \includegraphics[scale=0.3]{choix_SSrate_model1_P_vs_B_RF-mod1_ntree500_nb-iter50nombrepoint400_mtrydefault.pdf} \\
\includegraphics[scale=0.3]{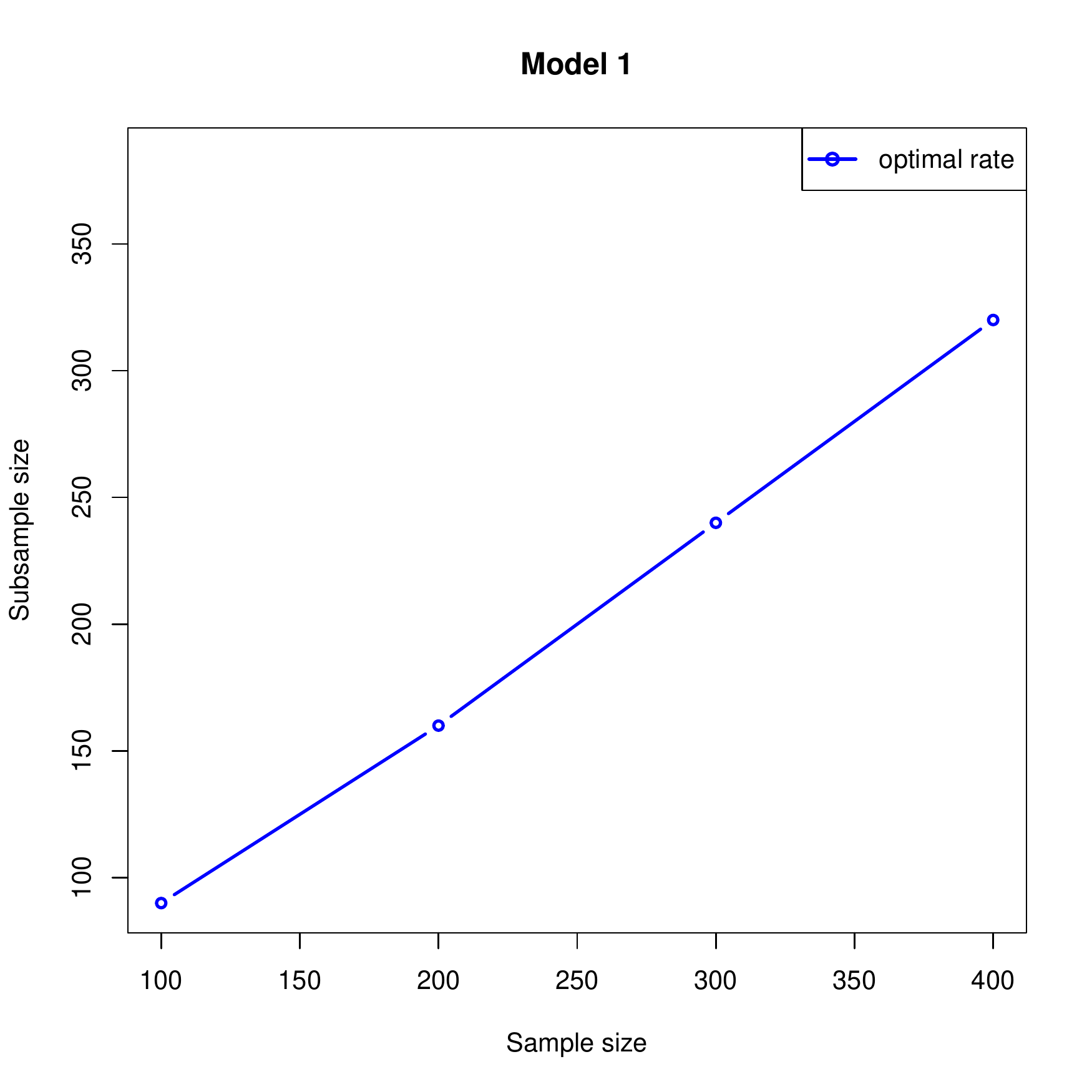}
\end{tabular}
\caption{Tuning of subsampling rate (model 1).}
\label{fig6}
\end{figure}

\begin{figure}[h!!]
\begin{tabular}{cc}
\includegraphics[scale=0.3]{choix_SSrate_model1_ss_optimal_ntree500_nb-iter50_mtrydefault.pdf}
& \includegraphics[scale=0.3]{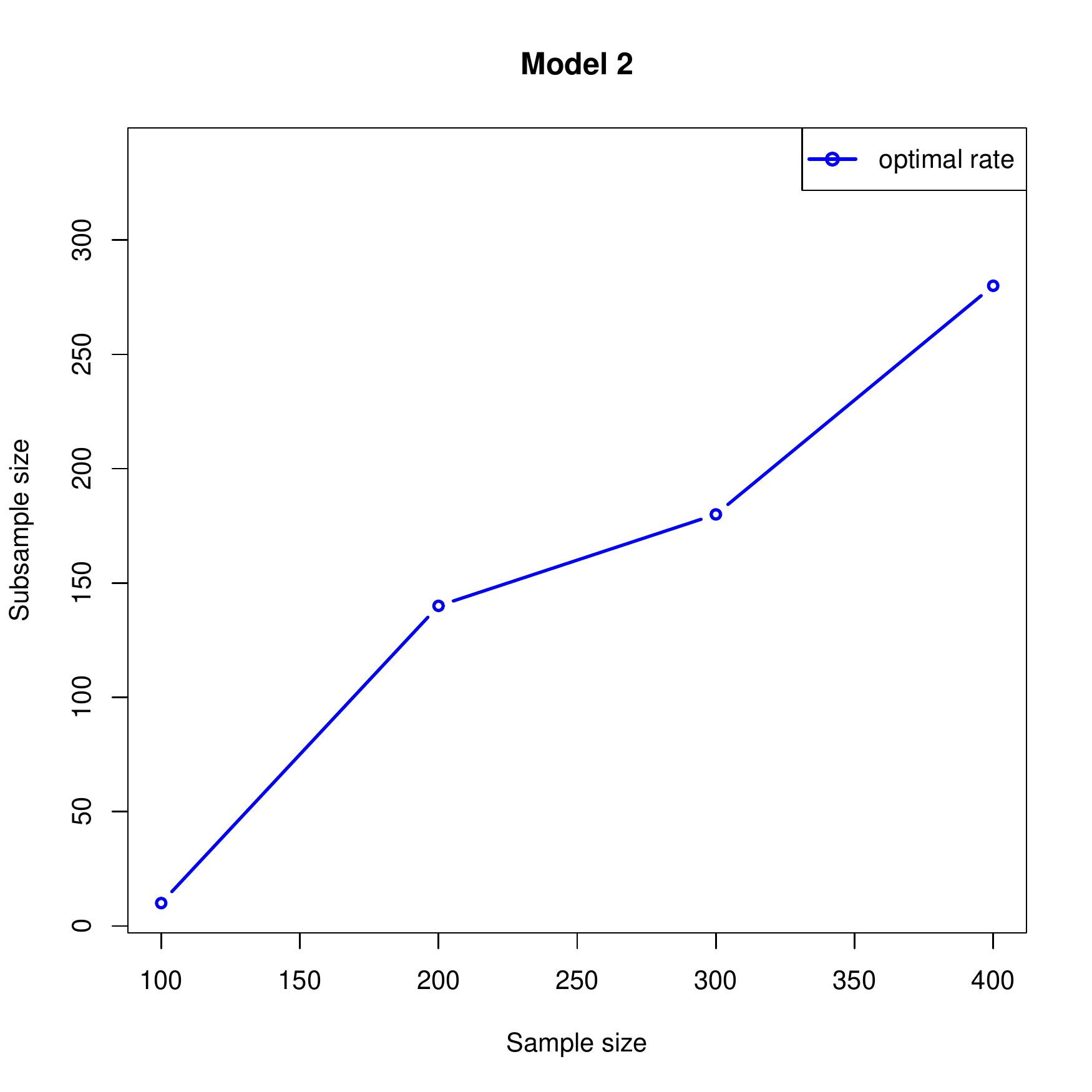}
\\
\includegraphics[scale=0.3]{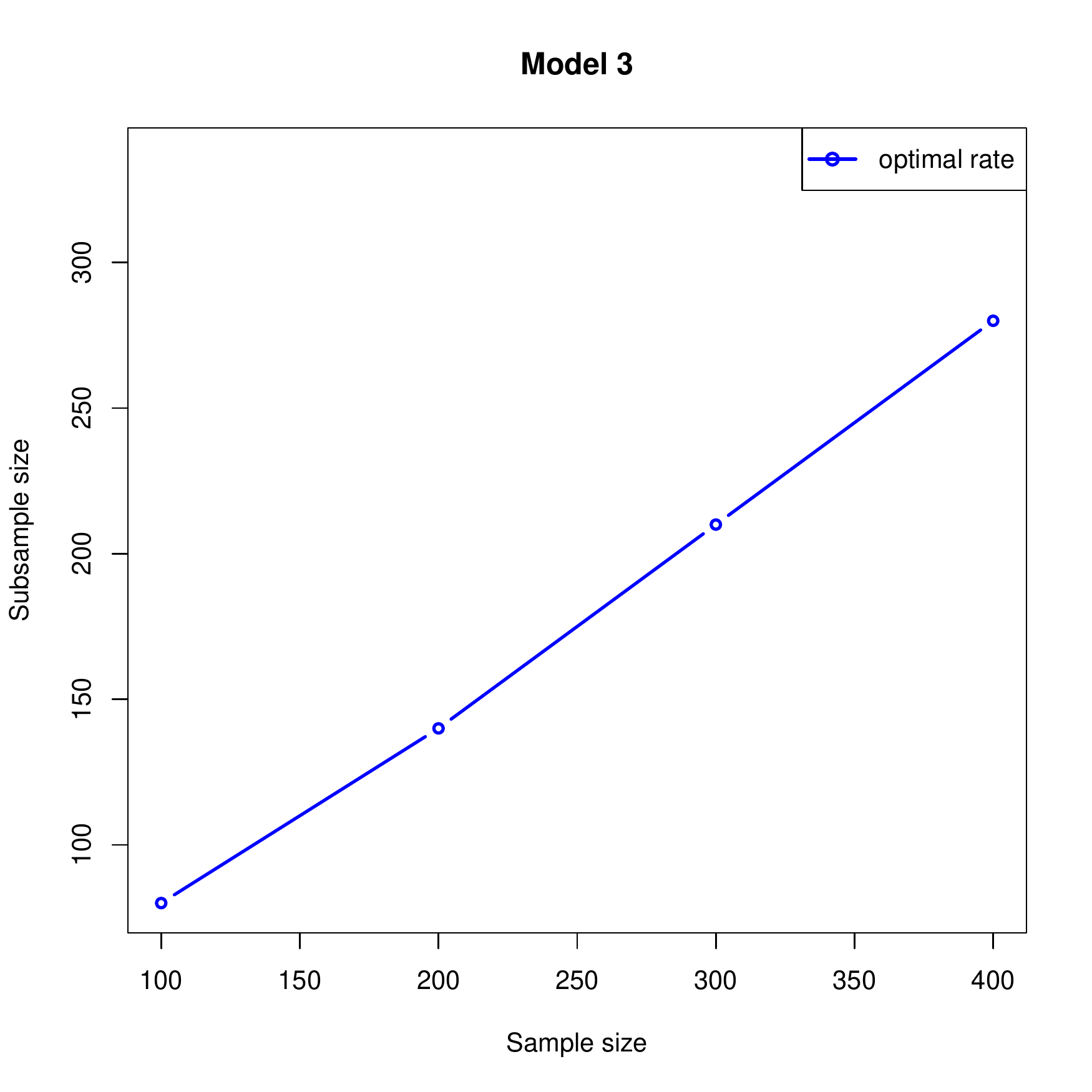}
& \includegraphics[scale=0.3]{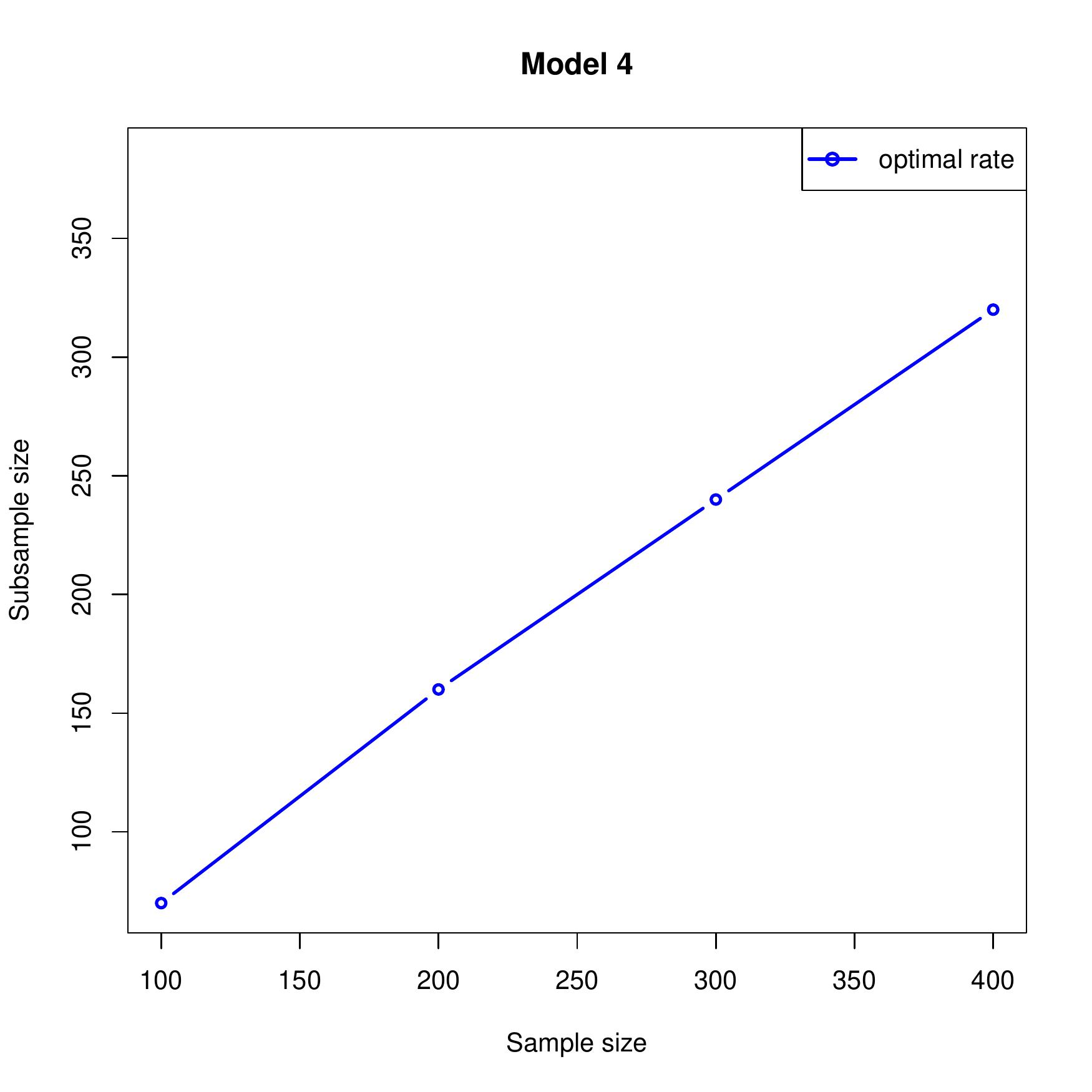}
\\
\includegraphics[scale=0.3]{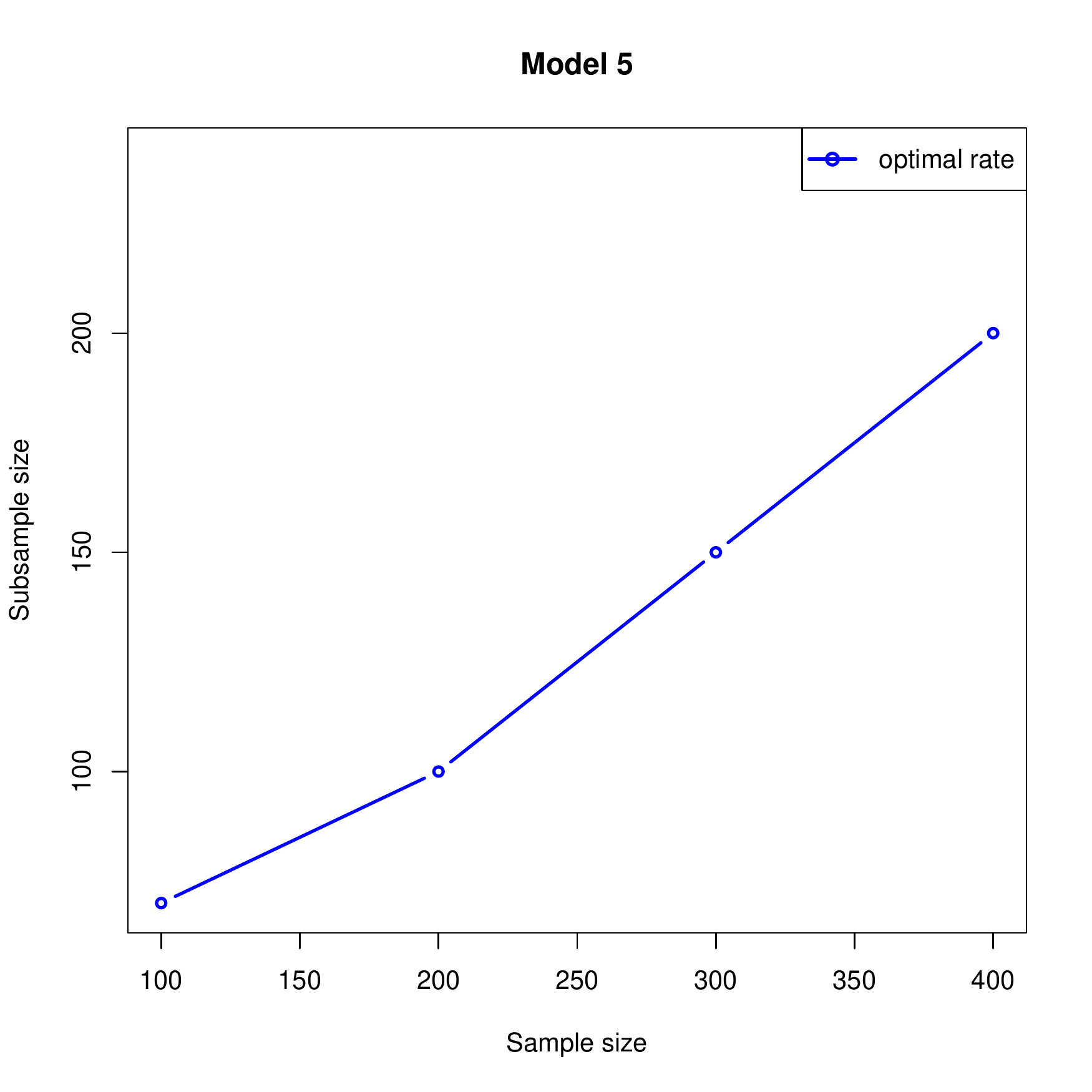}
& \includegraphics[scale=0.3]{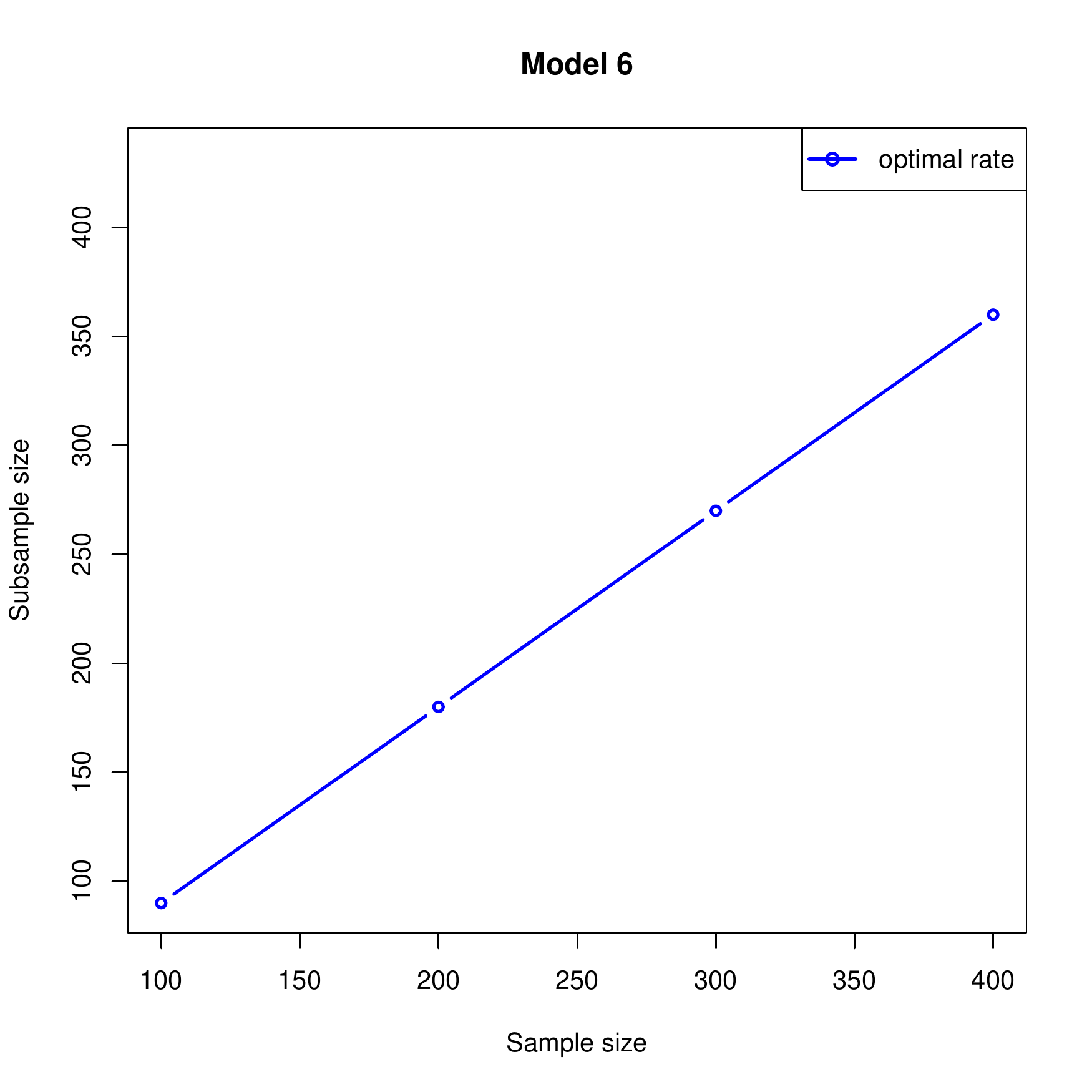}
\\
\includegraphics[scale=0.3]{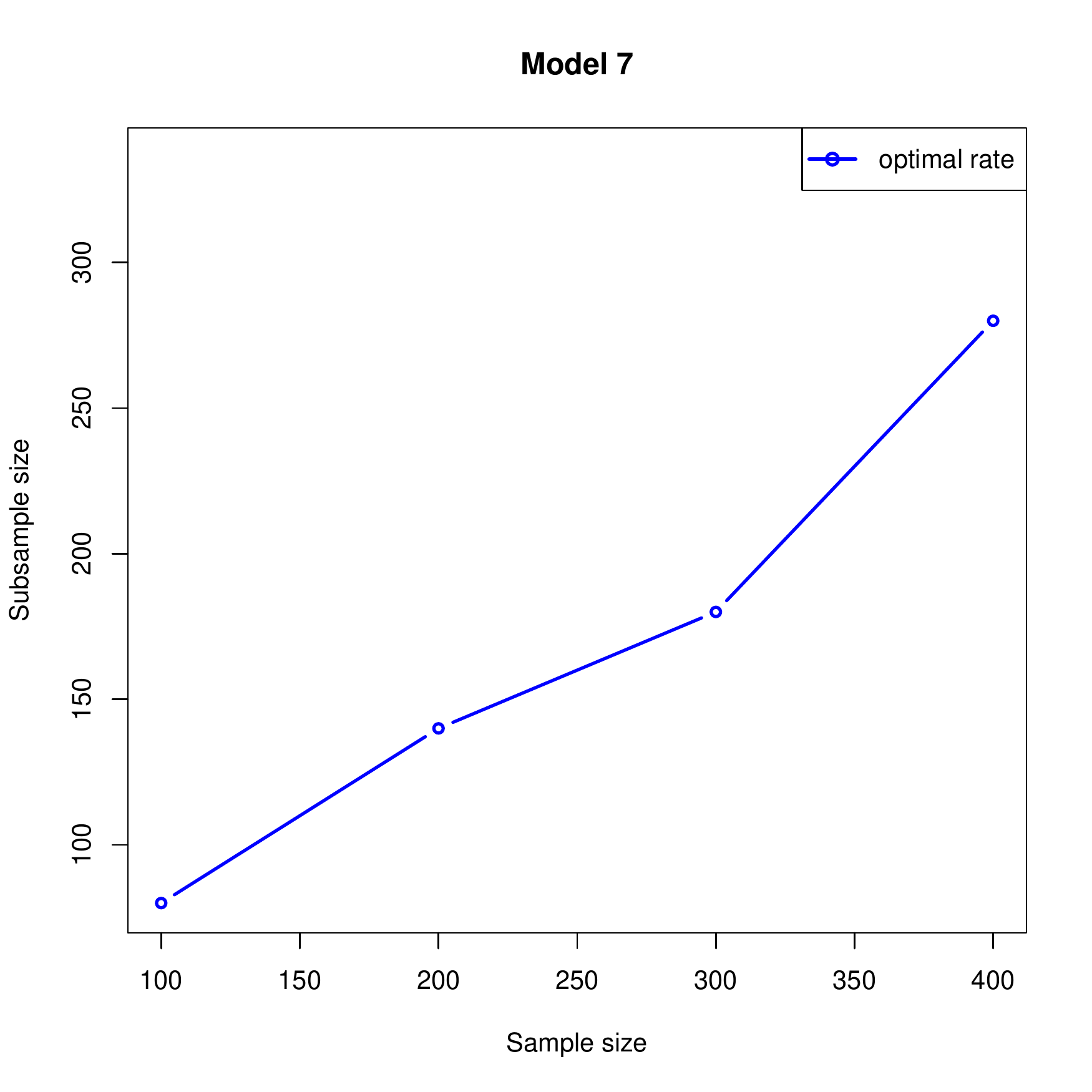}
& \includegraphics[scale=0.3]{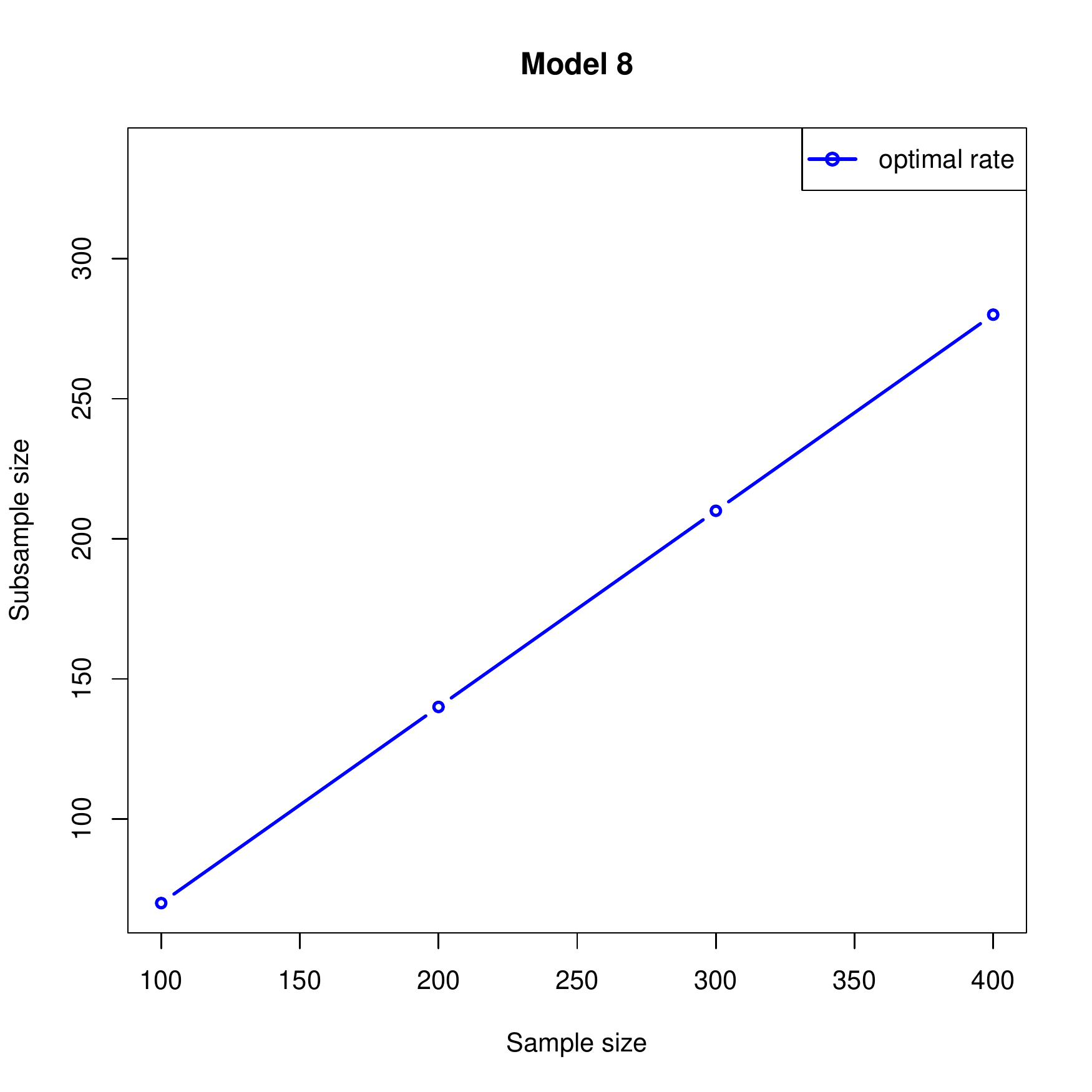}
\\
\end{tabular}
\caption{Optimal values of pruning parameter.}
\label{fig7}
\end{figure}

Then we present, in Figure \ref{fig8}, the $\mathbb{L}^2$ errors of subsampled Breiman's forests for different subsampling sizes ($0.4n$, $0.5n$, $0.63n$ and $0.9n$), when the sample size is fixed, for {\bf Models 1-8}. We can notice that the forests with a subsampling size of $0.63n$ give similar performances than the standard Breiman's forests. This is not surprising. Indeed, a bootstrap sample contains around $63\%$ of distinct observations. Moreover the high subsampling sizes, around $0.9n$, lead to small $\mathbb{L}^2$ errors. It may arise from the probably high signal/noise rate. In each model, when the noise is increasing, the results, exemplified in Figure \ref{fig9}, are less obvious. That is why we can lawfully use the subsampling size as an optimization parameter for the Breiman's forest performance.

\begin{figure}[h!!]
\begin{tabular}{cc}
\includegraphics[scale=0.3]{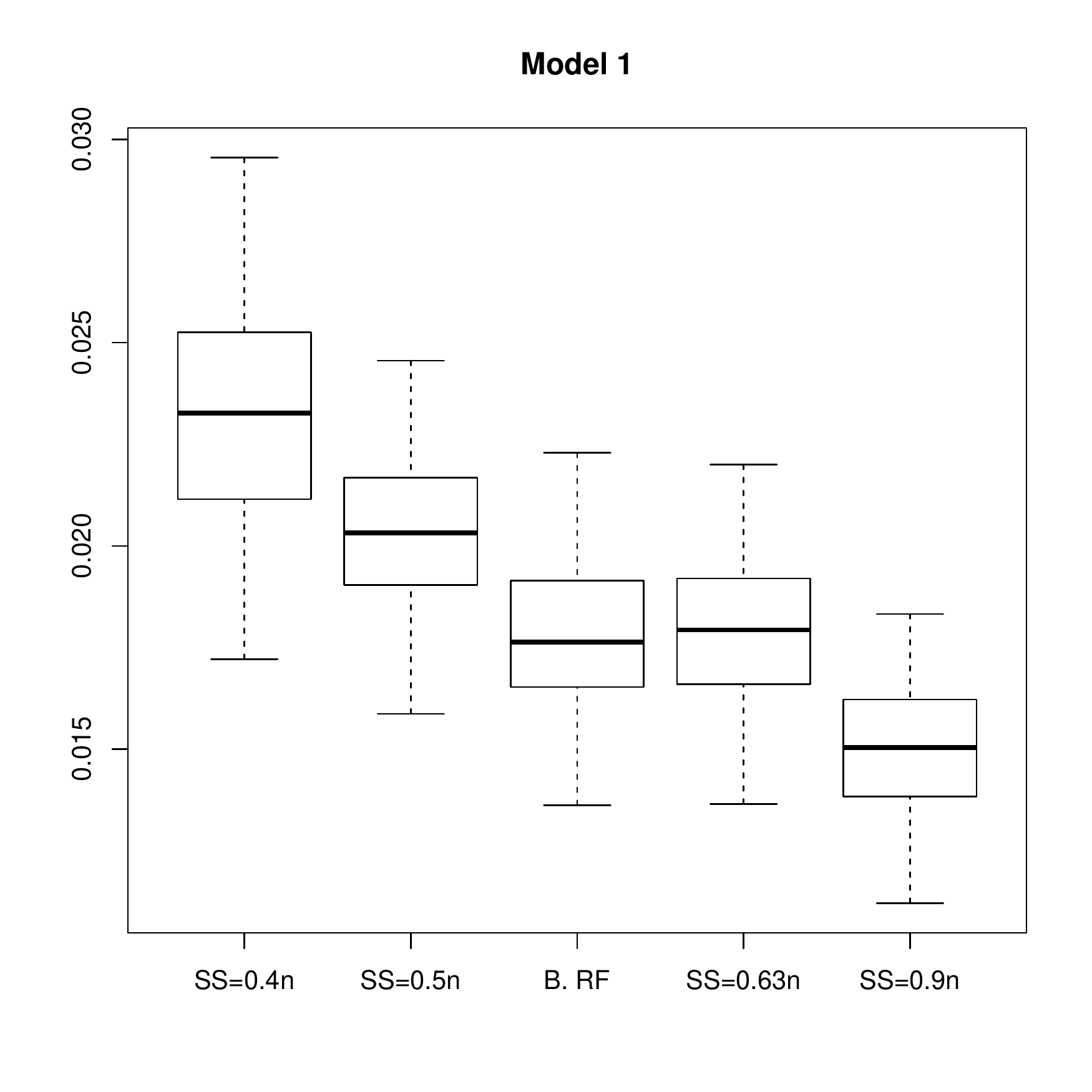}
& \includegraphics[scale=0.3]{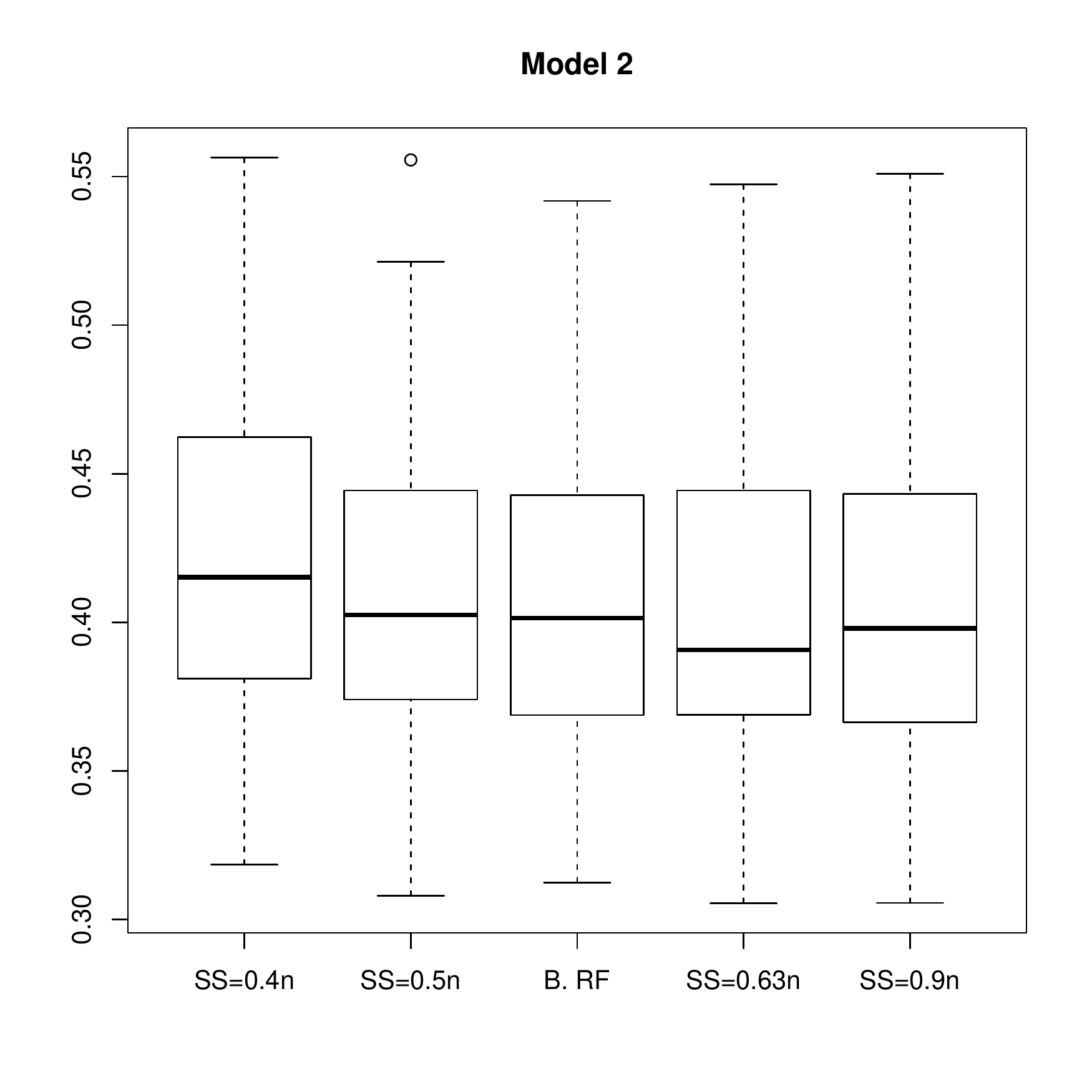}
\\
\includegraphics[scale=0.3]{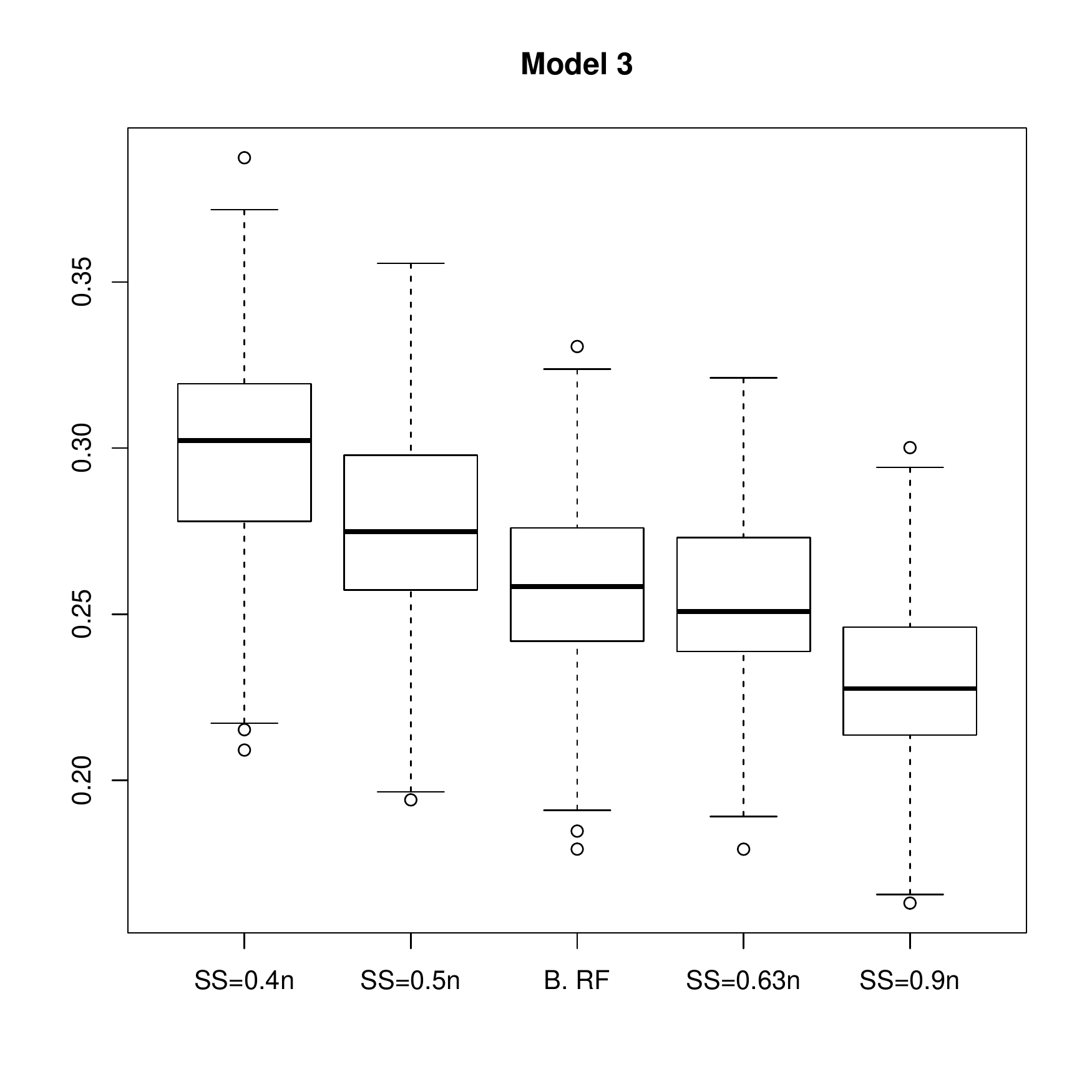}
& \includegraphics[scale=0.3]{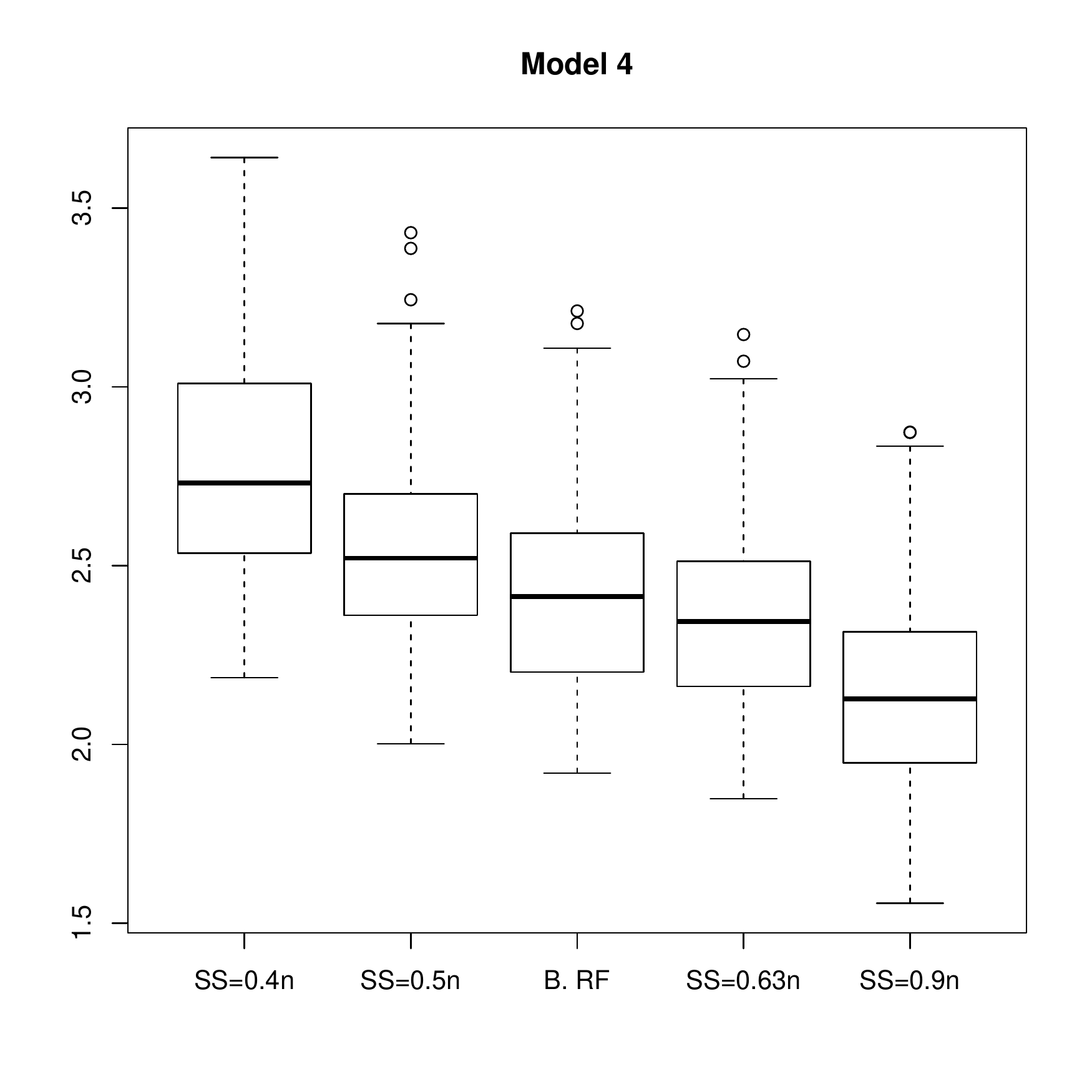}
\\
\includegraphics[scale=0.3]{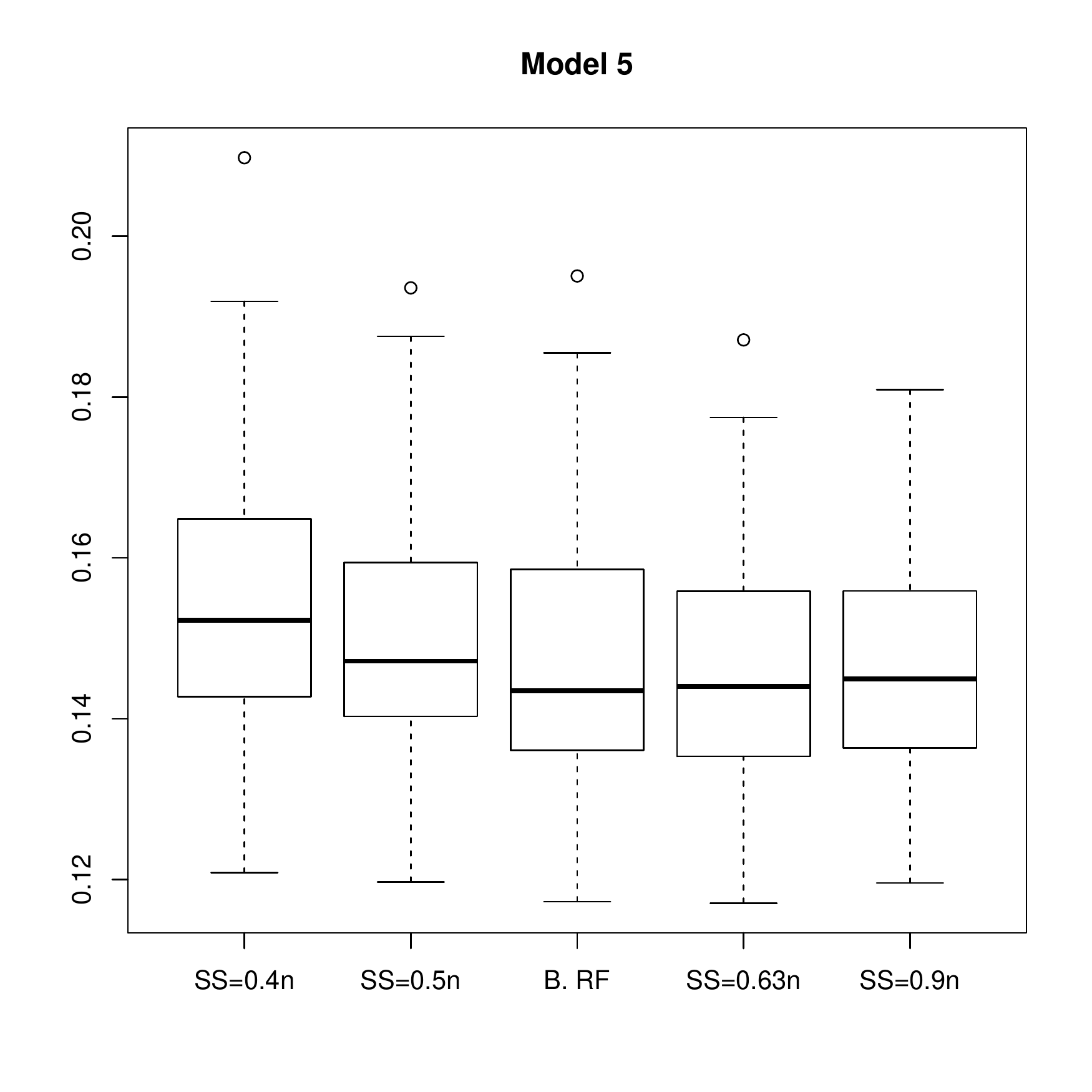}
& \includegraphics[scale=0.3]{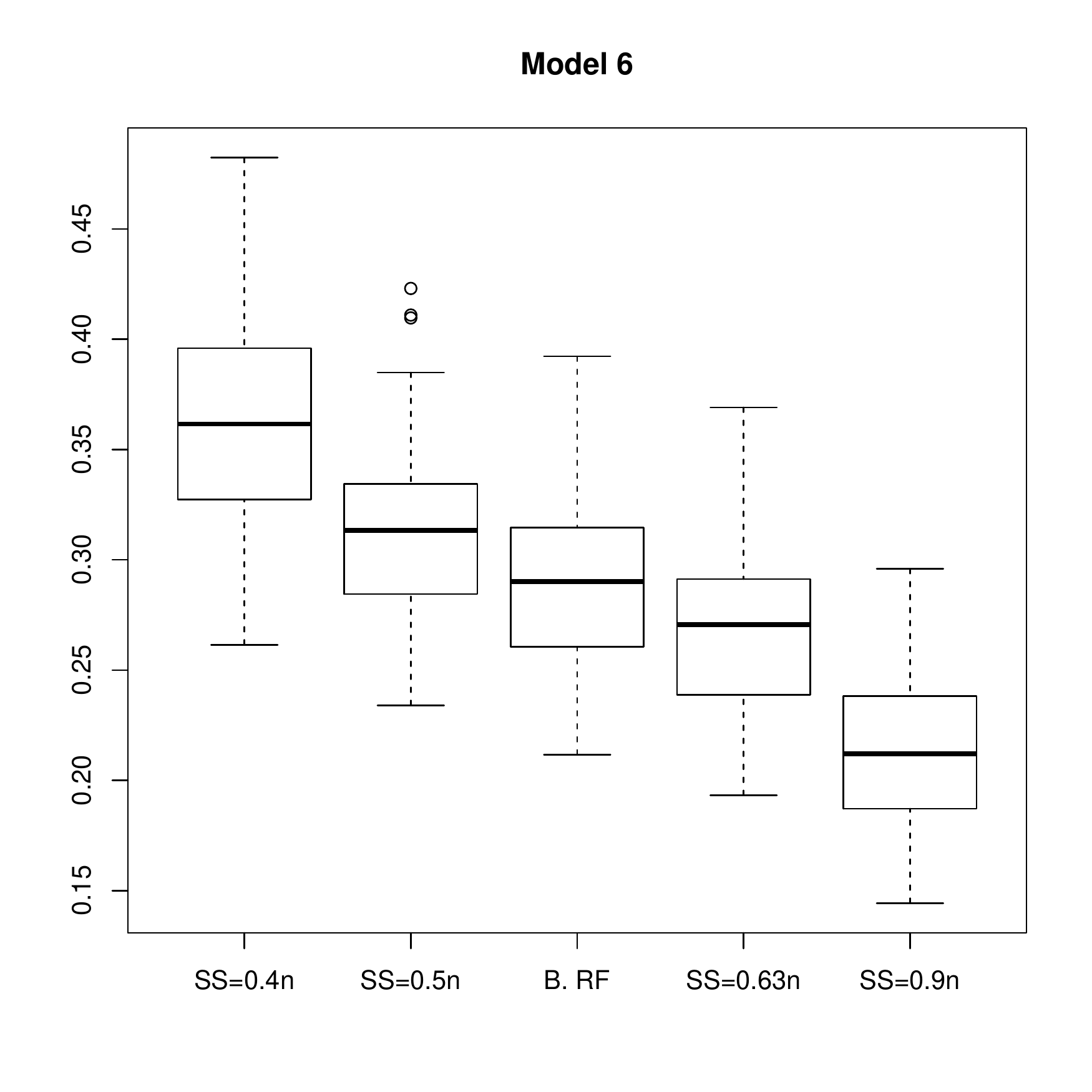}
\\
\includegraphics[scale=0.3]{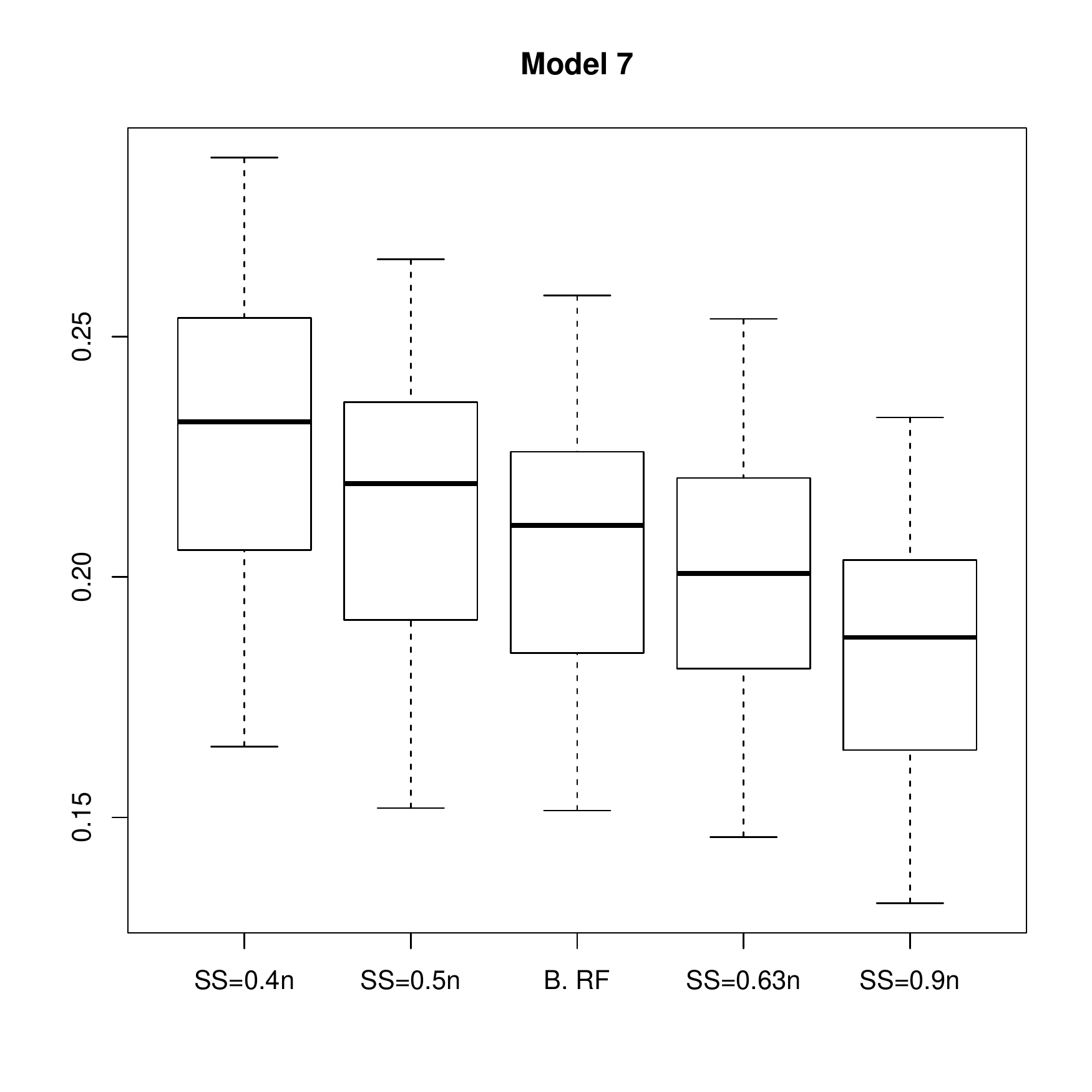}
& \includegraphics[scale=0.3]{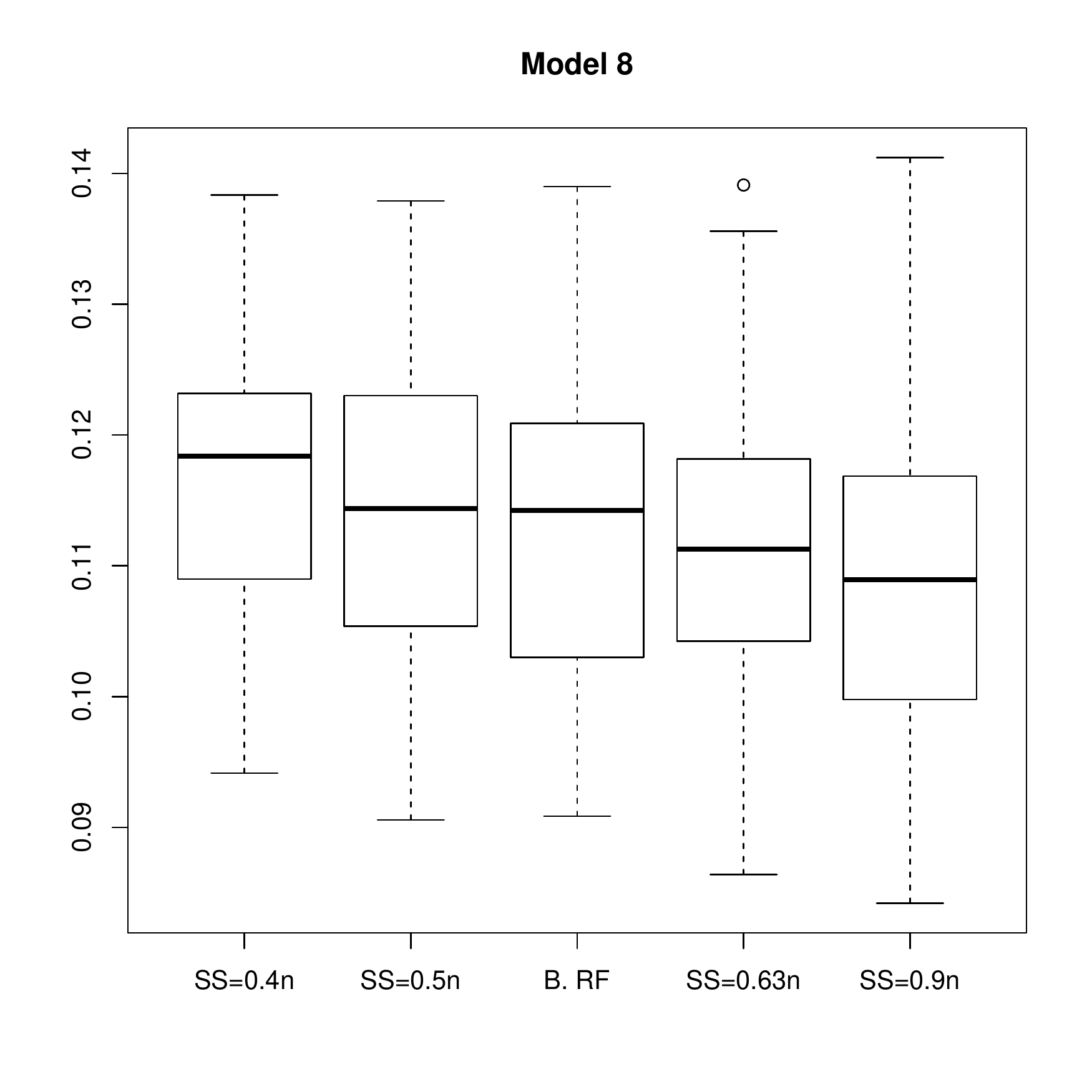}
\\
\end{tabular}
\caption{Standard Breiman forests versus several pruned Breiman forests.}
\label{fig8}
\end{figure}

\begin{figure}[h!!]
\begin{tabular}{cc}
\includegraphics[scale=0.3]{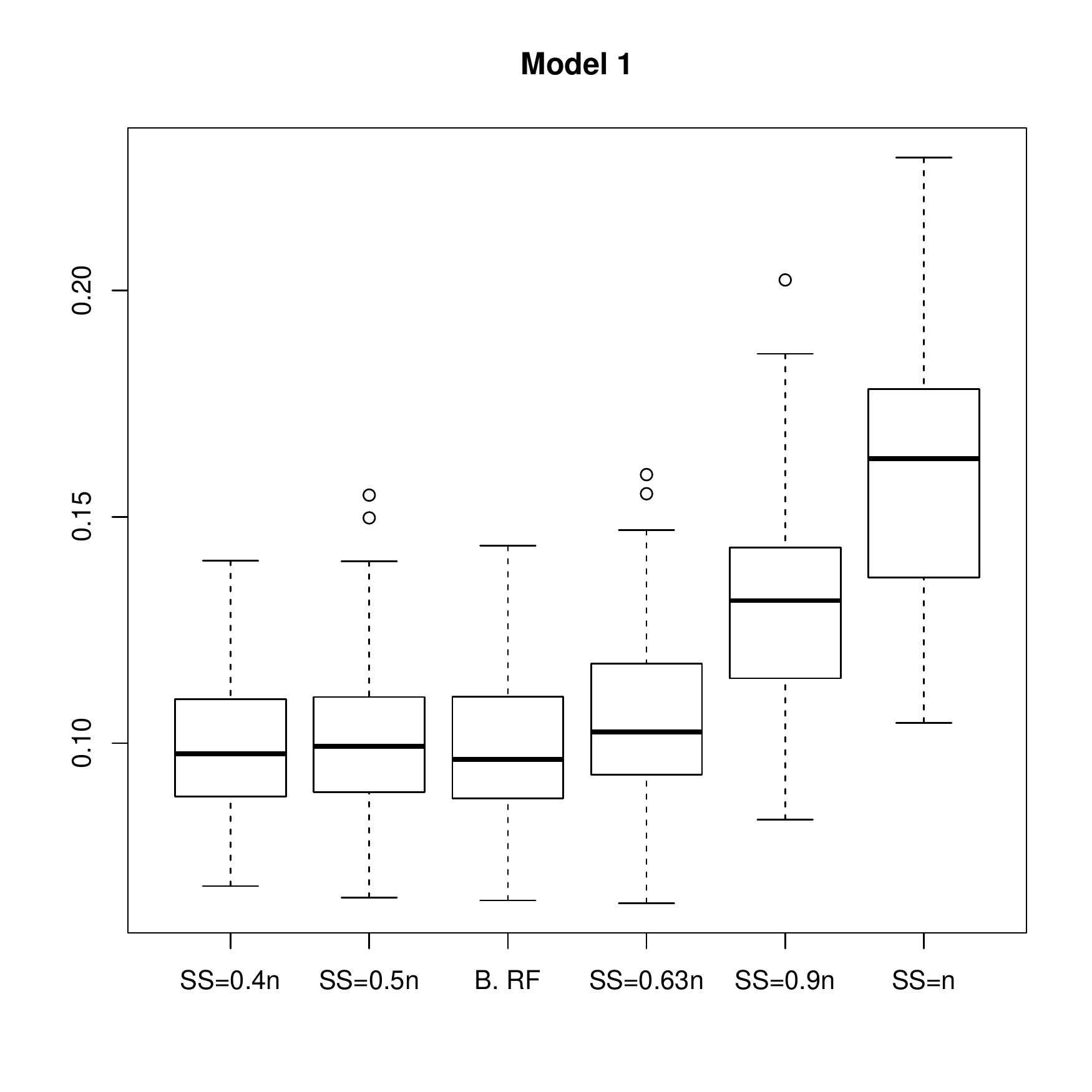}
& \includegraphics[scale=0.3]{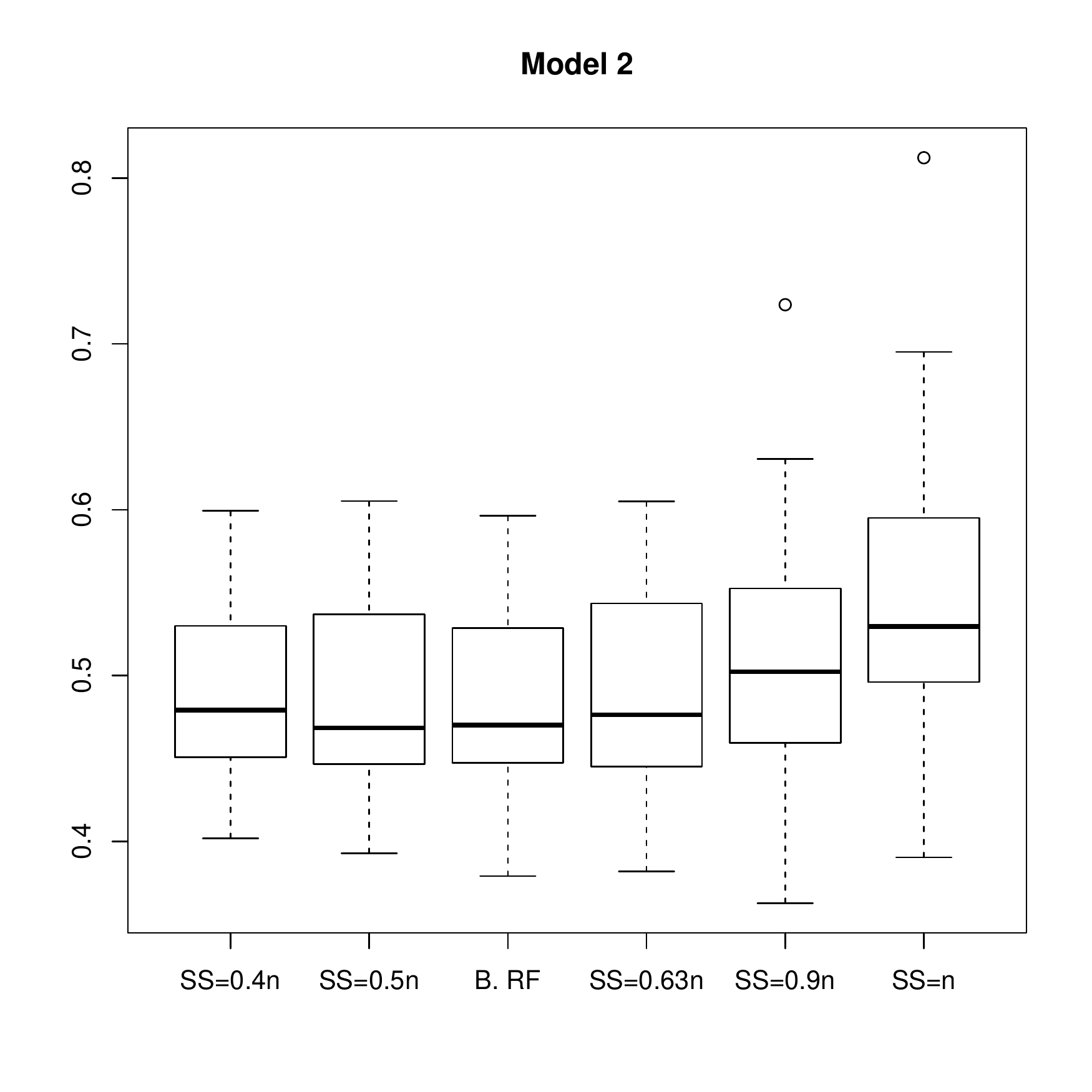}
\\
\includegraphics[scale=0.3]{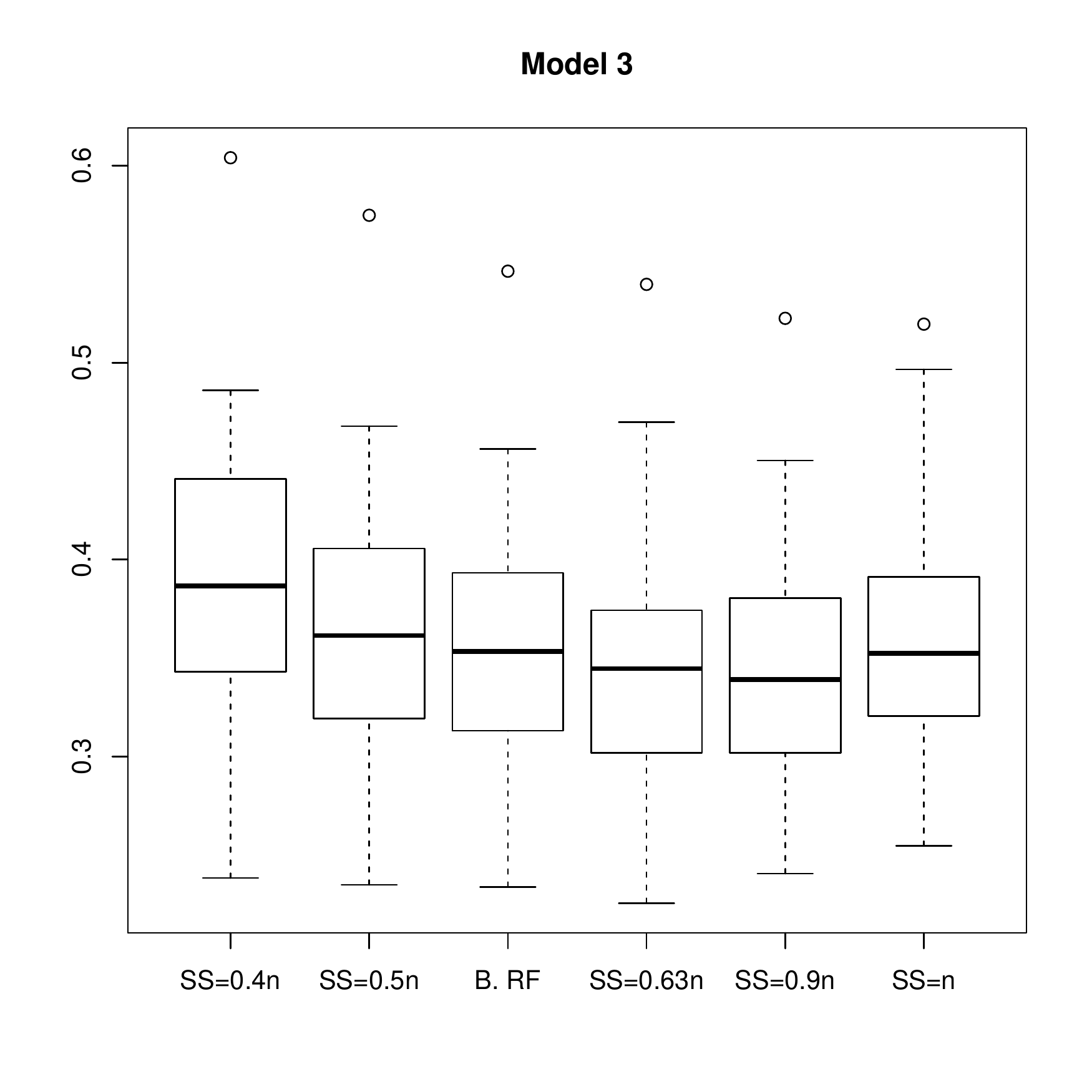}
& \includegraphics[scale=0.3]{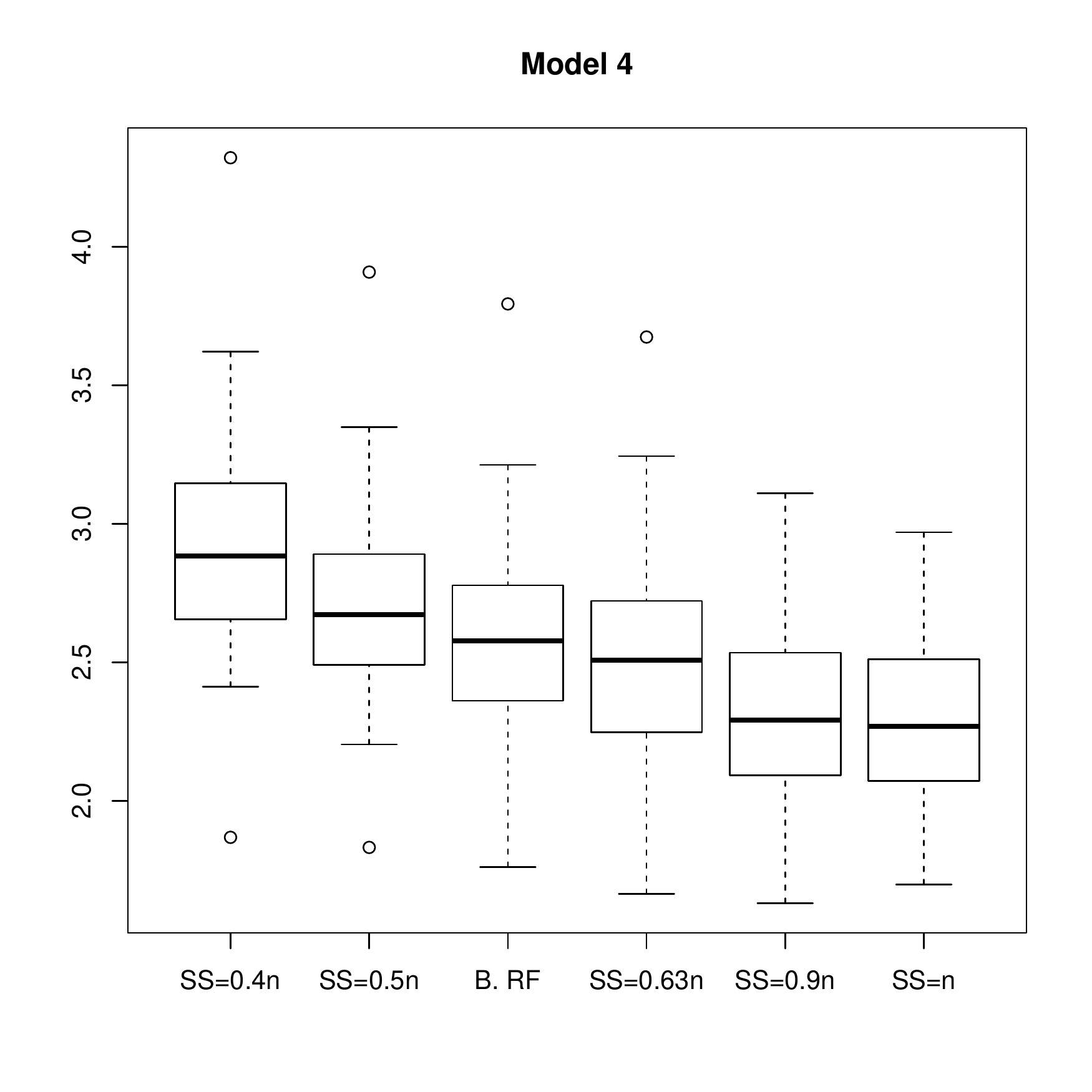}
\\
\includegraphics[scale=0.3]{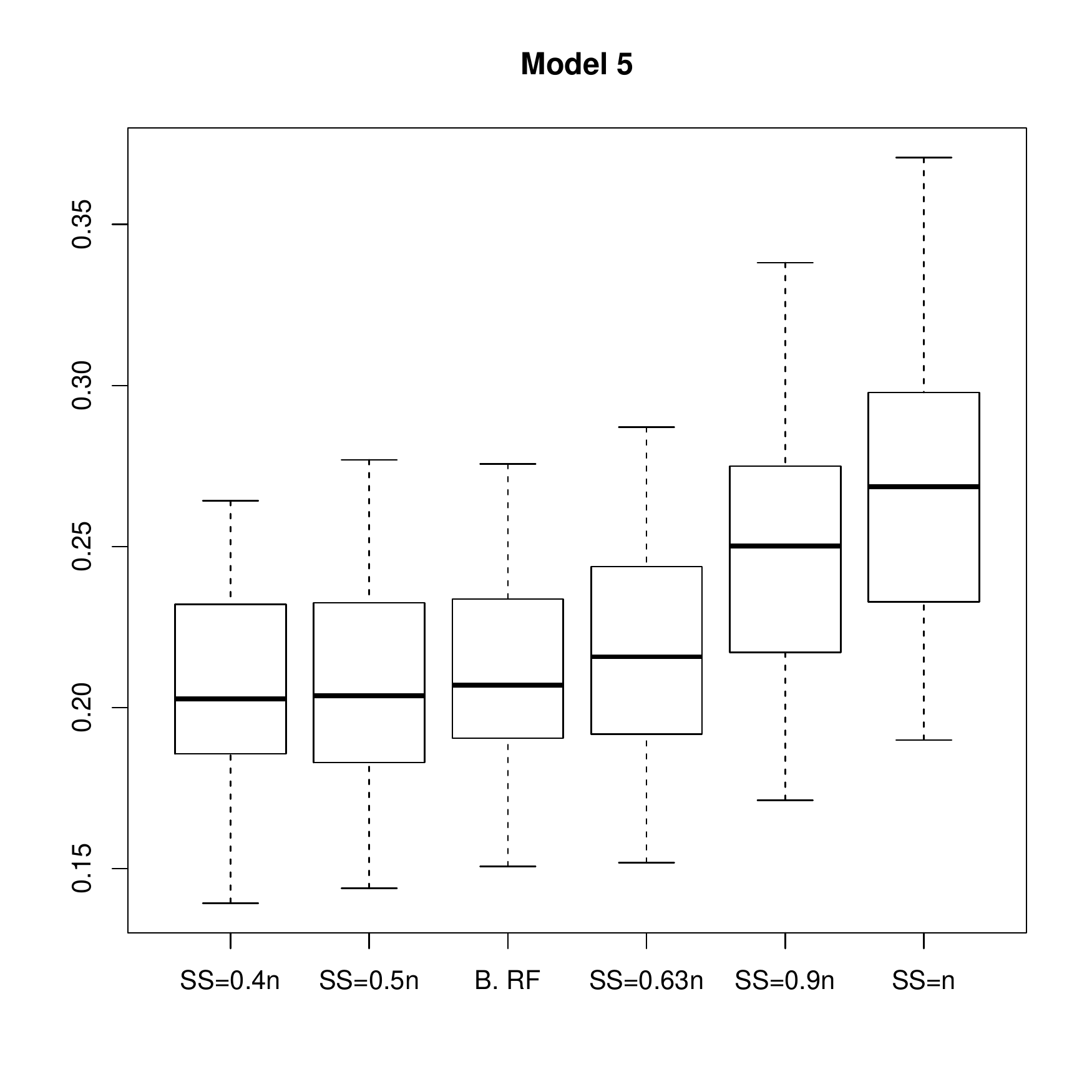}
& \includegraphics[scale=0.3]{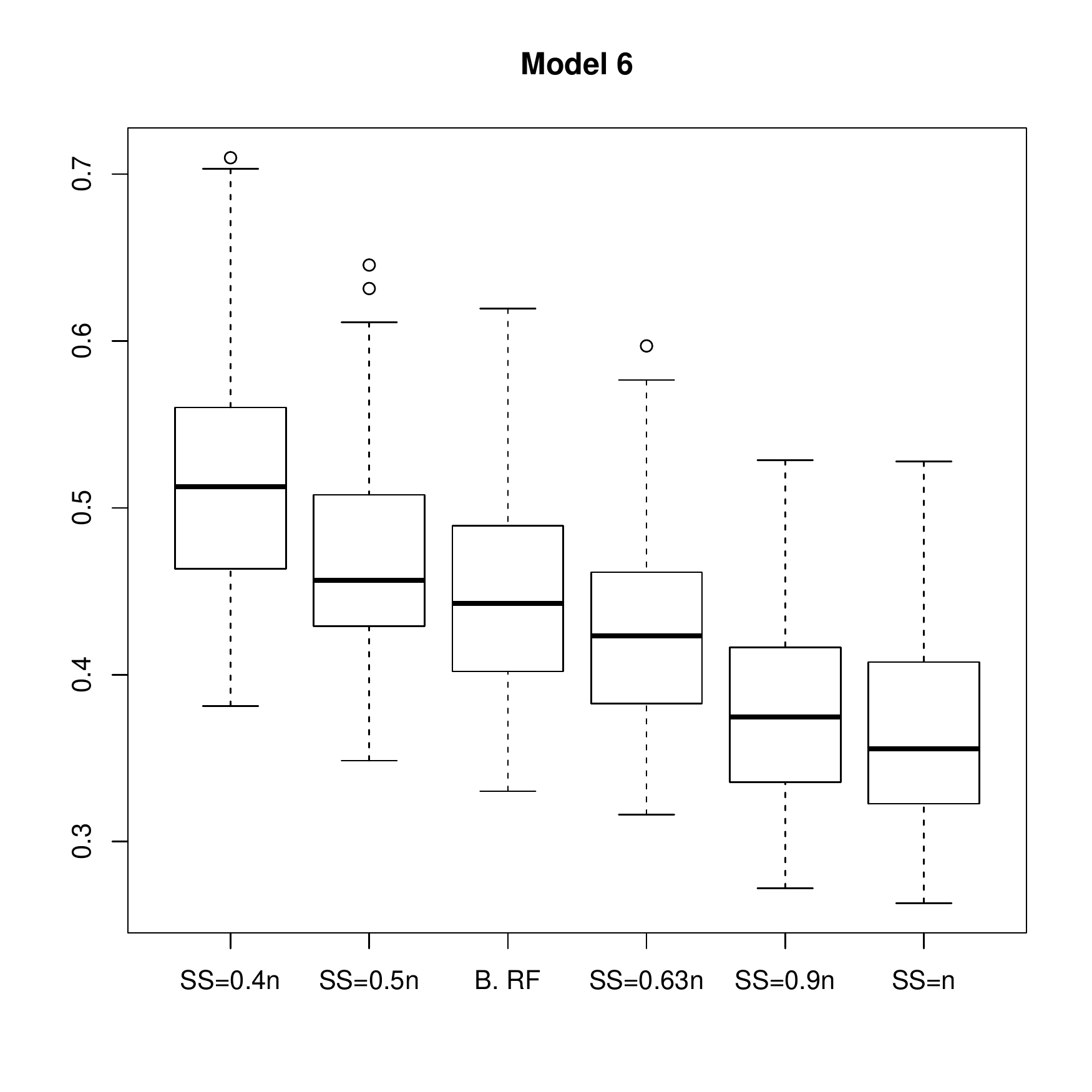}
\\
\includegraphics[scale=0.3]{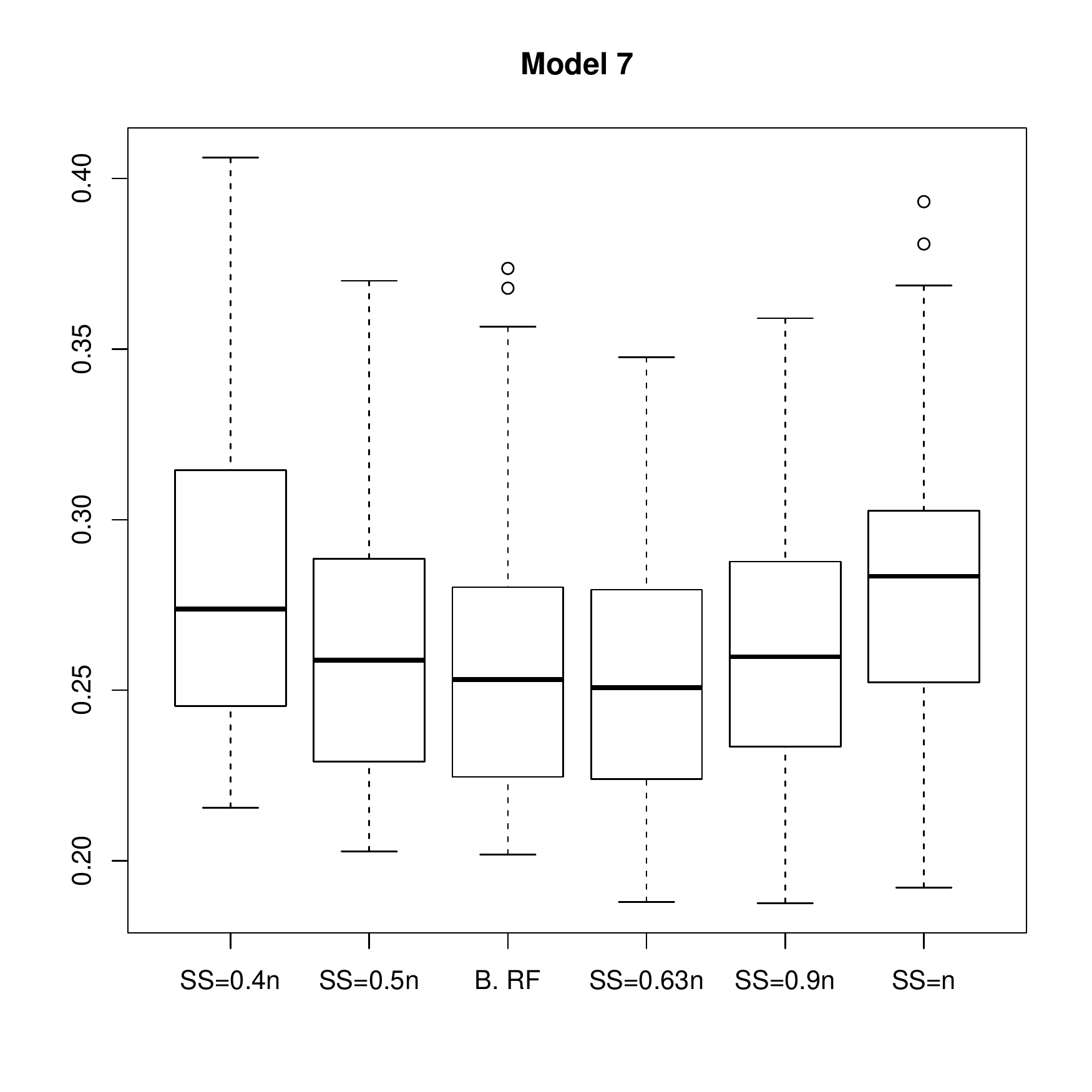}
& \includegraphics[scale=0.3]{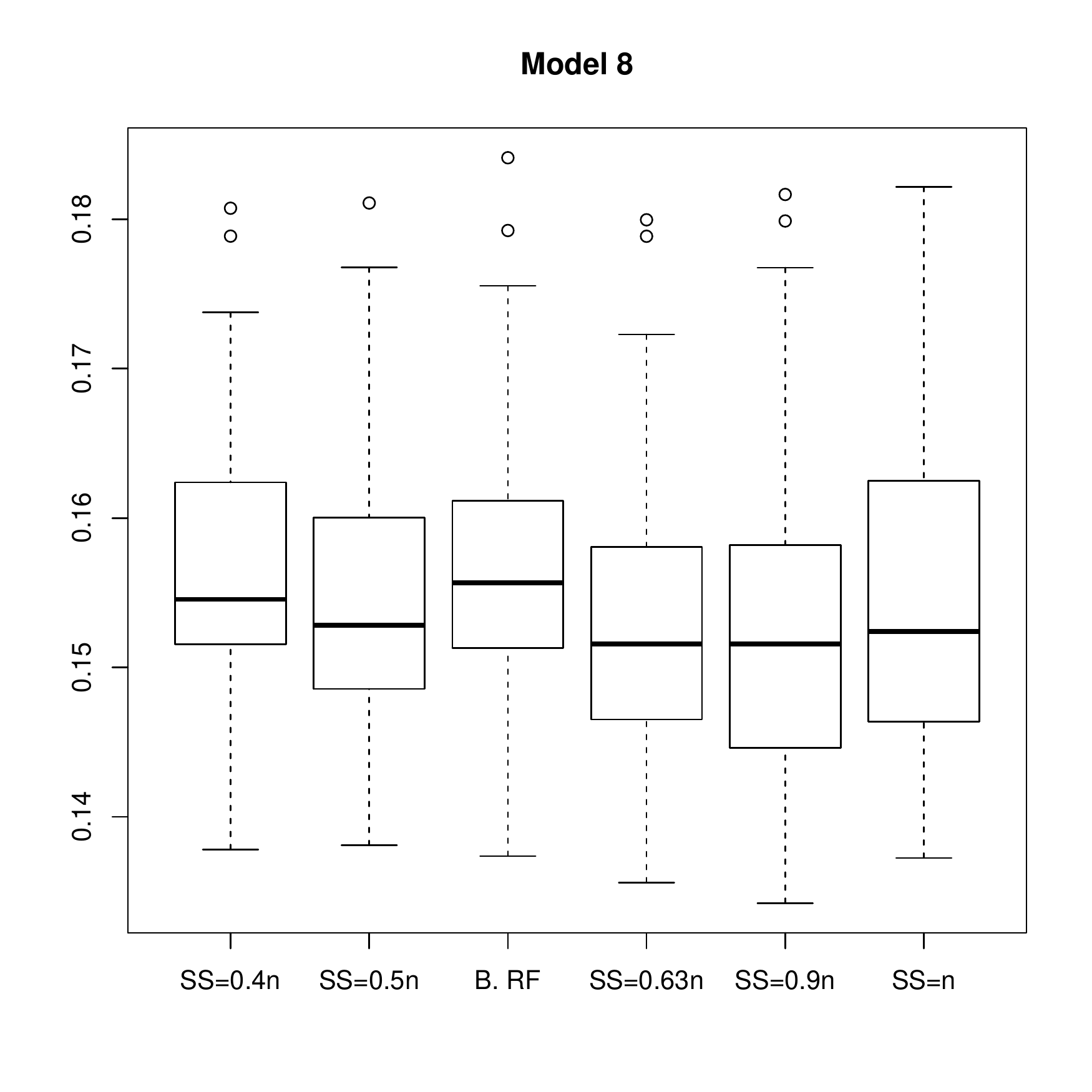}
\\
\end{tabular}
\caption{Standard Breiman forests versus several pruned Breiman forests (noisy models).}
\label{fig9}
\end{figure}

\section{Discussion}

In this paper, we studied the role of subsampling step and tree depth in Breiman's forest procedure. By analysing a simple version of random forests, we show that the performance of fully grown subsampled forests  and that of pruned forests with no subsampling step are similar, provided a proper tuning of the parameter of interest (subsample size and tree depth respectively). 

The extended experiments have shown similar results: Breiman's forests can be outperformed by either subsampled or pruned Breiman's forests by properly tuning parameters. Noteworthy, tuning tree depth can be done at almost no additional cost while running Breiman's forests (due to the intrinsic recursive nature of forests). However if one is interested in a faster procedure, subsampled Breiman's forests are to be preferred to pruned forests. 

As a by-product, our analysis also shows that there is no particular interest in bootstrapping data instead of subsampling: in our experiments, bootstrap is comparable (or worse) than subsampling. 
This sheds some light on several previous theoretical analysis where the bootstrap step was replaced by subsampling, which is more amenable to analyse. Similarly, proving theoretical results on fully grown Breiman's forests turned out to be extremely difficult. Our analysis shows that there is no theoretical background for considering Breiman's forests with default parameters values  instead of pruned or subsampled Breiman's forests, which reveal themselves to be easier to examine.

\clearpage

\section{Proofs}\label{S-preuves}

\begin{proof}[Proof of Theorem \ref{th_general_bound}]

Let us start by recalling that the random forest estimate $m_n$ can be written as a local averaging estimate
\begin{align*}
m_{\infty,n}(\bx) = \sum_{i=1}^n W_{ni}(\bx) Y_i,
\end{align*}
where 
\begin{align*}
W_{ni}(\bx) = \frac{\mathds{1}_{\bX_i \overset{\Theta}{\leftrightarrow} \bx}}{N_n(\bx, \T)}.
\end{align*}
The quantity $\mathds{1}_{\bX_i \overset{\Theta}{\leftrightarrow} \bx}$ indicates whether the observation $\bX_i$ is in the cell of the tree which contains $\bx$ or not, and $N_n(\bx, \T)$ denotes the number of data points falling in the same cell as $\bx$. The $\mathds{L}^2$-error of the forest estimate takes then the form
\begin{align*}
\E \big[ m_{\infty,n}(\bx) - m(\bx) \big]^2 & \leq 2 \E \bigg[ \sum_{i=1}^n W_{ni}(\bx) (Y_i - m(\bX_i)) \bigg]^2 \\
& \quad + 2 \E \bigg[ \sum_{i=1}^n W_{ni}(\bx) (m(\bX_i) - m(\bx)) \bigg]^2\\
 & = 2I_n + 2 J_n.
\end{align*}
We can identify the term $I_n$ as the estimation error and $J_n$ as the approximation error, and then work on each term $I_n$ and $J_n$ separately. 

\paragraph{Approximation error.}
Let $A_n(\bx, \Theta)$ be the cell containing $\bx$ in the tree built with the random parameter $\Theta$. Regarding $J_n$, by the Cauchy Schwartz inequality, 
\begin{align*}
J_n & \leq \E \bigg[ \sum_{i=1}^n \sqrt{W_{ni}(\bx)} \sqrt{W_{ni}(\bx)} |m(\bX_i) - m(\bx)|  \bigg]^2\\
& \leq \E \bigg[ \sum_{i=1}^n W_{ni}(\bx) (m(\bX_i) - m(\bx))^2  \bigg]\\
& \leq \E \left[ \sum_{i=1}^n \frac{\mathds{1}_{\bX_i \overset{\Theta}{\leftrightarrow} \bx}}{N_n(\bx, \T)} 
\sup\limits_{\substack{\bx, \bz,\\ |\bx - \bz| \leq \textrm{diam}(A_n(\bx))}}  |m(\bx) - m(\bz)|^2  \right]\\
& \leq L^2 \E \left[ \frac{1}{N_n(\bx, \T)} \sum_{i=1}^n \mathds{1}_{\bX_i \overset{\Theta}{\leftrightarrow} \bx} \left( \textrm{diam}(A_n(\bx, \Theta)) \right)^2  \right]\\
& \leq L^2 \E \bigg[ \left( \textrm{diam}(A_n(\bx, \Theta)) \right)^2  \bigg],
\end{align*}
where the fourth inequality is due to the $L$-Lipschitz continuity of $m$. Let $V_{\ell}(\bx, \Theta)$ be the length of the cell containing $\bx$ along the $\ell$-th side. Then, 
\begin{align*}
J_n & \leq L^2 \sum_{l=1}^d \E \bigg[ V_l(\bx, \T)^2  \bigg].
\end{align*}
According to Lemma \ref{Lemme_cell_length_beta} specified further, we have
\[
\E \bigg[ V_l(\bx, \T)^2 \bigg] \leq C \bigg(1- \frac{3}{4d} \bigg)^k,
\]
with $C= \exp (12/(4d-3))$. Thus, for all $k$, we have 
\begin{align*}
J_n \leq d L^2C \bigg(1- \frac{3}{4d} \bigg)^k.
\end{align*}

\paragraph{Estimation error.} Let us now focusing on the term $I_n$, we have
\begin{align*}
I_n & = \E \bigg[ \sum_{i=1}^n W_{ni}(\bx) (Y_i - m(\bX_i)) \bigg]^2\\
& =  \sum_{i=1}^n \sum_{i=1}^n \E \bigg[ W_{ni}(\bx) W_{nj}(\bx) (Y_i - m(\bX_i)) (Y_j - m(\bX_j))\bigg]\\
& = \E \bigg[ \sum_{i=1}^n W_{ni}^2(\bx) (Y_i - m(\bX_i))^2 \bigg]\\
& \leq \sigma^2 \E \bigg[ \max_{1 \leq i \leq n } W_{ni}(\bx) \bigg],
\end{align*}
since,  by {\bf(H)}, the variance of $\varepsilon_i$ is bounded above by $\sigma^2$.
Recalling that $a_n$ is the number of subsampled observations used to build the tree, we can note that 
\begin{align*}
\E \bigg[ \max_{1 \leq i \leq n } W_{ni}(\bx) \bigg]  & = \E \bigg[  \max_{1 \leq i \leq n } \E_{\Theta} \bigg[ \frac{\mathds{1}_{\bx \overset{\Theta}{\leftrightarrow} \bX_i}}{N_n(\bx, \Theta)} \bigg] \bigg]\\
& \leq  \frac{1}{\frac{a_n}{2^k} - 2}\E \bigg[ \max_{1 \leq i \leq n } \P_{\Theta} \bigg[ \bx \overset{\Theta}{\leftrightarrow} \bX_i \bigg] \bigg].
\end{align*}
Observe that in the subsampling step, there are exactly $\binom{a_n-1}{n-1}$ choices to pick a fixed observation $\bX_i$. Since $\bx$ and $\bX_i$ belong to the same cell only if $\bX_i$ is selected in the subsampling step, we see that
\begin{align*}
 \mathds{P}_{\Theta}\left[ \bx  \overset{\Theta}{\leftrightarrow} \bX_i \right] 
\leq & \frac{\binom{a_n-1}{n-1}}{\binom{a_n}{n}} = \frac{a_n}{n}.
\end{align*} 
So, 
\begin{align*}
I_n  \leq \sigma^2 \frac{1}{\frac{a_n}{2^k} - 2}  \frac{a_n}{n} 
\leq 2\sigma^2 \frac{2^k}{n},
\end{align*}
since $a_n/2^k \geq 4$.
Consequently, we obtain
\begin{align*}
\E \big[ m_{\infty,n}(\bx) - m(\bx) \big]^2 & \leq I_n + J_n  \leq  2\sigma^2 \frac{2^k}{n} + d L^2C \bigg(1- \frac{3}{4d} \bigg)^k.
\end{align*}

\end{proof}

We set up now Lemma \ref{Lemme_cell_length_beta} about the length of a cell that we used to bound the approximation error.

\begin{lemme}
\label{Lemme_cell_length_beta}
For all $\ell \in \{1, \hdots, d\}$ and $k \in \N^*$, we have
\[
\E \bigg[ V_l(\bx, \T)^2 \bigg] \leq C \bigg(1- \frac{3}{4d} \bigg)^k,
\]
with $C= \exp (12/(4d-3))$.
\end{lemme}

\begin{proof}[Proof of Lemma \ref{Lemme_cell_length_beta}]

Let us fix $\bx \in [0,1]^d$ and denote by $n_0, n_1, \hdots, n_k$ the number of points in the successive cells containing $\bx$ (for example, $n_0$ is the number of points in the root of the tree, that is $n_0 = a_n$). Note that $n_0, n_1, \hdots, n_k$ depends on $\mathcal{D}_n$ and $\Theta$, but to lighten notations, we omit these dependencies.
Recalling that $V_{\ell}(\bx, \Theta)$ is the length of the $\ell$-th side of the cell containing $\bx$, this quantity can be written as a product of independent beta distributions:  
\begin{align*}
V_{\ell}(\bx, \Theta) \overset{\mathcal{D}}{=} \prod_{j=1}^k \big[ B(n_j +1,n_{j-1} - n_j) \big]^{\delta_{\ell, j}(\bx, \Theta)},
\end{align*}
where $B(\alpha, \beta)$ denotes the beta distribution of parameters $\alpha$ and $\beta$, and the indicator $\delta_{\ell, j}(\bx, \Theta)$ equals to $1$ if the $j$-th split of the cell containing $\bx$ is performed along the $\ell$-th dimension (and $0$ otherwise). 
%
%
%
Consequently, 
\begin{align}
\mathds{E} \big[ V_{\ell}(\bx, \Theta)^2 \big] & = \prod_{j=1}^k \mathds{E} \bigg[ \big[ B(n_j +1,n_{j-1} - n_j) \big]^{2 \delta_{\ell, j}(\bx, \Theta)} \bigg] \nonumber\\
& = \prod_{j=1}^k \mathds{E} \bigg[ \mathds{E} \bigg[ \big[ B(n_j +1,n_{j-1} - n_j) \big]^{2\delta_{\ell, j}(\bx, \Theta)} \big| \delta_{\ell, j}(\bx, \Theta) \bigg] \bigg] \nonumber\\
& = \prod_{j=1}^k \mathds{E} \bigg[ \mathds{1}_{\delta_{\ell, j}(\bx, \Theta) =0} + \mathds{E}\big[ B(n_j +1,n_{j-1} - n_j) \big]^{2} \mathds{1}_{\delta_{\ell, j}(\bx, \Theta) =1} \bigg] \nonumber\\
& = \prod_{j=1}^k \bigg( \frac{d-1}{d} + \frac{1}{d} \mathds{E}\big[ B(n_j +1,n_{j-1} - n_j) \big]^{2} \bigg) \nonumber\\
& = \prod_{j=1}^k \bigg( \frac{d-1}{d} + \frac{1}{d} \frac{(n_j + 1)(n_j+2)}{(n_{j-1} + 1)(n_{j-1}+2)} \bigg) \nonumber\\
& \leq \prod_{j=1}^k \bigg( \frac{d-1}{d} + \frac{1}{4d} \frac{(n_{j-1} + 2)(n_{j-1}+4)}{(n_{j-1} + 1)(n_{j-1}+2)} \bigg) \nonumber\\
& \leq \prod_{j=1}^k \bigg( 1 - \frac{1}{d} + \frac{1}{4d} \frac{n_{j-1}+4}{n_{j-1} + 1} \bigg), \label{inequality-moment-Vl}
\end{align}
where the first inequality stems from the relation $n_j \leq n_{j-1}/2$ for all $j \in \{1, \hdots, k\}$.
We have the following inequalities.
\begin{align*}
\frac{n_{j-1}+4}{n_{j-1}+1} & \leq \frac{a_n+2^{j+1}}{a_n-2^{j-1}} = \frac{a_n+2^{j+1}}{a_n(1-\frac{2^{j-1}}{a_n})}\\
& \leq \frac{a_n+2^{j+1}}{a_n} \bigg( 1+ \frac{2^{j-1}}{a_n} \frac{1}{1- \frac{2^{j-1}}{a_n}} \bigg)\\
& \leq \bigg(1+ \frac{2^{j+1}}{a_n} \bigg)^2,
\end{align*}
since
\[
\frac{2^{j-1}}{a_n} \leq \frac{2^{k-1}}{a_n} \leq \frac{1}{2}.
\]

Going back to inequality \eqref{inequality-moment-Vl}, we find
\begin{align*}
\E \bigg[ V_l(\bx, \T)^2 \bigg] & \leq \prod_{j=1}^k \bigg[ 1 - \frac{1}{d} + \frac{1}{4d} \bigg( 1+ \frac{2^{j+1}}{a_n} \bigg)^2 \bigg]\\
& \leq \prod_{j=1}^k \bigg[ 1 - \frac{3}{4d} + \frac{3}{d} \frac{2^{j-1}}{a_n} \bigg]\\
& \leq \prod_{j=1}^k \bigg[ 1 - \frac{3}{4d} + \frac{3}{d} \frac{2^k}{a_n}2^{j-k} \bigg]\\
& \leq \prod_{j=0}^{k-1} \bigg[ 1 - \frac{3}{4d} + \frac{3}{d} 2^{-j-1} \bigg].
\end{align*}

Moreover, we can notice that
\begin{align*}
\ln \left( \prod_{j=0}^{k-1} \bigg[ 1 - \frac{3}{4d} + \frac{3}{d} 2^{-j-1} \bigg] \right) & = k \ln \bigg( 1- \frac{3}{4d} \bigg) + \sum_{j=0}^{k-1} \ln \left( 1 + 6\frac{2^{-j}}{4d-3} \right)\\
& \leq k \ln \bigg( 1- \frac{3}{4d} \bigg) + \frac{12}{4d-3}.
\end{align*}
This yields to the desired upper bound
\[
\E \bigg[ V_l(\bx, \T)^2 \bigg] \leq C \bigg(1- \frac{3}{4d} \bigg)^k,
\]
with $C= \exp (12/(4d-3))$.
\end{proof}

We now put our interest in the proofs of the two Corollaries presented in Section \ref{S-theorie}.

\begin{proof}[Proof of Corollary \ref{Corollary_1}]
Regarding Theorem \ref{th_general_bound}, we want to find the optimal value of pruning, in order to obtain the best rate of convergence for the forest estimate.

Let $C_1 =  \frac{2\sigma^2}{n}$ and $C_2 = d^{\frac{3}{2}} L^2C$ and $\beta = \Big(1- \frac{3}{4d} \Big)$. Then, 
\begin{align*}
\E \big[ m_{\infty,n}(\bx) - m(\bx) \big]^2 & \leq   C_1 2^k + C_2 \beta^{k}.
\end{align*}
Let $f : x \mapsto C_1 e^{x \ln 2} + C_2 e^{x\ln(\beta)}$. Thus, 
\begin{align*}
f'(x) & = C_1 \ln 2 e^{x \ln 2} + C_2 \ln(\beta) e^{x\ln(\beta)}\\
& = C_1 \ln 2 e^{x \ln 2} \left( 1 + \frac{C_2 \ln(\beta)}{C_1 \ln 2} e^{x(\ln(\beta) - \ln 2)} \right).
\end{align*}
Since $\beta \leq 1$, $f'(x) \leq 0$ for all $x \leq x^{\star}$ and $f'(x) \geq 0$ for all $x \geq x^{\star}$, where $x^{\star}$ satisfies
\begin{align*}
& \quad f'(x^{\star}) = 0 \\
\Longleftrightarrow & \quad  x^{\star} = \frac{1}{\ln 2 - \ln(\beta)} 
\ln \bigg( -\frac{C_2 \ln(\beta)}{C_1 \ln 2} \bigg)\\
\Longleftrightarrow & \quad x^{\star} = \frac{1}{\ln 2 - \ln(\beta)} \bigg[ 
\ln \bigg( \frac{1}{C_1} \bigg) +  \ln \bigg( -\frac{C_2 \ln(\beta)}{ \ln 2} \bigg) \bigg]\\
\Longleftrightarrow & \quad  x^{\star} = \frac{1}{\ln 2 - \ln \Big(1- \frac{3}{4d} \Big)} \left[ 
\ln ( n ) +  \ln \left(- \frac{d L^2C \ln \Big(1- \frac{3}{4d} \Big)}{2 \sigma^2 \ln 2} \right) \right]\\
\Longleftrightarrow & \quad x^{\star} = \frac{1}{\ln 2 - \ln \beta} \bigg[ 
\ln (n) +  C_3 \bigg],
\end{align*}
where $C_3 = \ln \left(- \frac{dL^2C \ln \Big(1- \frac{3}{4d} \Big)}{2 \sigma^2 \ln 2} \right)$. 

Consequently, 
\begin{align*}
\E \big[ m_{\infty,n}(\bx) - m(\bx) \big]^2 & \leq   C_1 \exp(x^{\star} \ln 2) + C_2 \exp( x^{\star} \ln \beta)\\
& \leq  C_1 \exp\bigg(\frac{1}{\ln 2 - \ln \beta} \bigg[ 
\ln (n) +  C_3 \bigg] \ln 2 \bigg) \\
& \quad + C_2 \exp\bigg( \frac{1}{\ln 2 - \ln \beta} \bigg[ 
\ln (n) +  C_3 \bigg] \ln \beta \bigg)\\
& \leq  C_1 \exp\bigg( \frac{C_3 \ln 2}{\ln 2 - \ln \beta}  \bigg) \exp\bigg(\frac{\ln 2}{\ln 2 - \ln \beta} \ln (n)  \bigg) \\
& \quad + C_2 \exp\bigg( \frac{ C_3 \ln \beta }{\ln 2 - \ln \beta} \bigg) 
\exp\bigg( \frac{ \ln \beta}{\ln 2 - \ln \beta} \ln (n) \bigg)\\
& \leq C_5 n^{\frac{ \ln 2}{\ln 2 - \ln \beta}-1} + C_6 n^{\frac{ \ln \beta}{\ln 2 - \ln \beta}}\\
& \leq \bigg(C_5 + C_6 \bigg)n^{\frac{\ln \Big(1- \frac{3}{4d} \Big)}{\ln 2 - \ln \Big(1- \frac{3}{4d} \Big)}},
\end{align*}

where $C_5 = 2 \sigma^2 \exp \bigg( \frac{C_3 \ln 2}{\ln 2 - \ln \beta} \bigg)$ and $C_6 = C_2 \exp \bigg( \frac{C_3 \ln \beta}{\ln 2 - \ln \beta} \bigg)$.
\end{proof}

\begin{proof}[Proof of Corollary \ref{Corollary_2}]
In this Corollary, we focus on the optimal value of subsampling, always for the speed of convergence for the forest estimate. Since
$k_n = \log_2(a_n) - 2 $ and $k_n$ satisfies equation (\ref{kn_inequality}), we have 
\begin{align*}
a_n = 4. 2^{\frac{C_3}{\ln 2 - \ln \beta}}. n^{\frac{\ln 2}{\ln 2 - \ln \beta}},
\end{align*}
%
where, simple calculations show that
\begin{align*}
2^{\frac{C_3}{\ln 2 - \ln \beta}}
& = \left( \frac{3  L^2 e^{12/(4d-3)} }{8 \sigma^2 \ln 2} \right)^{\frac{\ln 2}{\ln 2 - \ln \beta}}.
\end{align*}
This concludes the proof, according to Theorem \ref{th_general_bound}.
\end{proof}

\bibliography{biblio-these}

\end{document}